\documentclass[12pt]{amsart}

\usepackage{amsmath,amssymb,amsfonts,amsthm}
\usepackage{microtype}
\usepackage{ifpdf}
\usepackage{enumerate}
\usepackage{ulem}
\usepackage{leftidx}

\usepackage{breqn}

\usepackage{tikz}
\usetikzlibrary{calc}
\usepackage{tikz-qtree}
\usetikzlibrary{backgrounds}

\ifpdf
\usepackage[pdftex,plainpages=false,hypertexnames=false,pdfpagelabels]{hyperref}
\else
\usepackage[dvips,plainpages=false,hypertexnames=false]{hyperref}
\fi

\newcommand{\pathtodiagrams}{diagrams/pdf/}

\newcommand{\mathfig}[2]{{\hspace{-3pt}\begin{array}{c}%
  \raisebox{-2.5pt}{\includegraphics[width=#1\textwidth]{\pathtodiagrams #2}}%
\end{array}\hspace{-3pt}}}

\newcommand{\arxiv}[1]{\href{http://arxiv.org/abs/#1}{\tt arXiv:\nolinkurl{#1}}}
\newcommand{\doi}[1]{\href{http://dx.doi.org/#1}{{\tt DOI:#1}}}

\newcommand{\googlebooks}[1]{(preview at \href{http://books.google.com/books?id=#1}{google books})}

\theoremstyle{plain}
\newtheorem{prop}{Proposition}[section]
\newtheorem{fact}[prop]{Fact}
\newtheorem{conj}[prop]{Conjecture}
\newtheorem{thm}{Theorem}

\newtheorem{lem}[prop]{Lemma}

\newtheorem{cor}[prop]{Corollary}

\newtheorem*{thm*}{Theorem}
\newtheorem*{cor*}{Corollary}
\newtheorem{question}[prop]{Question}
\newtheorem*{question*}{Question}
\numberwithin{equation}{section}

\theoremstyle{remark}
\newtheorem{ex}[prop]{Example}

\newtheorem*{rem*}{Remark}               
\newtheorem*{ex*}{Example}                

\theoremstyle{definition}
\newtheorem{remark}[prop]{Remark}           
\newtheorem{defn}[prop]{Definition}         

\newtheorem*{defn*}{Definition}             

\theoremstyle{plain}

\newcounter{comment}
\newcommand{\noop}[1]{}

\def\clap#1{\hbox to 0pt{\hss#1\hss}}

\def\semicolon{;}
\def\applytolist#1{
    \expandafter\def\csname multi#1\endcsname##1{
        \def\multiack{##1}\ifx\multiack\semicolon
            \def\next{\relax}
        \else
            \csname #1\endcsname{##1}
            \def\next{\csname multi#1\endcsname}
        \fi
        \next}
    \csname multi#1\endcsname}

\def\calc#1{\expandafter\def\csname c#1\endcsname{{\mathcal #1}}}
\applytolist{calc}QWERTYUIOPLKJHGFDSAZXCVBNM;
\def\bbc#1{\expandafter\def\csname bb#1\endcsname{{\mathbb #1}}}
\applytolist{bbc}QWERTYUIOPLKJHGFDSAZXCVBNM;
\def\bfc#1{\expandafter\def\csname bf#1\endcsname{{\mathbf #1}}}
\applytolist{bfc}QWERTYUIOPLKJHGFDSAZXCVBNM;

\newcommand{\Span}{\operatorname{span}}

\renewcommand{\imath}{\mathfrak{i}}
\renewcommand{\jmath}{\mathfrak{j}}

\newcommand{\iso}{\cong}

\newcommand{\overcrossing}{\begin{tikzpicture}[baseline]
\node (x) at (0,0){};
    \draw (x.45)-- (.5,.5);
    \draw (x.135) -- (-.5,.5);
    \draw (x.315) -- (.5,-.5);
    \draw (x.45) -- (-.5,-.5);
\end{tikzpicture}}
\newcommand{\symmetriccrossing}{\begin{tikzpicture}[baseline]
\node (x) at (0,0){};
    \draw (x.45)-- (.5,.5);
    \draw (x.135) -- (-.5,.5);
    \draw (x.135) -- (.5,-.5);
    \draw (x.45) -- (-.5,-.5);
\end{tikzpicture}}

\newcommand{\identity}{\begin{tikzpicture}[baseline]
    \draw (1.5,.5) .. controls (2,0) .. (1.5,-.5);
    \draw (2.5,.5) .. controls (2,0) .. (2.5,-.5);
\end{tikzpicture}}
\newcommand{\cupcap}{\begin{tikzpicture}[baseline]
    \draw (3.5,.5) .. controls (4,0) .. (4.5,.5);
    \draw (3.5,-.5) .. controls (4,0) .. (4.5,-.5);
\end{tikzpicture}}

\newcommand{\tensor}{\otimes}

\newcommand{\Inv}{\operatorname{Inv}}
\newcommand{\Hom}{\operatorname{Hom}}

\newcommand{\tr}[1]{\text{tr}(#1)}

\makeatletter

\def\@testdef #1#2#3{%
  \def\reserved@a{#3}\expandafter \ifx \csname #1@#2\endcsname
 \reserved@a  \else
\typeout{^^Jlabel #2 changed:^^J%
\meaning\reserved@a^^J%
\expandafter\meaning\csname #1@#2\endcsname^^J}%
\@tempswatrue \fi}

\title{Categories generated by a trivalent vertex}
\author{Kim Morrison}
\author{Emily Peters}
\author{Noah Snyder}

\usetikzlibrary{patterns,decorations.markings}

\newcommand{\Cube}{\prism[45]{4}}
\newcommand{\PentPrism}{\prism{5}}

\newcommand{\prism}[2][0]{
\begin{tikzpicture}[scale=0.8,baseline=-0.1cm]
\fill[white] (0,0) circle (0.5cm);
\foreach \x in {1, ..., #2} {
	\draw (360*\x/#2-360/#2+#1:.3cm) -- (360*\x/#2+#1:.3cm)--(360*\x/#2+#1:.5cm) -- (360*\x/#2+360/#2+#1:.5cm);
}
\end{tikzpicture}
}

\newcommand{\ngon}[2][0]{
\begin{tikzpicture}[baseline=0cm]
\foreach \x in {1, ..., #2}
	\draw (360*\x/#2+#1:.8cm)--(360*\x/#2+#1:.5cm);
\foreach \x in {1, ..., #2}
	\draw (360*\x/#2+#1:.5cm) .. controls +(360*\x/#2+120+#1:.3cm) and +(360*\x/#2+360/#2-120+#1:.3cm) .. (360*\x/#2+360/#2+#1:.5cm);
\end{tikzpicture}
}
\newcommand{\ngonb}[2][0]{
\begin{tikzpicture}[baseline=0cm]
\foreach \x in {1, ..., #2}
	\draw (360*\x/#2+#1:.8cm)--(360*\x/#2+#1:.5cm);
\foreach \x in {1, ..., #2}
	\draw (360*\x/#2+#1:.5cm) .. controls +(360*\x/#2+120+#1:.3cm) and +(360*\x/#2+360/#2-120+#1:.3cm) .. (360*\x/#2+360/#2+#1:.5cm);
\foreach \x in {1, ..., #2}
	\node at (360*\x/#2+#1:.3cm) {$\bullet$};
\end{tikzpicture}
}

\newcommand{\drawI}{ \begin{tikzpicture}[baseline=0cm]
 	\draw (0,.2) .. controls +(30:.3cm) .. (45:.8cm);
 	\draw (0,.2) .. controls +(150:.3cm) .. (135:.8cm);
	\draw (0,.2) -- (0,-.2);
 	\draw (0,-.2) .. controls +(-30:.3cm) .. (-45:.8cm);
 	\draw (0,-.2) .. controls +(-150:.3cm) .. (-135:.8cm);
\end{tikzpicture}
}

\newcommand{\drawH}{ \begin{tikzpicture}[baseline=0cm,rotate=90]
 	\draw (0,.2) .. controls +(30:.3cm) .. (45:.8cm);
 	\draw (0,.2) .. controls +(150:.3cm) .. (135:.8cm);
	\draw (0,.2) -- (0,-.2);
 	\draw (0,-.2) .. controls +(-30:.3cm) .. (-45:.8cm);
 	\draw (0,-.2) .. controls +(-150:.3cm) .. (-135:.8cm);
\end{tikzpicture}}

\renewcommand{\cupcap}{\begin{tikzpicture}[baseline=0cm]
	\draw (45:.8cm) .. controls (0,0) .. (135:.8cm);
	\draw (-45:.8cm) .. controls (0,0) .. (-135:.8cm);
\end{tikzpicture}}

\newcommand{\twostrandid}{\begin{tikzpicture}[baseline=0cm,rotate=90]
	\draw (45:.8cm) .. controls (0,0) .. (135:.8cm);
	\draw (-45:.8cm) .. controls (0,0) .. (-135:.8cm);
\end{tikzpicture}}

\newcommand{\drawIb}{ \begin{tikzpicture}[baseline=0cm]
 	\draw (0,.2) .. controls +(30:.3cm) .. (45:.8cm);
 	\draw (0,.2) .. controls +(150:.3cm) .. (135:.8cm);
	\draw (0,.2) -- (0,-.2);
 	\draw (0,-.2) .. controls +(-30:.3cm) .. (-45:.8cm);
 	\draw (0,-.2) .. controls +(-150:.3cm) .. (-135:.8cm);
 	\node at (0,.4) {$\bullet$};
	\node at (0,-.4) {$\bullet$};
\end{tikzpicture}
}

\newcommand{\drawHb}{ \begin{tikzpicture}[baseline=0cm,rotate=90]
 	\draw (0,.2) .. controls +(30:.3cm) .. (45:.8cm);
 	\draw (0,.2) .. controls +(150:.3cm) .. (135:.8cm);
	\draw (0,.2) -- (0,-.2);
 	\draw (0,-.2) .. controls +(-30:.3cm) .. (-45:.8cm);
 	\draw (0,-.2) .. controls +(-150:.3cm) .. (-135:.8cm);
 	\node at (0,.4) {$\bullet$};
	\node at (0,-.4) {$\bullet$};
\end{tikzpicture}}

\newcommand{\pentafork}[1]{
\begin{tikzpicture}[baseline=0, rotate=#1]
	\coordinate (F) at (30:.6cm);
	\coordinate (P1) at (30:.4cm);
	\coordinate (P2) at (120:.6cm);
	\coordinate (P3) at (180:.7cm);
	\coordinate (P4) at (240:.7cm);
	\coordinate (P5) at (300:.6cm);
	\draw (60:1cm)--(F)--(0:1cm);
	\draw (F)--(P1)--(P2)--(P3)--(P4)--(P5)--(P1);
	\draw (P2)--(120:1cm);
	\draw (P3)--(180:1cm);
	\draw (P4)--(240:1cm);
	\draw (P5)--(300:1cm);
	\draw[dotted, thin, gray] (0,0) circle (1cm);
\end{tikzpicture}
}

\newcommand{\pentaforkb}{
\begin{tikzpicture}[baseline=0]
	\coordinate (F) at (90:.5cm);
	\coordinate (P1) at (0:.5cm);
	\coordinate (P2) at (90:.35cm);
	\coordinate (P3) at (180:.5cm);
	\coordinate (P4) at (240:.5cm);
	\coordinate (P5) at (300:.5cm);
	\coordinate (E1) at (0:.8cm);
	\coordinate (E2) at (60:.8cm);
	\coordinate (E3) at (120:.8cm);
	\coordinate (E4) at (180:.8cm);
	\coordinate (E5) at (240:.8cm);
	\coordinate (E6) at (300:.8cm);
	\draw (E2)--(F)--(E3);
	\draw (F)--(P2)--(P3)--(P4)--(P5)--(P1)--(P2);
	\draw (P1)--(E1);
	\draw (P3)--(E4);
	\draw (P4)--(E5);
	\draw (P5)--(E6);
	\path (F) ++(90:.15) node {$\bullet$};	
	\path (P1) ++(190:.2) node {$\bullet$};
	\path (P2) ++(270:.175) node {$\bullet$};
	\path (P3) ++(-10:.2) node {$\bullet$};
	\path (P4) ++(60:.15) node {$\bullet$};
	\path (P5) ++(120:.15) node {$\bullet$};
\end{tikzpicture}
}

\newcommand{\doubletrib}{
\begin{tikzpicture}[baseline=0]
	\coordinate (E1) at (0:.8cm);
	\coordinate (E2) at (60:.8cm);
	\coordinate (E3) at (120:.8cm);
	\coordinate (E4) at (180:.8cm);
	\coordinate (E5) at (240:.8cm);
	\coordinate (E6) at (300:.8cm);
	\coordinate (P1) at (60:.3cm);
	\coordinate (P2) at (240:.3cm);
	\draw (E1)--(P1)--(E3);
	\draw (P1)--(E2);
	\draw (E4)--(P2)--(E6);
	\draw (P2)--(E5);	
	\path (P1) ++(240:.15) node {$\bullet$};
	\path (P2) ++(60:.15) node {$\bullet$};
\end{tikzpicture}
}

\newcommand{\Icupb}{
\begin{tikzpicture}[baseline=0]
	\coordinate (E1) at (0:.8cm);
	\coordinate (E2) at (60:.8cm);
	\coordinate (E3) at (120:.8cm);
	\coordinate (E4) at (180:.8cm);
	\coordinate (E5) at (240:.8cm);
	\coordinate (E6) at (300:.8cm);
	\coordinate (P1) at (90:.5cm);
	\coordinate (P2) at (90:.1cm);
	\draw (E1)--(P2)--(E4);
	\draw (P1)--(P2);
	\draw (E2)--(P1)--(E3);
	\draw (E5) .. controls (270:.2cm) .. (E6);
	\path (P1) ++(90:.15) node {$\bullet$};
	\path (P2) ++(270:.15) node {$\bullet$};
\end{tikzpicture}
}

\newcommand{\Hcupb}{
\begin{tikzpicture}[baseline=0]
	\coordinate (E1) at (0:.8cm);
	\coordinate (E2) at (60:.8cm);
	\coordinate (E3) at (120:.8cm);
	\coordinate (E4) at (180:.8cm);
	\coordinate (E5) at (240:.8cm);
	\coordinate (E6) at (300:.8cm);
	\coordinate (P1) at (60:.2cm);
	\coordinate (P2) at (120:.2cm);
	\draw (E1)--(P1)--(E2);
	\draw (P1)--(P2);
	\draw (E3)--(P2)--(E4);
	\draw (E5) .. controls (270:.2cm) .. (E6);
	\path (P1) ++(20:.2) node {$\bullet$};
	\path (P2) ++(160:.2) node {$\bullet$};
\end{tikzpicture}
}

\newcommand{\threecupsb}{
\begin{tikzpicture}[baseline=0]
	\coordinate (E1) at (0:.8cm);
	\coordinate (E2) at (60:.8cm);
	\coordinate (E3) at (120:.8cm);
	\coordinate (E4) at (180:.8cm);
	\coordinate (E5) at (240:.8cm);
	\coordinate (E6) at (300:.8cm);
	\draw (E2) .. controls (90:.15) .. (E3);
	\draw (E4) .. controls (210:.15) .. (E5);
	\draw (E6) .. controls (320:.15) .. (E1);
\end{tikzpicture}
}

\newcommand{\firstpinwheelb}{
\begin{tikzpicture}[baseline=0]
	\coordinate (E1) at (0:.8cm);
	\coordinate (E2) at (60:.8cm);
	\coordinate (E3) at (120:.8cm);
	\coordinate (E4) at (180:.8cm);
	\coordinate (E5) at (240:.8cm);
	\coordinate (E6) at (300:.8cm);
	\coordinate (P1) at (45:.5cm);
	\coordinate (P2) at (90:.3cm);
	\coordinate (P3) at (270:.3cm);
	\coordinate (P4) at (225:.5cm);
	\draw (E1)--(P1)--(E2);
	\draw (E3)--(P2)--(P3)--(E6);
	\draw (E4)--(P4)--(E5);
	\draw (P1)--(P2)--(P3)--(P4);
	\path (P1) ++(20:.2cm) node {$\bullet$};
	\path (P4) ++(200:.2cm) node {$\bullet$};
	\path (P2) ++(200:.2cm) node {$\bullet$};
	\path (P3) ++(20:.2cm) node {$\bullet$};
\end{tikzpicture}
}

\newcommand{\secondpinwheelb}{
\begin{tikzpicture}[baseline=0, xscale=-1]
	\coordinate (E1) at (0:.8cm);
	\coordinate (E2) at (60:.8cm);
	\coordinate (E3) at (120:.8cm);
	\coordinate (E4) at (180:.8cm);
	\coordinate (E5) at (240:.8cm);
	\coordinate (E6) at (300:.8cm);
	\coordinate (P1) at (45:.5cm);
	\coordinate (P2) at (90:.3cm);
	\coordinate (P3) at (270:.3cm);
	\coordinate (P4) at (225:.5cm);
	\draw (E1)--(P1)--(E2);
	\draw (E3)--(P2)--(P3)--(E6);
	\draw (E4)--(P4)--(E5);
	\draw (P1)--(P2)--(P3)--(P4);
	\path (P1) ++(20:.2cm) node {$\bullet$};
	\path (P4) ++(200:.2cm) node {$\bullet$};
	\path (P2) ++(200:.2cm) node {$\bullet$};
	\path (P3) ++(20:.2cm) node {$\bullet$};
\end{tikzpicture}
}

\newcommand{\sixforestb}{
\begin{tikzpicture}[baseline=0, xscale=-1]
	\coordinate (E1) at (0:.8cm);
	\coordinate (E2) at (60:.8cm);
	\coordinate (E3) at (120:.8cm);
	\coordinate (E4) at (180:.8cm);
	\coordinate (E5) at (240:.8cm);
	\coordinate (E6) at (300:.8cm);
	\coordinate (P1) at (30:.5cm);
	\coordinate (P2) at (150:.5cm);
	\coordinate (P3) at (240:.4cm);
	\coordinate (P4) at (300:.4cm);
	\draw (E1)--(P1)--(E2);
	\draw (E3)--(P2)--(E4);
	\draw (E5)--(P3);
	\draw (E6)--(P4);
	\draw (P2)--(P3)--(P4)--(P1);
	\path (P1) ++(180:.15cm) node {$\bullet$};
	\path (P2) ++(0:.15cm) node {$\bullet$};
	\path (P3) ++(60:.15cm) node {$\bullet$};
	\path (P4) ++(120:.15cm) node {$\bullet$};
\end{tikzpicture}
}

\newcommand{\Rtwo}{
\begin{tikzpicture}[baseline=0]
	\coordinate (E1) at (45:.8cm);
	\coordinate (E2) at (135:.8cm);
	\coordinate (E3) at (225:.8cm);
	\coordinate (E4) at (315:.8cm);
	\draw[white, line width = 4pt] (E1) .. controls (180:1cm) .. (E4);	
	\draw (E1) .. controls (180:1cm) .. (E4);
	\draw[white, line width = 4pt] (E2) .. controls (0:1cm) .. (E3);	
	\draw (E2) .. controls (0:1cm) .. (E3);		
\end{tikzpicture}
}

\newcommand{\trivalentpullthroughone}{
\begin{tikzpicture}[baseline=0]
	\coordinate (E1) at (90-72:.8cm);
	\coordinate (E2) at (90:.8cm);
	\coordinate (E3) at (90+72:.8cm);
	\coordinate (E4) at (90+2*72:.8cm);
	\coordinate (E5) at (90+3*72:.8cm);
	\coordinate (P1) at (270:.2cm);
	\draw (E4) -- (P1) -- (E5);
	\draw (P1) -- (E2);
	\draw[white, line width = 4pt] (E1) -- (E3);
	\draw (E1) -- (E3);
	\path (P1) ++(270:.2cm) node {$\bullet$};
\end{tikzpicture}}

\newcommand{\trivalentpullthroughtwo}{
\begin{tikzpicture}[baseline=0]
	\coordinate (E1) at (90-72:.8cm);
	\coordinate (E2) at (90:.8cm);
	\coordinate (E3) at (90+72:.8cm);
	\coordinate (E4) at (90+2*72:.8cm);
	\coordinate (E5) at (90+3*72:.8cm);
	\coordinate (P1) at (90:.2cm);
	\draw (E4) .. controls (180:.2cm) .. (P1) .. controls (0:.2cm) ..  (E5);
	\draw (P1) -- (E2);
	\draw[white, line width = 4pt] (E1) .. controls (270:.4cm) .. (E3);
	\draw (E1) .. controls (270:.4cm) .. (E3);
	\path (P1) ++(270:.2cm) node {$\bullet$};	
\end{tikzpicture}}

\newcommand{\trivalentpullthroughthree}{
\begin{tikzpicture}[baseline=0]
	\coordinate (E1) at (90-72:.8cm);
	\coordinate (E2) at (90:.8cm);
	\coordinate (E3) at (90+72:.8cm);
	\coordinate (E4) at (90+2*72:.8cm);
	\coordinate (E5) at (90+3*72:.8cm);
	\coordinate (P1) at (270:.2cm);
	\draw (E4) -- (P1) -- (E5);
	\draw (E1) -- (E3);
	\draw[white, line width = 4pt] (P1) -- (E2);
	\draw (P1) -- (E2);
	\path (P1) ++(270:.2cm) node {$\bullet$};
\end{tikzpicture}
}

\newcommand{\trivalentpullthroughfour}{
\begin{tikzpicture}[baseline=0]
	\coordinate (E1) at (90-72:.8cm);
	\coordinate (E2) at (90:.8cm);
	\coordinate (E3) at (90+72:.8cm);
	\coordinate (E4) at (90+2*72:.8cm);
	\coordinate (E5) at (90+3*72:.8cm);
	\coordinate (P1) at (90:.2cm);
	\draw (E1) .. controls (270:.4cm) .. (E3);
	\draw[white, line width = 4pt] (E4) .. controls (180:.2cm) .. (P1) .. controls (0:.2cm) ..  (E5);
	\draw (E4) .. controls (180:.2cm) .. (P1) .. controls (0:.2cm) ..  (E5);
	\draw (P1) -- (E2);
	\path (P1) ++(270:.2cm) node {$\bullet$};	
\end{tikzpicture}
}

\newcommand{\strandtwist}{
\begin{tikzpicture}[baseline=0]
	\coordinate (E1) at (315:.8cm);
	\coordinate (E2) at (225:.8cm);
	\coordinate (P1) at (45:.2cm);
	\coordinate (P2) at (135:.2cm);
	\coordinate (C1) at (0:.6cm);
	\coordinate (C2) at (180:.6cm);
	\coordinate (C3) at (90:.16cm);
	\draw (E1) .. controls (C2) .. (P2);
	\draw[white, line width = 4pt] (P2) .. controls (C3) .. (P1) .. controls (C1) .. (E2);
	\draw (P2) .. controls (C3) .. (P1) .. controls (C1) .. (E2);
\end{tikzpicture}
}

\newcommand{\singlecap}{
\begin{tikzpicture}[baseline=0]
	\coordinate (E1) at (225:.8cm);
	\coordinate (E2) at (315:.8cm);
	\draw (E1) .. controls (90:.2cm) .. (E2);
\end{tikzpicture}
}

\newcommand{\trivalenttwistb}{
\begin{tikzpicture}[baseline=0]
	\coordinate (E1) at (330:.8cm);
	\coordinate (E2) at (210:.8cm);
	\coordinate (E3) at (90:.8cm);
	\coordinate (P1) at (90:.2cm);
	\coordinate (C1) at (0:.6cm);
	\coordinate (C2) at (180:.6cm);
	\draw (E1) .. controls (C2) .. (P1);
	\draw[white, line width=4pt] (P1) .. controls (C1) .. (E2);	
	\draw (P1) .. controls (C1) .. (E2);	
	\draw (P1) -- (E3);
	\path (P1) ++(270:.2cm) node {$\bullet$};
\end{tikzpicture}
}

\newcommand{\trivalentb}{
\begin{tikzpicture}[baseline=0]
	\coordinate (E1) at (330:.8cm);
	\coordinate (E2) at (210:.8cm);
	\coordinate (E3) at (90:.8cm);
	\coordinate (P1) at (0,0);
	\draw (E1) -- (P1) -- (E2);
	\draw (P1) -- (E3);
	\path (P1) ++(270:.2cm) node {$\bullet$};
\end{tikzpicture}
}

\newcommand{\trivalentotherdotb}{
\begin{tikzpicture}[baseline=0]
	\coordinate (E1) at (330:.8cm);
	\coordinate (E2) at (210:.8cm);
	\coordinate (E3) at (90:.8cm);
	\coordinate (P1) at (0,0);
	\draw (E1) -- (P1) -- (E2);
	\draw (P1) -- (E3);
	\path (P1) ++(150:.2cm) node {$\bullet$};
\end{tikzpicture}
}

\newcommand{\overviolinb}{
\begin{tikzpicture}[baseline=0]
	\coordinate (E1) at (45:.8cm);
	\coordinate (E2) at (225:.8cm);
	\coordinate (E3) at (270:.8cm);
	\coordinate (P1) at (0,0);
	\coordinate (P2) at (120:.6cm);
	\coordinate (P3) at (.4cm,.2cm);
	\coordinate (C1) at (100:.9cm);
	\coordinate (C2) at (135:.5cm);
	\coordinate (C3) at (.8,.1);
	\draw (E1) -- (P1) -- (E3);
	\draw (P1) .. controls  (C1) .. (P2);
	\draw[white, line width=4pt] (P2) .. controls  (C2) .. (P3);	
	\draw[white, line width=4pt] (P3) .. controls (C3) .. (E2);
	\draw (P2) .. controls  (C2) .. (P3);	
	\draw (P3) .. controls (C3) .. (E2);
	\path (P1) ++(330:.15cm) node {$\bullet$};
\end{tikzpicture}
}

\newcommand{\rotatedtrivalentb}{
\begin{tikzpicture}[baseline=0]
	\coordinate (E1) at (45:.8cm);
	\coordinate (E2) at (225:.8cm);
	\coordinate (E3) at (270:.8cm);
	\coordinate (P0) at (0,0);
	\coordinate (C1) at (-.1,.7);
	\coordinate (C2) at (-.3,-.2);
	\coordinate (C3) at (.2,-.2);
	\draw (P0) .. controls (C1) .. (E2);
	\draw (P0) .. controls (C2) .. (E3);
	\draw (P0) .. controls (C3) .. (E1);
	\path (P0) ++(270:.2cm) node {$\bullet$};
\end{tikzpicture}
}

\newcommand{\underviolinb}{
\begin{tikzpicture}[baseline=0]
	\coordinate (E1) at (45:.8cm);
	\coordinate (E2) at (225:.8cm);
	\coordinate (E3) at (270:.8cm);
	\coordinate (P1) at (0,0);
	\coordinate (P2) at (120:.6cm);
	\coordinate (P3) at (.4cm,.2cm);
	\coordinate (C1) at (100:.9cm);
	\coordinate (C2) at (135:.5cm);
	\coordinate (C3) at (.8,.1);
	\draw (P2) .. controls  (C2) .. (P3);	
	\draw (P3) .. controls (C3) .. (E2);
	\draw[white, line width=4pt] (E1) -- (P1) -- (E3);
	\draw[white, line width=4pt] (P1) .. controls  (C1) .. (P2);
	\draw (E1) -- (P1) -- (E3);
	\draw (P1) .. controls  (C1) .. (P2);
	
	\path (P1) ++(330:.15cm) node {$\bullet$};
\end{tikzpicture}
}

\newcommand{\trivalentcap}{
\begin{tikzpicture}[scale=.25]
	\coordinate (E1) at (90-72:.8cm);
	\coordinate (E2) at (90:.8cm);
	\coordinate (E3) at (90+72:.8cm);
	\coordinate (E4) at (90+2*72:.8cm);
	\coordinate (E5) at (90+3*72:.8cm);
	\coordinate (P1) at (0,0);
	\draw (E1)--(P1)--(E3);
	\draw (P1)--(E2);
	\draw (E4) .. controls (270:.2cm) .. (E5);
\end{tikzpicture}
}

\newcommand{\smalltwostrandid}{\begin{tikzpicture}[rotate=90,scale=.25]
	\draw (45:.8cm) .. controls (0,0) .. (135:.8cm);
	\draw (-45:.8cm) .. controls (0,0) .. (-135:.8cm);
\end{tikzpicture}}

\newcommand{\smalldrawI}{ \begin{tikzpicture}[scale=.25]
 	\draw (0,.2) .. controls +(30:.3cm) .. (45:.8cm);
 	\draw (0,.2) .. controls +(150:.3cm) .. (135:.8cm);
	\draw (0,.2) -- (0,-.2);
 	\draw (0,-.2) .. controls +(-30:.3cm) .. (-45:.8cm);
 	\draw (0,-.2) .. controls +(-150:.3cm) .. (-135:.8cm);
\end{tikzpicture}
}

\begin{document}

\begin{abstract}
This is the first paper in a general program to automate skein theoretic arguments.  In this paper, we study skein
theoretic invariants of planar trivalent graphs.  Equivalently, we classify trivalent categories, which are
nondegenerate pivotal tensor categories over $\bbC$ generated by a symmetric self-dual simple object $X$ and a
rotationally invariant morphism $1 \rightarrow X \tensor X \tensor X$.  Our main result is that the only trivalent
categories with $\dim \Hom(1 \to X^{\otimes n})$ bounded by $1,0,1,1,4,11,40$ for $0 \leq n \leq 6$ are quantum $SO(3)$,
quantum $G_2$, a one-parameter family of free products of certain Temperley-Lieb categories (which we call ABA
categories), and the
$H3$ Haagerup fusion category.  We also prove similar results where the map $1 \rightarrow X^{\otimes 3}$ is not rotationally invariant, and we give a complete classification of nondegenerate braided trivalent categories with dimensions of invariant spaces bounded by the sequence $1,0,1,1,4$.  Our main techniques are a new approach to finding skein relations which can be easily automated using Gr\"obner bases, and evaluation algorithms which use the discharging method developed in the proof of the $4$-color theorem.
\end{abstract}

\maketitle


\section{Introduction}

This is the first paper in a general program to automate skein theoretic arguments in quantum algebra and quantum
topology.  In this paper, we study skein theoretic invariants of planar trivalent graphs following Kuperberg \cite
{MR1265145}.  However, the general approach will work much more broadly and in later papers we will consider other
situations like those in \cite{MR2132671,MR1733737,MR1972635,MR2783128,1410.2876}.  One might think of this
program as attempting for skein theoretic arguments what \cite{MR2914056} did for principal graph arguments.  

Before getting into the particulars of this paper, we will recall the basic notions of skein theory, which we illustrate
using its most famous example.  The Jones polynomial invariant \cite{MR0766964} of framed links 
can be computed by applying the Kauffman bracket skein relations
\begin{align*}
\scalebox{1.4}{$\bigcirc$} & = -A^2 -A^{-2} \notag \\
\intertext{and}
\overcrossing & = A \identity + A^{-1} \cupcap.
\end{align*}

These relations not only assign a Laurent polynomial to framed links, they also
assign to each tangle a linear combination of noncrossing tangles.  Unlike for
ordinary tangles, the number of noncrossing tangles with $n$ boundary points is
finite (and indeed given by Catalan numbers).  Our goal is to find and prove
theorems modelled on the following well-known result.

\begin{thm*}  \label{thm:bracket}
Suppose a skein theoretic invariant of framed links has the property that the 
span of $4$-boundary point tangles modulo the relations is $2$-dimensional. This
invariant must be a specialization of the Jones polynomial.
\end{thm*}

The proof of this statement is straightforward.
Since the span of $4$-boundary point tangles is only $2$-dimensional, there must 
be some linear relation between the three diagrams which occur in the bracket 
relation.  Some quick calculations show that only the specific Kauffman bracket 
relation is compatible with the second and third Reidemeister moves.  Finally it is
 clear that the bracket relations are enough to determine the invariant, since 
 applying the crossing relation will turn any framed link into a linear 
 combination of unlinked, unknotted, circles.

A number of similar results have been proved following this same outline. 
In each case, assumptions on the dimensions of spaces of diagrams guarantee 
relations of a certain form must hold, and an involved calculation determines 
the coefficients of these relations (possibly in terms of some parameters). 
Subsequently one finds an evaluation algorithm using these relations, demonstrating 
that they suffice to determine the invariant.
Finally, one may also want to know that this invariant actually exists!

Our program has a two-fold goal. First, we are interested in generalizing this 
approach beyond invariants of links. We take the view that links are `merely' a 
certain class of planar graphs (with vertices modeled on the under-crossing and 
the over-crossing), subject to some local relations. We would like to be able to 
prove theorems about arbitrary such classes. Second, to the extent possible we 
want to automate the technique --- to find very general methods to derive such 
theorems, and where possible to implement these methods in code. This will 
enable us to rapidly explore many different skein theoretic settings, and moreover 
to explore much further out into the space of examples than is possible with 
by-hand calculations. (This paper restricts itself to the case of planar 
trivalent graphs --- but we go much further than previous investigations of 
trivalent skein theories.)

In particular, the first step in the outline above --- finding relations --- is 
highly amenable to computer calculation, as explained in this paper (see 
\cite{F4E6} for another approach to automating skein theory).  By contrast, the 
second step --- giving an evaluation algorithm --- remains an art.  In this 
paper, our evaluation algorithms all come from the discharging technique 
developed in the proof of the $4$-color theorem. Indeed, we hope that in the 
future it may be possible to systematically discover evaluation algorithms based 
on discharging, adapted to the available relations in a given skein theoretic 
setting.

In this paper we concentrate on skein theoretic invariants of planar trivalent 
graphs.  By the diagrammatic calculus for tensor categories, 
this paper can also be thought of as providing classifications of nondegenerate 
pivotal tensor categories over $\bbC$ generated by a symmetrically self-dual simple object $X$ 
and a rotationally invariant morphism $1 \rightarrow X \tensor X \tensor X$, 
where the dimensions of the first few invariant spaces 
$\Hom(1 \to X^{\otimes n})$ are small. 
We call such a tensor category a 
\emph{trivalent category}.  We note that the arguments in the paper are elementary and 
skein theoretic, so no knowledge of tensor categories is needed except for the existence proofs.

We now summarize this paper's results in a table. To use this table, compute 
the dimensions of the spaces of diagrams with $n$ boundary points (equivalently the 
invariant spaces $\Hom(1 \to X^{\otimes n})$ for your tensor category) for the first few values of $n$. 
If an initial segment of this sequence of dimensions appears in the first column 
of some row of this table, then the second column explicitly identifies the 
category for you. If, on the other hand, you only have upper 
bounds for the dimensions of invariant spaces, then your category appears in the 
corresponding row, or some previous row.  

\vspace{5mm}

\begin{center}
\begin{tabular}{l|l|c}
dimension bounds         & new examples 				 	& reference  $\quad$ \\
\hline
1,0,1,1,2,\ldots           & $SO(3)_{\zeta_5}$       			& Theorem \ref{thm:4} \\
1,0,1,1,3,\ldots           & $SO(3)_q$  or $OSp(1|2)$   	& ---$\|$--- \\
1,0,1,1,4,8,\ldots         & $ABA \subset TL_{\sqrt{d t^{-1}}} \ast TL_t$ & Theorem \ref{thm:5}\\
1,0,1,1,4,9,\ldots         & $(G_2)_{\zeta_{20}}$ 			& ---$\|$--- \\
1,0,1,1,4,10,\ldots        & $(G_2)_{q}$ 						& ---$\|$--- \\
1,0,1,1,4,11,37,\ldots     & $H3$ 							& Theorem \ref{thm:6}\\
1,0,1,1,4,11,40,\ldots     & nothing more 					& ---$\|$--- 
\end{tabular}
\end{center}

\vspace{5mm}

In this table, $\ast$ denotes the 
free product, and $H3$ denotes the fusion category Morita equivalent to the 
Haagerup fusion category constructed in \cite{MR2909758}.  It is fascinating to see 
the Haagerup subfactor once again appearing as the first surprising example in 
a classification of `small' categories.

The same classification is shown in Figure \ref{fig:treeoflife}.

Kuperberg proved in \cite{MR1265145} that if the dimensions are exactly $1,0,1,1,4,10$ and in
addition the diagrams with no internal faces give bases for these spaces, then
the category must be $(G_2)_{q}$.   In order to apply Kuperberg's result one needs to do a calculation
to verify linear independence (cf. \cite[Lemma 3.9]{MR2783128}).  Our classification is more satisfying, as it does not
include this linear independence assumption. To get an even more satisfying result, one would need to drop the condition
that the
trivalent vertex generates all morphisms.  Dropping the generating assumption would introduce some
additional examples (e.g. from subfactors, or from quantum subgroups of $(G_2)_q$).

It is worth noting that for these results up to and including the row $1,0,1,1,4,10$, 
we can also give proofs that do not use a computer, following Kuperberg.  However, 
these by-hand calculations are not enlightening, and we prefer to give 
computer-assisted arguments uniformly in all cases, because they are easier to 
follow and more reliable.\footnote{Indeed, N.S. initially did such calculations 
by hand. 
Due to human error this initial version missed the ABA case, but the error was 
easily caught by the more reliable computer.}  By contrast, the $1,0,1,1,4,11$ 
results would be quite difficult, and probably impossible, to check by hand.



We also prove classification results when the map $1 \rightarrow X
\tensor X \tensor X$ has a nontrivial rotational eigenvalue.  This case turns
out to be much easier than the rotationally invariant case (likely easy enough to check by hand in a tedious week), and there are
correspondingly many fewer possibilities.  If the dimensions are below
$1,0,1,1,4,11,40,\ldots$ then the category must be a twisted version of
$\mathrm{Rep}(S_3)$ or a twisted version of the Haagerup fusion category, or
possibly one other new tensor category.  
This new candidate category is interesting as it can not come from subfactors or quantum groups.  Finding such exotic tensor categories is one of our main motivations for this project.

As a corollary in the spirit of the results in \cite{MR2783128}, we see that if 
$X$ is a simple object in a pivotal category and 
$X^{\otimes 2} \iso 1 \oplus X \oplus A \oplus B$ for some simple objects $A$ 
and $B$, and moreover $\dim \Hom(X \to (A \oplus B)^{\otimes 2}) \leq 3$, then 
$\dim \Hom(1 \to X^{\otimes 5}) \leq 10$ and so the category generated by the 
morphism $X^{\otimes 2} \to X$ must be either a twisted $\mathrm{Rep}(S_3)$ 
category, or an $SO(3)_q$, $ABA$, or $(G_2)_q$ category.

These results also allow a complete classification of braided trivalent
categories with $\dim \Hom(1 \to X^{\otimes 4}) \leq 4$. given in Section \ref{sec:braided}.   A quick argument shows
that the braiding guarantees that $\dim \Hom(1 \to X^
{\otimes 5}) \leq 10$.  By the table above, any braided trivalent category  
with $\dim \Hom(1 \to X^{\otimes 4}) \leq 4$
must be $OSp(1|2)$, $SO(3)_q$ or $(G_2)_q$ (the ABA categories are not braided).  We also classify all braidings on these categories.
Note 
that these results on braided trivalent categories only use the $1,0,1,1,4,10$ 
classifications, and so can be checked by hand in a reasonable amount of time.

\pdfpageattr {/Group << /S /Transparency /I true /CS /DeviceRGB>>} 

\begin{figure}[!ht]
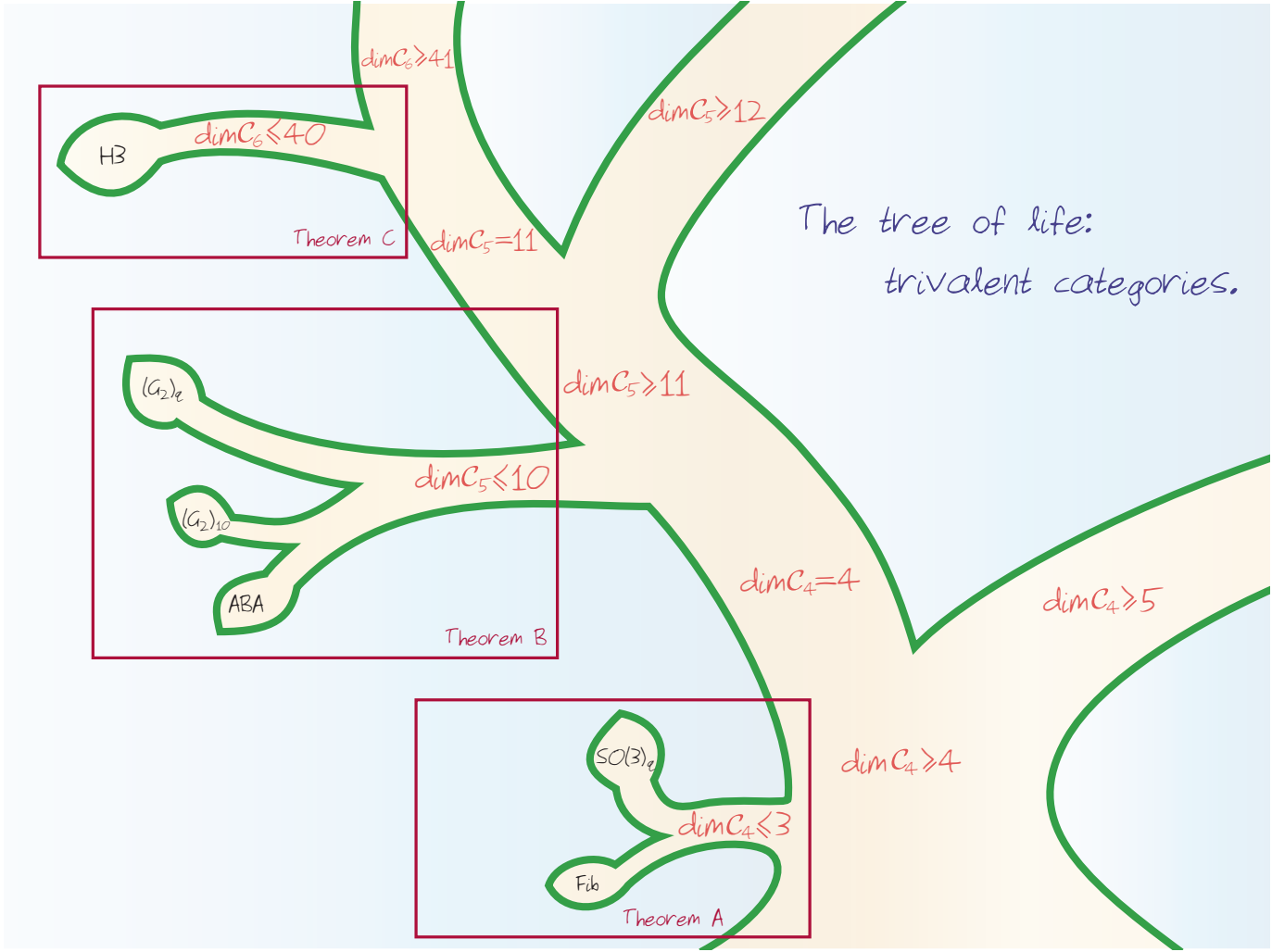

\makebox[\textwidth][c]{
$
\mathfig{1.3}{tree-of-life}
$
}
\caption{The `tree of life' of trivalent categories, as described in this
paper.
The rightmost branch, corresponding to trivalent categories with $\dim \cC\sb{4} \geq 5$, certainly has representatives:
the
(quantum) representation categories of the complex simple Lie algebras (excepting some small cases), with the trivalent vertex the Lie bracket on the adjoint representation. The other two branches,
corresponding to $\dim \cC\sb{4} = 4$ and $\dim \cC\sb{5} \geq 12$ or $\dim \cC\sb{4} = 4, \dim \cC\sb{5} = 11$, and
$\dim \cC\sb{6} \geq
41$, might well be extinct. It is tempting to believe that once some of the first few $\dim \cC\sb{n}$'s are small, it is
hard for later ones to be large.
}
\label{fig:treeoflife}
\end{figure}

\subsection{Source code}
This article relies on a number of computer calculations. In the interests of verifiability, the source code for all
these calculations are bundled with the {\tt arXiv} source of this article. After downloading the source, you'll find
a {\tt code/} subdirectory containing a number of Mathematica notebooks. These 
notebooks are referenced at the necessary points through the text.  As described 
above, equivalent calculations could also be performed by hand except for the 
computer calculations in Section \ref{sec:six} and parts of Section \ref{sec:rotationalev}.

\subsection{Acknowledgements}
Kim Morrison was supported by an Australian Research Council Discovery Early Career Researcher Award DE120100232,
and Discovery Projects DP140100732 and DP160103479. Emily Peters was supported by the NSF grant DMS-1501116.
Noah Snyder was supported by the NSF grant DMS-1454767. 
  All three authors were supported by DOD-DARPA grant HR0011-12-1-0009. Kim
Morrison would like to thank the Erwin Schr\"odinger Institute and its 2014 programme on ``Modern Trends in Topological
Quantum Field Theory'' for their hospitality.  We would like to thank Greg Kuperberg for a blog comment
\cite{how-to-almost-prove-the-4-color-theorem}
suggesting applying the discharging method to skein theory, Victor Ostrik for explaining his construction of the twisted Haagerup categories, and David Roe and Dylan Thurston for helpful suggestions.

\section{Trivalent categories}
In this section, we introduce the notion of a trivalent category, as a pivotal category which is `generated by a trivalent vertex'. In particular, every morphism in such a category is a linear combination of trivalent graphs (possibly with boundary) embedded in the plane, and indeed any such trivalent graph is allowed as a morphism.

In Appendix \ref{sec:local} we motivate trivalent categories for a wider audience --- particularly graph theorists ---
by explaining an equivalence between trivalent categories and certain skein theoretic invariants of planar trivalent
graphs.  While this equivalence is not essential for understanding the paper, reading the appendix may be useful for readers
unfamiliar with diagrammatic methods in category theory.

This category theoretic language will be useful, because the examples we will encounter along the way
come from fields of mathematics where category theory is very convenient. Nevertheless this point of view is not needed in the proofs
of the main theorems.  The key to translating category theoretic statements below into graph theoretic statements
is to remember that $\Hom_{\mathcal C}(1\rightarrow X^{\otimes n})$ is the vector space
of formal linear combination of planar trivalent graphs with $n$ boundary points modulo the relevant skein relations for the category $\cC$.

\vspace{11pt}
Recall that a (strict) pivotal category is a rigid monoidal category such that $x^{**} = x$ for all $x$. In this paper, all of our categories will be $\bbC$-linear.

Pivotal categories axiomatize the nicest possible theory of duals, and
correspondingly have a diagrammatic calculus allowing arbitrary planar
isotopies.  As usual, we have string diagrams representing morphisms, with
oriented strings labelled by objects of the category, and vertices (or
`coupons') labelled by morphisms of the category. The strings may have critical
points, which we interpret as the evaluation and coevaluation maps provided by
the rigid structure. Arbitrary planar isotopies of a diagram preserve the
represented morphism; it is critical that $2\pi$ rotations of the vertices are
allowed, and this corresponds exactly to strict pivotality.

Given a pivotal category $\cC$ and a chosen object $X$, we use the notation
$\cC_k = \Inv_\cC(X^{\otimes k}) = \Hom_\cC (1 \to X^{\otimes k})$ for the
`invariant spaces' of $X$. (If you know about planar algebras \cite{math.QA/9909027}, recall these vector
spaces form an unshaded planar algebra.) We say a category $\cC$ is
\emph{evaluable} if $\dim \Hom(1 \to 1) = 1$, and in fact $\Hom(1 \to 1)$ may be identified with
the ground field by sending the empty diagram to $1$. The category $\cC$ is
\emph{nondegenerate} if for every morphism $x: a \to b$, there is another
morphism $x': b \to a$ so $\operatorname{tr}(x x') \neq 0 \in \Hom(1 \to 1)$.

\begin{defn}
A \emph{trivalent} category $(\cC, X, \tau)$ is a nondegenerate evaluable pivotal category over $\bbC$ with 
an object $X$ with $\dim \cC_1 = 0$, $\dim{\cC_2} = 1$, and $\dim{\cC_3} = 1$, with a rotationally invariant morphism $\tau \in \cC_3$
called `the trivalent vertex', such that the category is generated (as a pivotal category) by $\tau$.
\end{defn}

We'll often simply refer to $\cC$ itself as a trivalent category.  

The rotational invariance of $\tau$ allows us to drop the ``coupon'' attached to
$\tau$ in string diagrams and treat it as an undecorated trivalent vertex.

We want one more simplification to our diagrammatic calculus: in the present situation it
turns out that we can always ignore the orientations on strings, because the
object $X$ is automatically symmetrically self-dual. Said another way, the
2-valent vertex corresponding to the self-duality $X \iso X^*$ is rotationally
symmetric.

\begin{lem}
The object $X$ is symmetrically self-dual.
\end{lem}

\begin{proof}
  Suppose that $\alpha: X \rightarrow X^*$ is any self-duality and that $\psi: X \rightarrow X \otimes X$ is
an inclusion.   Because $\cC$ is nondegenerate and $\dim \cC_3 = 1$, given any two non-zero 
maps $\beta: X \to X \otimes X$ and $\gamma: X \otimes X \to X^*$, 
$\operatorname{tr}(\alpha^{-1} \circ \gamma \circ \beta) \neq 0$, so 
$\alpha^{-1} \circ \gamma \circ \beta$ and  $\gamma \circ \beta$ are 
nonzero too. Taking $\gamma = \psi^* \circ (\alpha \otimes \alpha^*)$ and 
$\beta = \psi$, we see that the map
$$
\psi^* \circ (\alpha \otimes \alpha^*) \circ \psi = 
\begin{tikzpicture}[baseline=-2, scale=1.5];
	\node[rectangle, draw] (A) at (1,0) {$\psi$};
	\node[rectangle, draw] (B) at (-1,0) {\rotatebox{180}{$\psi$}};
	\node[rectangle, draw] (C) at (0,-.5) {$\alpha$};
	\node[rectangle, draw] (D) at (0,.5) {\rotatebox{180}{$\alpha$}};

\begin{scope}[thick, decoration={
    markings,
    mark=at position 0.5 with {\arrow{<}}}
    ] 
	\draw[postaction={decorate}] (D.0) to [in=180, out=0] (A.160);
	\draw[postaction={decorate}] (C.0) to [in=180, out=0] (A.200);
	\draw[postaction={decorate}] (D.180) to [in=0,out=180] (B.20);
	\draw[postaction={decorate}] (C.180) to [in=0, out=180] (B.-20);
	\draw[postaction={decorate}] (A.0)--  ++(0:5mm);
	\draw[postaction={decorate}] (B.180)--  ++(180:5mm);
\end{scope}
	\end{tikzpicture}
$$ is nonzero and manifestly rotationally invariant.
\end{proof}

Combining the symmetric self-duality of $X$ with the rotational invariance of $\tau$, we can interpret any unoriented planar trivalent graph with $n$ boundary points as an element of $\cC_n$.

To any trivalent category we assign several important parameters as follows.   
Since $\dim \cC_0=1$, any diagram with a loop in it is a multiple $d$ of the same diagram missing that loop. 
The loop value $d$ must be nonzero because it is the pairing of the 
unique-up-to-scalar element of $\cC_2$ with itself, and $\cC$ is nondegenerate.
  In addition, we must have a relation 
\begin{equation}
\label{eq:bigon}
\ngon[90]{2}
= b \cdot
\begin{tikzpicture}[baseline=.3cm]
\draw (0,0)--(0,1);
\end{tikzpicture}
\end{equation}
for some  parameter $b$ which is again nonzero, since the theta graph must be nonzero. Because one can rescale the trivalent vertex by a constant, without loss of generality we can assume $b=1$.  Finally, we see that \begin{equation}
\label{eq:triangle}
\ngon[90]{3}
= t \cdot
\begin{tikzpicture}[baseline=.1cm,scale=0.75]
\draw (0,0) -- (0,1);
\draw (0,0) -- (0.7,-0.5);
\draw (0,0) -- (-0.7,-0.5);
\end{tikzpicture},
\end{equation}
 although in this case the parameter $t$ can be zero.  These parameters $d$ and $t$ will be the key parameters in this paper.

We now give a simple example of a trivalent category, the `chromatic category'.  Further examples appear throughout this paper, in particular 
$SO(3)_q$ (which is essentially the same as the chromatic category) in the proof of Proposition \ref{prop:4:realization},
$ABA_{(d,t)}$ in the proof of Proposition \ref{prop:5:realization:ABA},
$(G_2)_q$ in Definition \ref{def:G2}, and the H3 category of \cite{MR2909758} 
which is given a trivalent presentation in Section \ref{sec:six}.

The chromatic category has as objects finite subsets of an interval, and a 
morphism between two such sets is a linear combination of trivalent graphs 
embedded in a strip between these intervals, subject to the following local 
relations:
\begin{align*}
\tikz[baseline=0]{\draw (0,0) circle (4mm);} & = n - 1 \\ 
\ngon[90]{2}
& = (n-2) \cdot
\begin{tikzpicture}[baseline=.3cm]
\draw (0,0)--(0,1);
\end{tikzpicture} \\
\ngon[90]{3}
& = (n-3) \cdot
\begin{tikzpicture}[baseline=.1cm,scale=0.75]
\draw (0,0) -- (0,1);
\draw (0,0) -- (0.7,-0.5);
\draw (0,0) -- (-0.7,-0.5);
\end{tikzpicture} \\
\drawH +  \twostrandid & = \drawI \; + \cupcap
\end{align*}
(Here the object $X$ is just a singleton on an interval, and the morphism $\tau$
is just the trivalent vertex.) One can verify that these relations suffice to
evaluate any closed trivalent graph: the `$I=H$' relation ensures that we can
reduce the size of a chosen face without increasing the total number of vertices
in the graph, and once there is a small enough face, one of the other relations
lets us reduce the total number of vertices. Indeed, in the case that the
parameter $n$ is an integer, one can see that this evaluation is exactly the
normalized chromatic number of the graph --- that is, the number of ways to
color the faces of the planar trivalent graph with $n$ colors such that adjacent
faces do not share a color, with the `outer face' of the planar diagram always
having a fixed color. The proof of this fact is that in each relation, the total
number of colorings is the same on either side of the relation.

Later we will see that this category is actually equivalent to the category we
call $SO(3)_q$ below, namely the category of representations of
$U_q(\mathfrak{sl}(2))$ consisting of representations whose highest weight is a
root, with the equivalence sending the object $X$ to the irreducible
$3$-dimensional representation and sending the trivalent vertex to some multiple
of the quantum determinant. The parameters match up according to the formula $n=
q^2+2+q^{-2}$. (This is of course a well known equivalence in quantum topology.)

Similarly, the $G_2$ spider defined in \cite{MR1403861} is an example of a
trivalent category, and it is equivalent to the category of representations of
$U_q(\mathfrak{g}_2)$ with $X$ the $7$-dimensional representation and $\tau$ the
quantum deformation of the defining invariant antisymmetric trilinear form.

There is a well-known theorem about nondegeneracy and negligibles which we will need.
\begin{prop} \label{prop:general:uniqueness}
An evaluable pivotal category has a unique maximal ideal, the negligible ideal. (cf. \cite[Proposition 3.5]{MR2979509})
\end{prop}

\begin{cor}
\label{cor:reductions=>uniqueness}
Given a collection of linear relations amongst planar trivalent graphs, such that any closed diagram can be reduced to 
a multiple of the empty diagram by those relations, there is a unique 
nondegenerate trivalent category satisfying those relations.
\end{cor}

\begin{remark}
Any trivalent category is automatically spherical.  To see this note that since 
$X$ is simple we need only check that the dimension of $X$ and the dimension of 
$X^*$ agree, but since $X$ is self-dual this is obvious.
\end{remark}

\section{Small graphs}

We will write $D(n,k)$ for the collection of trivalent graphs with $n$ boundary 
points and at most $k$ internal faces having four or more edges. 
For a fixed trivalent category, we write $M(n,k)$ for the matrix of bilinear inner 
products, i.e.
$$
\langle X, Y \rangle =
\begin{tikzpicture}[baseline=0]
\draw (-0.2,0) arc (180:0:1.2);
\draw (0,0) arc (180:0:1);
\draw (0.2,0) arc (180:0:0.8);
\node[draw,fill=white,circle,inner sep=8pt] (X) at (0,0) {$X$};
\node[draw,fill=white,circle,inner sep=8pt] (Y) at (2,0) {$Y$};
\end{tikzpicture},
$$
of the elements of $D(n,k)$,
and $\Delta(n,k)$ is the determinant of $M(n,k)$. 
Similarly, we will write $D^\square(n,k)$ for the collection of trivalent graphs 
with $n$ boundary points and at most $k$ internal faces having five or more edges, 
and analogously define $M^\square(n,k)$ and $\Delta^\square(n,k)$.   

\begin{prop}\label{thm:zeroalltheway}
For either $\mu=\emptyset$ or $\mu = \square$,
if there is a linear relation amongst diagrams in $D^\mu(n,k)$, then 
$\Delta^\mu(n',k')=0$ for all  $n' \geq n$ and $k' \geq k$.
\end{prop}

\begin{proof}
Take the  diagrams appearing in the relation and glue a fixed tree (with $n'-n$ leaves) to a fixed boundary point of each of them.
There is then a non-trivial relation amongst the resulting diagrams, and hence $\Delta^\mu(n',k')=0$ also.
\end{proof}

\begin{cor}
If $\Delta^\mu(n,k)=0$ and the diagrams in $D^\mu(n,k)$ span $\cC_n$, then $\Delta^\mu(n',k') = 0$ for all  $n' \geq n$,
$k' \geq k$.
\end{cor}

\begin{remark}
In all our examples, we will have enough conditions on the trivalent category that it will be possible to evaluate each of the $\Delta^\mu(n,k)$ that we consider as a rational function in $d$, $b$, and $t$.  We will always normalize to set $b=1$, but it is worth noting that you can recover the rational function up to an overall power of $b$, from the $b=1$ specialization as follows.  Notice that rescaling the trivalent vertex by $\lambda$ rescales the values of all closed planar trivalent graphs.  We can put a grading on closed planar trivalent graphs according to the number of trivalent vertices. Our parameters $d = \bigcirc, b = \tikz[scale=0.25,baseline=-2]{\draw (0,0) circle (1); \draw (1,0) -- (-1,0);} / d$, and
$t = \tikz[scale=0.4,baseline=-2]{\draw (90:1) -- (210:1) -- (330:1) -- (90:1) -- (0,0) -- (210:1) (0,0) -- (330:1);}/bd$
have gradings $0$, $+2$, and $+2$.  It is not difficult to see, by looking at the diagrams involved in the calculation, that $\Delta^\mu(n,k)$ is homogenous with respect to this grading.  So to recover the rational function from its $b=1$ factorization up to an overall power of $b$, we simply multiply each monomial by a power of $b$ to make it homogenous.
\end{remark}

\section{Diagrams with four boundary points} \label{sec:four}

Recall that for any trivalent category $\cC$ we get two numbers $(d,t)$ from the 
loop and the triangle, and furthermore $d \neq 0$.  In this section we prove
\begin{thm}\cite[Theorem 3.4]{1202.4396} \label{thm:4}
A trivalent category $\cC$ with $\dim \cC_4 \leq 3$ has $P_{SO(3)} = d+t-dt-2 = 0$ 
and must be either an $SO(3)_q$ category for $d = q^2+1+q^{-2}$ if $(d,t) \neq (-1,3/2)$,
 or $OSp(1|2)$ if  $(d,t) = (-1,3/2)$.
\end{thm}

(Throughout this paper we will use $P$'s with various subscripts to denote important polynomials in $d$ and $t$ whose
vanishing set corresponds to some existing trivalent category.  So, for example, $P_{SO(3)}$ is the polynomial which vanishes when $d$
and $t$ have the values that they have for quantum $SO(3)$.  By contrast, we will use $Q_{i,j}$ to denote polynomials
whose exact form is not important to the reader.  Here $i$ and $j$ are the degrees of the polynomial in $d$ and $t$
respectively.  The smaller polynomials are listed in the appendix, and all appear in computer readable form in the
{\tt polynomials/} directory of the {\tt arXiv} source of this article.)

This Theorem follows from three propositions.

\begin{prop}[Non-existence]
\label{prop:4:nonexistence}
For any $(d,t)$ not satisfying $P_{SO(3)}  = 0$ there are no trivalent
categories with $\dim \Span D(4,0) \leq 3$.
\end{prop}
\begin{prop}[Uniqueness]
\label{prop:4:uniqueness}
For every pair $(d,t)$ satisfying $P_{SO(3)}  = 0$ there is at most one trivalent category with $\dim \cC_4 \leq 4$.
\end{prop}
\begin{prop}[Realization]
\label{prop:4:realization}
The $SO(3)_q$ categories are trivalent categories with $\dim \cC_4 \leq
3$, and realize every pair $(d,t)$ satisfying $P_{SO(3)}  = 0$, except $(-1,3/2)$.  The remaining point $(-1,3/2)$ is realized by  $OSp(1|2)$.
\end{prop}

Although versions of these propositions were already proved in \cite{1202.4396}\footnote{
In \cite{1202.4396} the point $(d,t)=(-1,3/2)$ was not discussed, since in the subfactor context $d>0$.},
we give a slightly different argument which is easier to automate and thus scale to the needs of the later sections. 
For the first proposition, we use the dimension assumption to see that $\Delta(4,0)$ 
and $\Delta(4,1)$ vanish.  A short calculation shows that $P_{SO(3)}$ must vanish. 
(Later, this sort of calculation will be handled by Gr\"obner bases, but for now 
it's easy enough to do by hand.)   For the second proposition, we fix $(d, t)$
satisfying $P_{SO(3)}  = 0$. The proof divides into three phases. We use the degeneracy implied by $\dim \cC_4 \leq 3$ 
and an easy graph theoretic fact to show that $D(n,0)$ spans $\cC_n$ for all $n$. Specializing to $n = 4$, this shows 
that the kernel of $M(4,0)$ gives a relation in the category. The fact that this relation suffices to evaluate all 
closed diagrams shows that there is at most one trivalent category at this value of $(d,t)$.  Finally, the third 
proposition is a straightforward statement about a well-understood family of categories.

\begin{remark}
The proof of Theorem \ref{thm:4} only needs the weakening of Proposition \ref{prop:4:uniqueness} that covers the cases 
where $\dim \cC_4 \leq 3$. We will need the full strength later in the paper.
\end{remark}

\subsection*{Proof of Proposition \ref{prop:4:nonexistence} (Non-existence)}

The diagrams $D(4,1)$ are $$\twostrandid \, , \, \cupcap \, , \,  \drawI \, , \, \drawH \, , \, \ngon[45]{4}, $$
and they have matrix of inner products%
\newcommand{\sca}{0.3}
\newcommand{\twocircles}{
	\begin{tikzpicture}[scale=\sca,baseline=8]
	\draw (0,0) circle (.4cm);
	\draw (1,0) circle (.4cm);
	\end{tikzpicture}
}
\newcommand{\circlewithwaist}{
	\begin{tikzpicture}[scale=\sca]
	\draw (.5,-.2) .. controls (.4,-.2) .. (-45:.4cm) arc (-45:-315:.4cm) .. controls (.4,.2) .. (.5,.2);
	\begin{scope}[xshift=1cm, rotate=180]
		\draw (.5,-.2) .. controls (.4,-.2) .. (-45:.4cm) arc (-45:-315:.4cm) .. controls (.4,.2) .. (.5,.2);
	\end{scope}
	\end{tikzpicture}
}
\newcommand{\bigonwithwaist}{
	\begin{tikzpicture}[scale=\sca]
	\draw (.5,-.2) .. controls (.4,-.2) .. (-45:.4cm) arc (-45:-180:.4cm) node[shape=coordinate] (a) {} arc (-180:-315:.4cm) .. controls (.4,.2) .. (.5,.2);
	\begin{scope}[xshift=1cm, rotate=180]
		\draw (.5,-.2) .. controls (.4,-.2) .. (-45:.4cm) arc (-45:-180:.4cm) node[shape=coordinate] (b) {} arc (-180:-315:.4cm) .. controls (.4,.2) .. (.5,.2);
	\end{scope}
	\draw (a) .. controls ($(a)+(-0.5,0.0)$) and (-1,1) ..  (0.5,0.8) .. controls (2,1) and ($(b)+(0.5,0.0)$) .. (b);
	\end{tikzpicture}	
}
\newcommand{\bigon}{
	\begin{tikzpicture}[scale=\sca]
	\draw (0,-.35)--(0,.35);
	\draw (0,-.35) arc (-120:120:.4);	
	\draw (0,-.35) arc (300:60:.4);	
	\end{tikzpicture}
}
\newcommand{\earring}{
	\begin{tikzpicture}[scale=\sca]
	\draw (0,0) circle (.4cm);
	\draw (-.4,0) .. controls (-1,0) and (-1,0.6) .. (0.5,0.6) .. controls (2,0.6) and (2,0) .. (1.4,0);
	\draw (1,0) circle (.4cm);
	\end{tikzpicture}
}
\newcommand{\dumbbell}{
	\begin{tikzpicture}[scale=\sca]
	\draw (0,0) circle (.4cm);
	\draw (.4,0) -- (.6,0);
	\draw (1,0) circle (.4cm);
	\end{tikzpicture}
}
\newcommand{\drawIbox}{ \begin{tikzpicture}[scale=\sca,baseline=0cm]
 	\draw (0,.2) .. controls +(30:.3cm) .. (45:.8cm);
 	\draw (0,.2) .. controls +(150:.3cm) .. (135:.8cm);
	\draw (0,.2) -- (0,-.2);
 	\draw (0,-.2) .. controls +(-30:.3cm) .. (-45:.8cm);
 	\draw (0,-.2) .. controls +(-150:.3cm) .. (-135:.8cm);
	\draw (45:.8cm) -- (135:.8cm) -- (-135:.8cm) -- (-45:.8cm) -- cycle;
	\end{tikzpicture}
}
\newcommand{\drawHbox}{ \begin{tikzpicture}[scale=\sca,baseline=0cm,rotate=90]
 	\draw (0,.2) .. controls +(30:.3cm) .. (45:.8cm);
 	\draw (0,.2) .. controls +(150:.3cm) .. (135:.8cm);
	\draw (0,.2) -- (0,-.2);
 	\draw (0,-.2) .. controls +(-30:.3cm) .. (-45:.8cm);
 	\draw (0,-.2) .. controls +(-150:.3cm) .. (-135:.8cm);
	\draw (45:.8cm) -- (135:.8cm) -- (-135:.8cm) -- (-45:.8cm) -- cycle;
\end{tikzpicture}}
\newcommand{\cupcapbox}{\begin{tikzpicture}[scale=\sca,baseline=0cm]
	\draw (45:.8cm) .. controls (0,0) .. (135:.8cm);
	\draw (-45:.8cm) .. controls (0,0) .. (-135:.8cm);
	\draw (45:.8cm) -- (135:.8cm) -- (-135:.8cm) -- (-45:.8cm) -- cycle;
\end{tikzpicture}}
\newcommand{\twostrandidbox}{\begin{tikzpicture}[scale=\sca,baseline=0cm,rotate=90]
	\draw (45:.8cm) .. controls (0,0) .. (135:.8cm);
	\draw (-45:.8cm) .. controls (0,0) .. (-135:.8cm);
	\draw (45:.8cm) -- (135:.8cm) -- (-135:.8cm) -- (-45:.8cm) -- cycle;
\end{tikzpicture}}
\newcommand{\doublebigon}{
	\begin{tikzpicture}[scale=\sca,rotate=90]
	\draw (0,0) circle (.4cm);
	\draw (-.4,0) .. controls (-1,0) and (-1,0.6) .. (0.5,0.6) .. controls (2,0.6) and (2,0) .. (1.4,0);
	\draw (0.4,0) -- (0.6,0);
	\draw (1,0) circle (.4cm);
	\end{tikzpicture}
}
\newcommand{\tet}{
	\begin{tikzpicture}[scale=\sca]
	\draw (0,-.35)--(0,.35);
	\draw (0,-.35) arc (-120:0:.4) node[shape=coordinate] (a) {} arc (0:120:.4);	
	\draw (0,-.35) arc (300:180:.4) node[shape=coordinate] (b) {} arc (180:60:.4);	
	\draw (a) .. controls ($(a)+(1,0)$) and (0.8,-0.7) .. (0,-0.7) .. controls (-0.8,-0.7) and ($(b)+(-1,0)$) .. (b);
	\end{tikzpicture}
}
\newcommand{\cappedsquare}{
	\begin{tikzpicture}[scale=\sca]
	\draw (.4,.4) -- (.4,-.4)--(-.4,-.4) .. controls (-.8,-.8) and (.8,-.8) .. (.4,-.4);
	\draw (-.4,-.4) -- (-.4,.4)--(.4,.4) .. controls (.8,.8) and (-.8,.8) .. (-.4,.4);	
	\end{tikzpicture}
}
\newcommand{\cappedsquarearc}{
	\begin{tikzpicture}[scale=\sca]
	\draw (.4,.4) -- (.4,-.4)--(-.4,-.4) .. controls (-.8,-.8) and (.8,-.8) .. (.4,-.4);
	\draw (-.4,-.4) -- (-.4,.4)--(.4,.4) .. controls (.8,.8) and (-.8,.8) .. (-.4,.4);
	\draw (0,0.7) .. controls (0,1.0) and (1,1.0) .. (1,0) .. controls (1,-1) and (0,-1) .. (0,-0.7);
	\end{tikzpicture}
}%
\begin{align*}
\left(
\mathfig{0.3}{inner-product-matrix}
\right) & = 
\begin{pmatrix}
d^2 & d & d & 0 & d \\
d & d^2 & 0 & d & d \\ 
d & 0 & d & td & t^2d \\
0 & d & td & d & t^2d \\
d & d & t^2d & t^2d & \Cube
\end{pmatrix}.
\end{align*}

Recall $\Delta(4,1)$ is the determinant of this matrix, and $\Delta(4,0)$ is the determinant of the minor leaving off the last row and column.

\begin{fact}
In a trivalent category,
\begin{align*}
\Delta(4, 0) = d^4 (d+t-dt-2)(d+t+dt).
\end{align*}
\end{fact}

\begin{fact}
In a trivalent category,
\begin{align*}
\Delta(4, 1) = - d^4(d+t-dt-2)(2d+2dt-4dt^2+2dt^4+2d^2t^4 - \Cube(d+t+dt))
\end{align*}
\end{fact}

A trivalent category with $\dim \Span D(4,0) \leq 3$ must have $\Delta(4,0) =
\Delta(4,1) = 0$.  Since $d \neq 0$, we must have either $P_{SO(3)}  = 0$, or $(d+t+dt) = 0$ and $2d+2dt-4dt^2+2dt^4+2d^2t^4 = 0$.  In the latter case, solving gives $(d,t) =
\left( \frac{1 \pm \sqrt{5}}{2}, \frac{1 \mp \sqrt{5}}{2}\right)$ which also satisfies $P_{SO(3)} = 0$. \qed

\begin{remark}
In this paper we work over $\bbC$ throughout, but it is also interesting to ask
these questions in other characteristics.  The argument above works with no
modifications outside of characteristic $2$.  In characteristic $2$, the
argument breaks down since when $(d+t+dt) = 0$, we have that $\Delta(4,1)$ is
automatically zero.  However, in characteristic $2$, a closer look shows that
$\Delta(4,0) = d^4 P_{SO(3)}^2$, so the conclusion holds anyway.  We will ignore
non-zero characteristic in the rest of the paper.
\end{remark}

\subsection*{Proof of Proposition \ref{prop:4:uniqueness} (Uniqueness)}
We fix a value of $(d,t)$ satisfying $P_{SO(3)}  = 0$.

\begin{lem}
\label{lem:I=H=>spanning}
If there is a relation of the form
\begin{equation}
\label{eq:I=H}
 \drawH \; = \alpha \; \drawI \; + \beta \; \twostrandid \; + \gamma \; \cupcap,
 \end{equation}
then 
any trivalent graph in $\cC_n$ can be reduced to $\operatorname{span} D(n,0)$.
\end{lem}
\begin{proof}
Applying this relation to the largest face gives a sum of terms with either 
fewer faces or the same number of faces but with a smaller largest face.  By 
induction, we can write any diagram as a sum of terms with no internal faces.
\end{proof}

\begin{lem}
In a trivalent category with $\dim \Span D(4,0) \leq 3$, the diagrams $D(n,0)$ span $\cC_n$ for all
$n$.
\end{lem}
\begin{proof}
If the two diagrams \scalebox{0.6}{$\twostrandid$} and \scalebox{0.6}{$\cupcap$} are linearly dependent, we obtain a relation as in Equation \eqref{eq:I=H} by adding an ``H'' along the top boundary of both 
pictures. Otherwise, there must be a relation amongst the four diagrams in $D(4,0)$, with the coefficient of at least one of \scalebox{0.6}{$\drawH$} and \scalebox{0.6}{$\drawI$} being nonzero. Rescaling and rotating gives a relation as in Equation \eqref{eq:I=H}.
\end{proof}

\begin{lem}
\label{lem:4:spanning}
In a trivalent category with $\dim \cC_4 \leq 4$, the diagrams $D(4,0)$ span $\cC_4$.
\end{lem}
\begin{proof}
If $\dim \Span D(4,0) \leq 3$, then the previous Lemma applies.  Otherwise, $\dim \Span D(4,0) = 4$ and the conclusion
is immediate.
\end{proof}

Note that on $P_{SO(3)}  = 0$, $d \neq 1$.

\begin{lem} \label{lem:recognition:SO3}
In a trivalent category where $D(4,0)$ spans $\cC_4$ and $P_{SO(3)}  = 0$, there is
a relation of the form
$$  \drawH \; - \; \drawI \; + \frac{1}{d-1}\; \twostrandid \; -  \frac{1}{d-1}\; \cupcap = 0.$$
\end{lem}
\begin{proof}
Since $D(4,0)$ spans, any element of $\cC_4$ with inner product $0$ with all elements of $D(4,0)$ must vanish.
Computing the kernel of $M(4,0)$ as given above, we find this relation.
\end{proof}

To prove the proposition, we observe that by Lemma \ref{lem:4:spanning}, $\cC_4 = \Span D(4,0)$, so by Lemma \ref{lem:recognition:SO3}
there is a relation of the form in Equation \eqref{eq:I=H}. 
Finally, Lemma \ref{lem:I=H=>spanning} shows that this relation suffices to evaluate all closed diagrams as a
multiple of the empty diagram, and Corollary \ref{cor:reductions=>uniqueness} shows that there is a unique trivalent category
at this value of $(d, t)$.
\qed

\subsection*{Proof of Proposition \ref{prop:4:realization} (Realization)}

The curve $P_{SO(3)}  = 0$ is rational and can be parameterized by $d= \delta^2 -1$ 
and $t= \frac{\delta^2 -3}{\delta^2-2}$ where $\delta \neq \pm \sqrt{2}$.
It is more usual to change variables to $q+q^{-1} = \delta$ where $q$ is not a 
primitive $8$th root of unity.  Under this parameterization $q$ and $q^{-1}$ are sent to the same point.

So long as $\delta \neq 0$, that is $(d,t) \neq (-1,3/2)$, we have the following
realization.  Let $\mathrm{TL}_\delta$ be the Temperley-Lieb category, which
consists of linear combinations of planar tangles with the circle equal to
$\delta$.  This is not a trivalent category, since there is no trivalent vertex.
However, we can interpret trivalent graphs in $\mathrm{TL}_\delta$ as follows.
Take a trivalent graph and replace each strand with a pair of strands attached
to the second Jones-Wenzl projection $$f^{(2)} = \twostrandid \;
-\frac{1}{\delta} \; \cupcap,$$ and replace each trivalent vertex by
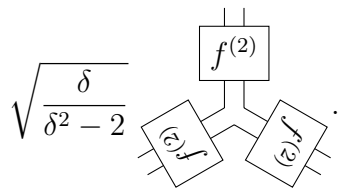
\begin{equation*}
\sqrt{\frac{\delta}{\delta^2-2}} \; \begin{tikzpicture}[baseline]
\foreach \n in {0, 120, 240} {
	\draw (30+\n:0.15) -- ($(30+\n:0.15) + (90+\n:1.33)$) (30+\n:0.15) -- ($(30+\n:0.15)+(-30+\n:1.33)$);
}
\foreach \n in {0, 120, 240} {
	\node[draw, rectangle, fill=white,rotate=\n] at (90+\n:0.8) {$f^{(2)}$};
}
\end{tikzpicture}.
\end{equation*}

This trivalent category is called $SO(3)_q$ where $\delta = q+q^{-1}$.  The trivalent vertex is normalized so that $b=1$, and quick calculation shows that $d= \delta^2 -1 = q^2+1+q^{-2}$ and $t= \frac{\delta^2 -3}{\delta^2-2} = \frac{q^2-1+q^{-2}}{q^2+q^{-2}}$.

The remaining point $(-1,3/2)$ is realized in a somewhat different way.  The Lie supergroup $OSp(1|2)$ has a standard $(1|2)$-dimensional representation which we denote $X$.  This representation is simple, and using highest weight theory, $X \otimes X \cong 1 \oplus X \oplus Y$ for some simple object $Y$ distinct from $1$ and $X$.  Thus $X$ is self-dual and the map $X \otimes X \rightarrow X$ gives a map $X^{\otimes 3} \rightarrow 1$.  A direct calculation shows that this map factors through the symmetric cube of $X$, and thus gives a trivalent vertex.  We normalize this vertex so that the value of the bigon is $1$.  This gives a trivalent category which we denote $OSp(1|2)$ which has $\dim \cC_4 = 3$ and $d=-1$.  Thus, it must realize the remaining point $(-1, 3/2)$.
 \qed
  
 \begin{remark}
We will abuse notation somewhat and use $SO(3)_{\pm i}$ to refer to $OSp(1|2)$.  This can be justified by giving an appropriately modified definition of $SO(3)_q$.  Specifically the categories of representations of $U_{\pm i q}(\mathfrak{su}(2))$  and $U_q(\mathfrak{osp}(1|2))$ are closely related but non-isomorphic due the appearance of certain signs. However, if you restrict attention to the representations of even highest weight the categories become equivalent because these signs all vanish \cite{MR1230843, MR3010460}.

This strange point $(-1,3/2)$ can also be realized by the nondegenerate quotient of $(G_2)_{\pm i}$.  This is a straightforward calculation, but we will delay it until Remark \ref{rem:G2-i} since we will discuss $(G_2)_q$ in much more detail in the next section.  
 \end{remark}

\begin{remark}
When $d$ is generic on $P_{SO(3)}  = 0$, the dimension of $\cC_n$ is given by
the Motzkin sums $1,0,1,1, 3,6,15,36,\ldots$ (A005043 in \cite{EIS}).  More
specifically, by computing the radical of the inner product, all of these
categories have $\dim \cC_4 = 3$ unless $d$ satisfies $d^2=d+1$.  In this last
case (where $d$ is the golden ratio or its Galois conjugate) instead we have
$\dim \cC_4=2$ and the category satisfies the additional skein relation $$\drawI
\; = \; \twostrandid \; -  \frac{1}{d}\; \cupcap.$$  These two examples are
often called the golden categories or the Fibonacci categories.  For these 
categories the dimensions are given by Fibonacci numbers $1,0,1,1,2,3,5,
\ldots$.
\end{remark}

\begin{remark}
One can make sense of $SO(3)_q$ at an $8$th root of unity, by not including the 
$\sqrt{\frac{\delta}{\delta^2-2}}$ factor in the trivalent vertex.  With this 
normalization, $b=0$ and so the category is degenerate.  Its nondegenerate 
quotient has no trivalent vertices.
\end{remark} 

\subsection{Cubic categories}

\begin{defn}
A \emph{cubic category} is a trivalent category $\cC$  with $\dim{\cC_4} = 4$.
\end{defn}

\begin{prop}
\label{prop:cubic}
For any cubic category,
\begin{enumerate}[(a)]
\item the diagrams $D(4,0)$ form a basis of $\cC_4$,
\item $d+t+dt \neq 0$ and $P_{SO(3)} = d+t-dt-2 \neq 0$,
\item $\Cube = \frac{2d+2dt-4dt^2+2dt^4+2d^2t^4 }{d+t+dt}$, and
\item the square satisfies the following relation
\begin{multline} \label{eq:square}  
  \ngon[45]{4} =  \frac{d t^2+t^2-1}{d t+d+t}
 \left( \; 
\drawI
\; + \;
 \drawH
\; \right)
 \; \\+ \; \frac{-t^2+t+1}{d t+d+t}
\left( \; 
\twostrandid
\; + \;
 \cupcap
\; \right)
\end{multline} 
\end{enumerate}
\end{prop}
\begin{proof}
By Lemma \ref{lem:4:spanning} if $D(4,0)$ is dependent then $\dim \cC_4 \leq 3$.  Hence, $D(4,0)$ must be
independent, and so a basis of $\cC_4$.  Hence,  $\Delta(4,0) \neq 0$, which implies 
$d+t+dt \neq 0$ and $d+t-dt-2 \neq 0$. On the other hand, we must have $\Delta(4,1) = 0$, 
giving (c). Finally,  equation (\ref{eq:square}) 
is in the radical of the inner product on $D(4,1)$.
\end{proof}

We now identify the minimal idempotents in $\cC_4$ for any cubic category
(subject to a certain quantity being invertible).

\begin{prop}
\label{thm:idempotents-and-traces}
Suppose $\cC$ is a cubic category with parameters $d$ and $t$.
If $$\xi=\sqrt{d^2 t^4+2 d \left(t^4-2 t^3-t^2+4
   t+2\right)+\left(t^2-2 t-1\right)^2}$$ is nonzero, then the
four minimal idempotents in $\cC_4$ (with respect the multiplication via vertical stacking) are
\begin{align*}
\iota & = \frac{1}{d} \cupcap \\
x & = \drawI \\
y_\pm & = 
\frac{-(d+1) t^2 \pm\xi +1}{\pm2 \xi } \twostrandid
+ \frac{d \left(t^2-2 t-2\right) \mp \xi +t^2-2 t-1}{\pm 2 d \xi } \cupcap \\
& \quad - \frac{d (t+2) t \pm \xi +t^2+1}{\pm 2 \xi } \drawI
+ \frac{d t+d+t}{\pm \xi } \drawH
\end{align*}
with dimensions
\begin{align*}
\tr \iota & = 1 \\
\tr x & = d \\
\tr {y_\pm} & = -\frac{d^3 t^2+d^2 \left(\mp\xi +2 t^2+2 t-1\right)+d
   (\pm\xi +2 t+3)\pm\xi -t^2+2 t+1}{\pm 2 \xi }
\end{align*}
We note that $\tr{y_\pm}$ never vanishes since $d \neq 0$ and $P_{SO(3)} \neq 0$.
\end{prop}
\begin{proof}
This is a direct calculation using Equation \eqref{eq:square}, performed in the Mathematica notebook {\tt
code/idempotents.nb} available with the {\tt arXiv} source of this article. That $\iota$ and $x$
are minimal idempotents is clear. We then solve the quadratic equations $y_\pm \iota = y_\pm x  = 0$, $y_+ + y_- = 1 - \iota - x$,
and
$y_\pm^2 = y_\pm \neq 0$.
\end{proof}

\begin{remark}
When $\xi$ vanishes there is no basis of projections and so the category is not 
semisimple.  This does not happen in any of our examples.  (Note that for $G_2$ 
at $q=\pm i$ we have that $\xi$ vanishes, but at this special value $G_2$ is no 
longer cubic.  See Remark \ref{rem:badpoints}.)
\end{remark}

\begin{lem}
In any cubic category, if $n+2k < 12$, all entries of 
$M^\square(n,k)$ can be written as rational functions in $d$ and $t$.
\end{lem}
\begin{proof}
There are fewer than twelve faces in each inner product appearing in 
$M^\square(n,k)$, because the number of faces is bounded by $n+2k$. Any 
polyhedron with fewer than twelve faces has at 
least one face which is a square or smaller. The relations for simplifying bigons, 
triangles and squares (Equations \eqref{eq:bigon}, \eqref{eq:triangle}, and 
\eqref{eq:square}) then suffice to rewrite this polyhedron as a linear combination 
of  polyhedra with strictly fewer faces; repeating the argument completely 
evaluates the original polyhedron as a function of $d$ and $t$.

In fact, the denominators in these rational functions are always powers of 
$Q_{1,1} = dt + d + t$, the denominator appearing in Equation \eqref{eq:square}.
\end{proof}

We unapologetically state the values of these determinants as facts, even though for larger $n$ and $k$ they are the
result of quite intensive calculations (computing the inner products is already time consuming, and subsequently the
determinant is even harder). Of course, a computer is doing these calculations (see {\tt
code/ComputingInnerProducts.nb}).

A careful reader will note that the matrices $M^\square(8,0)$, $M^\square(8,1)$,
$M^\square(9,0)$, $M^\square(9,1)$, $M^\square(10,0)$, and $M^\square(11,0)$ are
covered by the above lemma, but we do not compute their determinants in what follows. They seem
quite difficult without very considerable computational resources. In any case,
our analysis of small skein theories meets another hurdle first; we can't even
compute the intersections of $\Delta^\square(7,0)$ and $\Delta^\square(7,1)$, or
of $\Delta^\square(7,1)$ and $\Delta^\square(7,2)$.

\section{Diagrams with five boundary points}

In this section we study diagrams with five boundary points. Note that 
\begin{align*}
D(5,0) & = \left\{
\mathfig{0.1}{tree1}, \mathfig{0.1}{tree2}, \mathfig{0.1}{tree3}, \mathfig{0.1}{tree4}, \mathfig{0.1}{tree5}, \right. \\
& \qquad \left. \mathfig{0.1}{forest1}, \mathfig{0.1}{forest2}, \mathfig{0.1}{forest3}, \mathfig{0.1}{forest4}, \mathfig{0.1}{forest5}\
	\right\}
\end{align*}
is a set of 10 diagrams, and $$D^\square(5,1) \setminus D(5,0) = \left\{ \ngon[90]{5} \right\}$$ has just one element.

\begin{thm}
\label{thm:5}
If $\cC$ is a cubic category (that is $\dim \cC_4 = 4$) and $\dim \cC_5 \leq 10$, then either
\begin{enumerate}
\item $(d,t)$ satisfy $P_{ABA}= t^2-t-1 =0$ and $\cC$ is one of the ABA categories described below, or
\item $(d,t)$ satisfy $$P_{G_2} = d^2 t^5+2 d t^5-4 d t^4-d t^3+6 d t^2+4 d t+d+t^5-4 t^4+t^3+7 t^2-2=0$$ and $\cC$ is $(G_2)_q$ with
\begin{align*}
d & =q^{10} + q^8 + q^2 + 1 + q^{-2} + q^{-8} + q^{-10}, \\ 
\intertext{and}
t & = - \frac{q^2 - 1 + q^{-2}}{q^4 + q^{-4}}.
\end{align*}
\end{enumerate} 
\end{thm}

Theorem \ref{thm:5} follows from four propositions.
\begin{prop}[Non-existence]
\label{prop:5:nonexistence}
For any $(d,t)$ not satisfying $P_{ABA}=0$ or $P_{G_2}=0$ there are no 
trivalent categories with $\dim \cC_4 = 4$ and $\dim \Span D^\square(5,1) \leq 10$.
\end{prop}
\begin{prop}[Uniqueness]
\label{prop:5:uniqueness}
For every pair $(d,t)$ satisfying $P_{ABA}=0$ or $P_{G_2}=0$ there is at most 
one trivalent category with $\dim \cC_4 = 4$ and $\dim \cC_5 \leq 11$.
\end{prop}
\begin{prop}[Realization]
\label{prop:5:realization:ABA}
The trivalent categories $ABA_{(d,t)}$, defined below via a free product 
construction, satisfy $\dim \cC_4 = 4$ and $\dim \cC_5 \leq 10$, and realize 
every pair $(d,t)$ satisfying $P_{ABA}=0$ and $P_{SO(3)} \neq 0$.
\end{prop}
\begin{prop}[Realization]
\label{prop:5:realization:G2}
The $(G_2)_q$ categories are trivalent categories with $\dim \cC_4 = 4$ and 
$\dim \cC_5 \leq 10$, and realize every pair $(d,t)$ satisfying $P_{G_2}=0$ 
and $P_{SO(3)} \neq 0$.
\end{prop}

(In fact, in this section we only need the weakening of Proposition 
\ref{prop:5:uniqueness} that covers the cases where $\dim \cC_5 \leq 10$. We 
will need the full strength in the next section.)

\begin{remark}
\label{rem:badpoints}
By non-degeneracy $d \neq 0$.  By Proposition \ref{prop:cubic}, any cubic category has $P_{SO(3)} \neq 0$ and $d+t +d t \neq 0$. In this remark we catalog what happens at the remaining special points where one of these polynomials does vanish.

The points on the intersection of the $P_{ABA}$ curve and $P_{SO(3)}(d+t+d t) =0 $ are $(d,t) = (\tau, \bar{\tau})$ and $(\bar{\tau}, \tau)$.  These two points do correspond to trivalent, but non-cubic, categories: the golden categories with $\dim \cC_2 = 2$.

The intersection points of the $P_{G_2}$ curve and $P_{SO(3)}(d+t+d t) =0 $ are:
\begin{enumerate}
\item $(d,t)=(-1,3/2)$, corresponding to $q$ a primitive $4$th root of unity.
\item $(d,t)=(-2,-2)$, corresponding to $q$  a primitive $3$rd or $6$th root of unity.
\item $(d,t)=(2,0)$, corresponding to $q$ a primitive $12$th root of unity.
\item The two points $(d,t) = (\tau, \bar{\tau})$ and $(d,t)=(\bar{\tau},\tau)$, corresponding to $q$ a primitive $30$th root of unity.
\end{enumerate}
The first of these cases is the special trivalent non-cubic category $(G_2)_{q=\pm i} = OSp(1|2)$. For the second, the bigon for $G_2$ at that value is $0$, so the non-degenerate quotient is not trivalent.  The third case corresponds to the trivalent non-cubic category $SO(3)_{\zeta_{12}}$.  The last case corresponds to the golden categories with $\dim \cC_2 = 2$.

As always we assume $d \neq 0$, corresponding to $q$ is not a primitive $7$th, $14$th, or $24$th root of unity, as otherwise the non-degenerate quotient is trivial.
\end{remark}

\begin{remark}
If $\cC$ and $\cD$ are two pivotal categories, then their tensor product $\cC \boxtimes \cD$ has morphisms consisting of a red planar diagram and a blue planar diagram superimposed on each other, where the blue parts live in $\cC$ and the red parts live in $\cD$ and the red diagrams are allowed to cross the blue diagrams in a symmetric way.  Clearly the dimensions of the box spaces multiply, as $$\Hom_{\cC \boxtimes \cD}(1 \to (X \boxtimes Y)^{\otimes n}) \iso  \Hom_{\cC}(1 \to X^{\otimes n}) \otimes  \Hom_{\cD}(1 \to Y^{\otimes n}),$$ and the invariants of closed diagrams also just multiply.  Let $\cG_\tau$ denote the Golden category with loop value $\tau$ (and triangle value $\bar{\tau}$).  The tensor products $\cG_\tau \boxtimes  \cG_\tau$, $\cG_\tau \boxtimes  \cG_{\bar{\tau}}$, and $\cG_{\bar{\tau}} \boxtimes  \cG_{\bar{\tau}}$ give pivotal categories with a distinguished trivalent vertex with $n$-boundary point spaces of dimension $(1,0,1,1,4,9, \ldots)$.  It is natural to wonder how they fit into our classification.  The category $\cG_\tau \boxtimes  \cG_{\bar{\tau}}$ has $(d,t)=(-1,-1)$ which lies on the $G_2$ curve and corresponds to $q$ a primitive $20$th root.  The categories $\cG_\tau \boxtimes  \cG_\tau$ and $\cG_{\bar{\tau}} \boxtimes  \cG_{\bar{\tau}}$ have $(d,t)$ being $(\tau^2, \bar{\tau}^2)$ and $(\bar{\tau}^2, \tau^2)$.  These points lie on the $\mathrm{SO}(3)$ curve corresponding to $q$ being a primitive $20$th root of unity (half of which give $d=\tau^2$ and half of which give $d=\bar{\tau}^2$).  In particular, these categories are not generated by the trivalent vertex.  Instead, they correspond to the even part of the $D_6$ subfactor and its Galois conjugate.  These categories contain a trivalent subcategory $\mathrm{SO}(3)_{\zeta_{20}}$ and a single extra generating $4$-box satisfying a version of the relations from \cite{MR2559686}.  The relationship between $D_6$ and the tensor products of golden categories is explained by level-rank duality between $\mathrm{SO}(3)_4$ and $\mathrm{SO}(4)_3$ together with the coincidence of Lie algebras between $\mathfrak{so}(4)$ and $\mathfrak{sl}(2) \oplus \mathfrak{sl}(2)$, as explained in \cite[Theorem 4.1]{MR2783128}.
\end{remark}

We follow the same outline as in the previous section.

\subsection*{Proof of Proposition \ref{prop:5:nonexistence} (Non-existence)}

\begin{fact}
In any trivalent category, 
\begin{align*}
\Delta(5, 0) = d^{10} P_{ABA}^2 P_{SO(3)}^4 Q_{1,2}.
\end{align*} 
\end{fact}

\begin{fact}
In any trivalent category,
$$
\Delta^\square(5,1)  = d^{10} P_{ABA}^2 P_{SO(3)}^4  \left(-5 d t\left( d t^5+2  t^5 -2t^4 -2 t^3 +2  t^2 + t\right)+ Q_{1,2} \PentPrism \right),$$
and in a cubic category, this specializes to
$$\Delta^\square(5, 1)  = d^{11} P_{ABA}^3 P_{SO(3)}^5 P_{G_2} Q_{1,1}^{-2}.$$
\end{fact}

A cubic category with $\dim \Span D^\square(5,1) \leq 10$ must have $\Delta^\square(5,1) = 0$.  Since $d \neq 0$, we must have one of $P_{SO(3)} = 0$, $P_{ABA} = 0$, or $P_{G_2} = 0$.    Since $P_{SO(3)} \neq 0$, the proposition is proved.
\qed

\subsection*{Proof of Proposition \ref{prop:5:uniqueness} (Uniqueness)}

We fix a value of $(d,t)$ satisfying $P_{G_2}$ or $P_{ABA}$.  

It is well-known that any planar trivalent graph has a pentagon or smaller face.
We want an analogue of this result for open planar trivalent graphs, that is,
planar trivalent graphs having boundary.  (This is similar to the analysis in
\cite{MR1403861}.) In order to state this result we introduce some language.

\begin{defn}\label{defininggrowth}
An \emph{open} planar trivalent graph is a planar trivalent graph in the disc which meets the boundary in $n \geq 1$
specified
points. An open planar trivalent
graph has two kinds of faces, \emph{internal faces} and \emph{boundary
faces}, depending on whether they touch the boundary of the disc.

We say a planar trivalent graph is \emph{connected} if the vertices and edges (whether internal edges or boundary edges)
form a connected topological space.

We say that an open planar trivalent graph is \emph{boundary connected} if every component meets the boundary.
Note that a face need not be topologically a disc, unless the graph is boundary connected.

A \emph{boundary region} of a connected planar trivalent graph is a small neighborhood of some contiguous proper subset
of the
boundary faces.  A boundary region is called a \emph{growth region} if the 
number of edges meeting the boundary is greater than the number of edges not 
meeting the boundary.  Figure \ref{fig:growthregion} illustrates these definitions.
\end{defn}

\begin{figure}[ht]
\begin{tikzpicture}
\draw[dashed] (2,0) arc (0:180:2cm);
\draw (10:2cm) -- (30:1cm) -- (50:2cm);
\draw (170:2cm) -- (150:1cm) -- (130:2cm);
\draw (30:1cm)--(45:.5cm) -- ++(-90:.5cm);
\draw (150:1cm)--(135:.5cm) -- ++(-90:.5cm);
\draw (45:.5cm)--(90:.5cm)--(135:.5cm);
\draw (90:.5cm)--(90:2cm);
\node at (30:1.6cm) {$D$};
\node at (60:1cm) {$C$};
\node at (120:1cm) {$B$};
\node at (150:1.6cm) {$A$};
\end{tikzpicture}
\caption{Some boundary faces in a portion of an open planar trivalent graph.  
The boundary region which is a small neighborhood of faces $A$, $B$, $C$, and 
$D$ is a growth region -- it has five edges meeting the boundary and two which 
don't.  The boundary region which is a small neighborhood of $B$ and $C$ is not 
a growth region --- it has three edges meeting the boundary and four edges which 
don't. }
\label{fig:growthregion}
\end{figure}
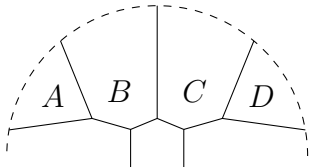

We now restrict our attention to boundary connected open trivalent graphs.
We assign a \emph{charge} to each face as follows.  Let $n$ be the number of edges meeting that face, and let $m$ be the
number of disjoint boundary intervals meeting that face.
Now assign the charge $6-n-2m$.  In particular, an internal $n$-gon face is assigned $6-n$, so the only positively
charged internal faces are pentagons or smaller.   A boundary face that only meets the boundary once and touches $n$
edges is
given charge $4-n$.   A standard Euler characteristic argument shows that the total charge of the graph is $6$.
\begin{lem}\label{lem:boundarycharge}
In a connected open planar trivalent graph, any boundary region for which the 
sum of the charges on its constituent boundary faces is at least $2$ is a growth region.
\end{lem}
\begin{proof}
Once we assume a graph is connected, each boundary face meets the boundary in a single interval.
Now, if a boundary region has $I$ incoming edges and $O$ outgoing edges, then the total charge of the boundary faces is
$O-I+1$.
\end{proof}

We say that an internal face is \emph{small} if it has $5$ or fewer sides.

\begin{lem}
Any connected open planar trivalent graph has an internal small face or a growth region.
\end{lem}
\begin{proof}
Since the total charge is $6$, either  the boundary charge is at least $6$, or 
there is a positively charged internal face, which must be a small face. If the boundary charge is 6 or more, 
then there are at least two boundary faces because a single boundary face has 
charge $4-n$ (where $n$ is the number of edges it touches).
Thus, we can  divide the boundary into two proper sub-regions.  One of these has 
charge three or greater, so by Lemma \ref{lem:boundarycharge} this is a growth region.
\end{proof}

It follows that we can easily inductively enumerate all boundary connected open graphs with $n$ boundary points and no
internal small faces by first enumerating the connected graphs by attaching 
growth regions to open graphs with strictly fewer boundary points and with
no internal small faces, and then writing down the all the planar unions of such graphs.

\begin{cor} \label{cor:nopents}
If $n \leq 5$, then any boundary connected open planar trivalent graph with $n$ boundary points and no small faces is in
$D(n,0)$.
\end{cor}

Of course, when $n = 6$ there's an open graph with a single hexagonal face which does not lie in $D(6,0)$.  

\begin{lem}
\label{lem:pentagon-reduction}
Suppose $\cC$ is a category generated by a trivalent vertex, with relations reducing $n$-gons for each $n \leq 4$.
Suppose further there is some relation between the diagrams in $D^\square(5,1)$. Then there is a relation reducing the
$5$-gon (as linear combination of diagrams in $D(5,0)$).
\end{lem}
\begin{proof}
If the relation between the diagrams in $D^\square(5,1)$ already includes the
pentagon, we are done. Otherwise there must be a relation only amongst the
diagrams in $D(5,0)$.  If this relation involves any of the diagrams of the form
$\mathfig{0.04}{tree1heavy}$ we can add an $H$ to the boundary of this relation and
obtain a relation writing a pentagon as a linear combination of diagrams without
internal faces.   Otherwise, if there's a relation only involving the diagrams
of the form $\mathfig{0.04}{forest1heavy}$, we find this diagram as a sub-diagram of
a pentagon consisting of a vertex and its opposite edge.  We  then apply the
relation inside the pentagon, and obtain another relation writing a pentagon as
a linear combination of  diagrams without internal faces. (One can readily
verify no pentagons appear in other terms.)
\end{proof}

\begin{cor}
\label{cor:5:evaluation}
Given fixed relations reducing $n$-gons for each $n \leq 4$, and some relation between the diagrams in $D^\square(5,1)$,
there is at most one trivalent category satisfying these relations.
\end{cor}
\begin{proof}
As any closed trivalent graphs contains an $n$-gon with $n \leq 5$, and by the previous Lemma we can reduce this
diagram, we see that the available relations suffice to evaluate all closed diagrams. By Corollary \ref{cor:reductions=>uniqueness}
we are done.
\end{proof}

\begin{lem}
\label{lem:5:dependent=>spans}
Suppose $\cC$ is a cubic category, with a relation between the diagrams in $D^\square(5,1)$. The $\cC_5$ is spanned by
$D(5,0)$.
\end{lem}
\begin{proof}
By Lemma \ref{lem:pentagon-reduction}, we can reduce any diagram in $\cC_5$ to a linear combination of diagrams without
small faces. By
Corollary
\ref{cor:nopents}, these are in $D(5,0)$.
\end{proof}

\begin{lem}
\label{lem:ABA-relation}
In a cubic category where $D^\square(5,1)$ spans $\cC_5$ and $P_{ABA}=0$ there are relations:
\begin{align}
\label{eq:free1}
\mathfig{0.1}{tree1} + \zeta \mathfig{0.1}{tree2} + \zeta^2 \mathfig{0.1}{tree3} + \zeta^3 \mathfig{0.1}{tree4} + \zeta^4 \mathfig{0.1}{tree5} = 0 \\
\label{eq:free2}
\mathfig{0.1}{tree1} + \zeta^{-1} \mathfig{0.1}{tree2} + \zeta^{-2} \mathfig{0.1}{tree3} + \zeta^{-3} \mathfig{0.1}{tree4} + \zeta^{-4} \mathfig{0.1}{tree5} = 0
\end{align}
with $\zeta^5 = 1$ and $t+\zeta^2 +\zeta^3 = 0$ (so if $t = \frac{1+\sqrt{5}}{2}$, $\zeta = \exp(\pm 2\pi i/5)$, and if $t = \frac{1-\sqrt{5}}{2}$, $\zeta=\exp(\pm 4 \pi i /5)$). 
\end{lem}
\begin{proof}
By non-degeneracy and $D^\square(5,1)$ spanning, anything in the kernel of $M(5,1)$ must be a relation.
When $P_{ABA} = t^2-t-1=0$, we obtain the relations above. (See the calculation in {\tt code/ABA.nb}.)
\end{proof}

Now we turn our attention to the case where $P_{G_2}=0$. The following Lemma is well-known \cite{MR1265145}.

\begin{lem}
The curve $P_{G_2}=0$ is rational and can be parameterized by
\begin{align*}
d & =x^5+x^4-5x^3-4x^2+6x+3, \\ 
\intertext{and}
t & = - \frac{x-1}{x^2-2}
\end{align*}
where $x \neq \pm \sqrt{2}$.
\end{lem}

It is more usual to change variables so that $x = q^2+q^{-2}$ in order to relate 
this to quantum groups with the usual variables.  Note that this change of 
variables is typically $4$-to-$1$ with $\pm q^{\pm 1}$ all corresponding to the 
same pair $(d,t)$.  In these variables we have 
\begin{align*}
d & =q^{10} + q^8 + q^2 + 1 + q^{-2} + q^{-8} + q^{-10}, \\ 
\intertext{and}
t & = - \frac{q^2-1+q^{-2}}{q^4+q^{-4}},
\end{align*}
 where $q$ is not a primitive $16$th root of unity.  Recall that we also have 
 that $q$ is not a primitive $3$rd or $6$th root of unity since $d+t+d t \neq 0$.

\begin{lem}
\label{lem:G2-relation}
In a cubic category where $D^\square(5,1)$ spans $\cC_5$ and $P_{G_2}=0$ there is a relation:
$$
\ngon[90]{5} = \alpha \left(\mathfig{0.1}{tree1} + \text{rotations}\right) +\beta \left(\mathfig{0.1}{forest1} + \text{rotations}\right),
$$ 
where
\begin{align*}
\alpha &= - \frac{1}{(q^2+1+q^{-2})(q^4+q^{-4})} \displaybreak[1]  \\
\beta &=- \frac{1}{(q^2+1+q^{-2})^2(q^4+q^{-4})^2} 
\end{align*}

\end{lem}
 \begin{proof}
When $P_{G_2}=0$  these relations are in the radical of the inner product on $M^\square(5,1)$. (See the calculation in {\tt code/G2.nb}.)
\end{proof}

We now prove the proposition. Suppose $D^\square(5,1)$ is dependent. Then Lemma
\ref{lem:5:dependent=>spans} shows that $D(5,0)$ spans, and hence trivially
$D^\square(5,1)$ spans also. On the other hand, if $D^\square(5,1)$ is
independent, it also must span, since we are assuming $\dim \cC_5 \leq 11$. Then
Lemma \ref{lem:ABA-relation} and Lemma \ref{lem:G2-relation} ensure that there
are relations amongst the diagrams in $D^\square(5,1)$.  By Corollary
\ref{cor:5:evaluation}  there is a unique cubic category at the given values of
$d$ and $t$.
\qed

\subsection*{Proof of Proposition \ref{prop:5:realization:ABA} (Realization)}
Note that for any $(d,t)$ with $P_{ABA} = t^2-t-1 = 0$ we can rewrite $(d,t) =
(\delta^2 \tau, \bar{\tau})$ where bar is the Galois conjugate $(1+\sqrt{5})/2
\leftrightarrow (1-\sqrt{5})/2$, by taking $\tau$ to be the Galois conjugate
of $t$ and $\delta = \sqrt{d \tau^{-1}}$. This change of variables is generally
2-to-1, and there is a symmetry $\delta \leftrightarrow -\delta$.

Let $TL_\delta$ be the Temperley-Lieb category of planar tangles with loop value
$\delta$, and let $\cG_{\tau}$ be the golden category with loop value $\tau$ (so
that the triangle value will be $\bar{\tau}$). If $\cC$ and $\cD$ are two
pivotal categories, then their free product $\cC * \cD$ consists of planar diagrams with
connected components labelled blue and red, where the blue parts live in $\cC$
and the red parts live in $\cD$ (cf. \cite{1308.5723} and \cite[Section 8]{MR1950890}). We've shown all blue components in what follows
with dashed lines. Consider the free product $TL_\delta \ast \cG_{\tau}$.  This
is a category of planar diagrams with connected components labelled blue and
red, where the blue strands have no vertices and have loop value $\delta$, while
red strands allow trivalent vertices and have loop value $\tau$.

Inside the free product $TL_\delta \ast \cG_{\tau}$ we can find a trivalent
category, which we call $ABA_{(\delta^2 \tau, \bar{\tau})}$ as follows.  Given a
trivalent graph interpret each strand as a red strand then a blue strand then a
red strand (hence the name $ABA$), and think of each trivalent vertex as given
by
$$
\frac{1}{\sqrt{\delta}}
\begin{tikzpicture}[baseline]
\draw[red] (0,0) -- (90:1);
\draw[red] (0,0) -- (210:1);
\draw[red] (0,0) -- (330:1);
\draw[blue,dashed] (100:1) to[out=-90,in=30] (200:1);
\draw[blue,dashed] (220:1) to[out=30,in=150] (320:1);
\draw[blue,dashed] (340:1) to[out=150,in=-90] (80:1);
\end{tikzpicture}.
$$
Here the normalization factor ensures that the bigon factor $b$ is $1$.  We 
call this category $ABA'_{(d,t)}$ and then define $ABA_{(d,t)}$ to be the 
quotient by the negligibles.

\begin{remark} \label{rem:generic}
As far as we know, it may be that $ABA'_{(d,t)}$ has no negligbles and so 
$ABA'_{(d,t)} = ABA_{(d,t)}$.  Indeed, when $dt^{-1}$ is a positive number 
bigger than $4$ the pairing is positive definite because it is the restriction 
of an obviously positive definite pairing on the tensor product.  Hence when 
$d$ is generic, $ABA'_{(d,t)}$ has no negligbles.
\end{remark}

We can now prove the proposition.  We first find a spanning set for the
$n$-boundary point space for $ABA$. First note that the blue (A-labelled) lines
of the $ABA$ diagram give a noncrossing partition of the red boundary points,
where two red boundary points are in the same partition if you can get between
them without crossing a blue line.  Second, recall that the dimension of the
space of $m$-boundary point red diagrams is $F_{m-1}$ (the $m-1$st Fibonacci
number).  There is a standard explicit basis given by fixing a trivalent tree
connecting all the boundary vertices and then picking a subset of internal edges to
delete, with the condition that each vertex have either two or three edges
coming out of it.

Putting the above two steps together, for each noncrossing partition (specifying
the location of the blue lines), we can find a spanning set for diagrams
compatible with that partition consisting of $\prod_{p \in \pi} F_{| p |-1}$
diagrams (where $\pi$ is a partition, $p$ ranges over parts of the partition,
and $|p|$ is the size of the part).  This process gives spanning sets for the
$n$-boundary point spaces for $0 \leq n \leq 5$ with sizes $1,0,1,1,4,8$.  We
illustrate this below by giving the spanning set for $n=5$.

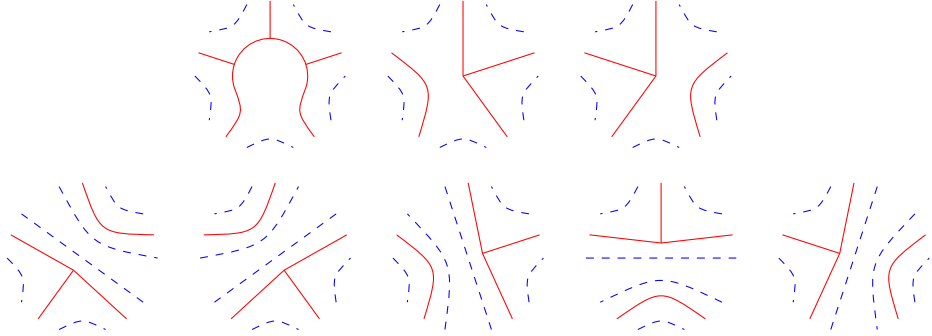
\begin{figure}[ht]
\begin{tikzpicture}[rotate=-72]
	\draw[blue,dashed] (36:1cm) .. controls (54:.8cm) .. (72:1cm);
	\draw[blue,dashed, rotate=72] (36:1cm) .. controls (54:.8cm) .. (72:1cm);
	\draw[blue,dashed, rotate=144] (36:1cm) .. controls (54:.8cm) .. (72:1cm);
	\draw[blue,dashed, rotate=216] (36:1cm) .. controls (54:.8cm) .. (72:1cm);
	\draw[blue,dashed, rotate=288] (36:1cm) .. controls (54:.8cm) .. (72:1cm);

	\draw[red] (18:1cm) .. controls (18:.6cm) .. (54:.5cm) arc (54:270:.5cm) .. controls (-54:.6cm) .. (-54:1cm);
	\draw[red] (90:1cm) -- (90:.5cm);
	\draw[red] (162:1cm) -- (162:.5cm);
	\draw[red] (234:1cm) -- (234:.5cm);
\end{tikzpicture}
\quad
\begin{tikzpicture}[rotate=-72]
	\draw[blue,dashed] (36:1cm) .. controls (54:.8cm) .. (72:1cm);
	\draw[blue,dashed, rotate=72] (36:1cm) .. controls (54:.8cm) .. (72:1cm);
	\draw[blue,dashed, rotate=144] (36:1cm) .. controls (54:.8cm) .. (72:1cm);
	\draw[blue,dashed, rotate=216] (36:1cm) .. controls (54:.8cm) .. (72:1cm);
	\draw[blue,dashed, rotate=288] (36:1cm) .. controls (54:.8cm) .. (72:1cm);

	\draw[red] (18:1cm) -- (0,0) -- (162:1cm);
	\draw[red] (90:1cm) -- (0,0);
	\draw[red] (234:1cm) .. controls (270:.4cm).. (306:1cm);
\end{tikzpicture}
\quad
\begin{tikzpicture}[rotate=-72]
	\draw[blue,dashed] (36:1cm) .. controls (54:.8cm) .. (72:1cm);
	\draw[blue,dashed, rotate=72] (36:1cm) .. controls (54:.8cm) .. (72:1cm);
	\draw[blue,dashed, rotate=144] (36:1cm) .. controls (54:.8cm) .. (72:1cm);
	\draw[blue,dashed, rotate=216] (36:1cm) .. controls (54:.8cm) .. (72:1cm);
	\draw[blue,dashed, rotate=288] (36:1cm) .. controls (54:.8cm) .. (72:1cm);

	\draw[red] (-54:1cm) -- (0,0) -- (162:1cm);
	\draw[red] (234:1cm) -- (0,0);
	\draw[red] (18:1cm) .. controls (54:.4cm).. (90:1cm);
\end{tikzpicture}
\vspace*{12pt}

\begin{tikzpicture}[rotate=0]
	\draw[blue,dashed] (36:1cm) .. controls (54:.8cm) .. (72:1cm);
	\draw[blue,dashed] (108:1cm) .. controls (54:.2cm) .. (0:1cm);
	\draw[blue,dashed] (180:1cm) .. controls (198:.8cm) .. (216:1cm);
	\draw[blue,dashed] (252:1cm) .. controls (270:.8cm) .. (288:1cm);
	\draw[blue,dashed] (324:1cm) .. controls (0,0) .. (144:1cm);

	\draw[red] (-54:1cm) -- (234:.2cm) -- (162:1cm);
	\draw[red] (234:1cm) -- (234:.2cm);
	\draw[red] (18:1cm) .. controls (54:.4cm).. (90:1cm);
\end{tikzpicture}
\quad
\begin{tikzpicture}[rotate=72]
	\draw[blue,dashed] (36:1cm) .. controls (54:.8cm) .. (72:1cm);
	\draw[blue,dashed] (108:1cm) .. controls (54:.2cm) .. (0:1cm);
	\draw[blue,dashed] (180:1cm) .. controls (198:.8cm) .. (216:1cm);
	\draw[blue,dashed] (252:1cm) .. controls (270:.8cm) .. (288:1cm);
	\draw[blue,dashed] (324:1cm) .. controls (0,0) .. (144:1cm);

	\draw[red] (-54:1cm) -- (234:.2cm) -- (162:1cm);
	\draw[red] (234:1cm) -- (234:.2cm);
	\draw[red] (18:1cm) .. controls (54:.4cm).. (90:1cm);
\end{tikzpicture}
\quad
\begin{tikzpicture}[rotate=144]
	\draw[blue,dashed] (36:1cm) .. controls (54:.8cm) .. (72:1cm);
	\draw[blue,dashed] (108:1cm) .. controls (54:.2cm) .. (0:1cm);
	\draw[blue,dashed] (180:1cm) .. controls (198:.8cm) .. (216:1cm);
	\draw[blue,dashed] (252:1cm) .. controls (270:.8cm) .. (288:1cm);
	\draw[blue,dashed] (324:1cm) .. controls (0,0) .. (144:1cm);

	\draw[red] (-54:1cm) -- (234:.2cm) -- (162:1cm);
	\draw[red] (234:1cm) -- (234:.2cm);
	\draw[red] (18:1cm) .. controls (54:.4cm).. (90:1cm);
\end{tikzpicture}
\quad
\begin{tikzpicture}[rotate=216]
	\draw[blue,dashed] (36:1cm) .. controls (54:.8cm) .. (72:1cm);
	\draw[blue,dashed] (108:1cm) .. controls (54:.2cm) .. (0:1cm);
	\draw[blue,dashed] (180:1cm) .. controls (198:.8cm) .. (216:1cm);
	\draw[blue,dashed] (252:1cm) .. controls (270:.8cm) .. (288:1cm);
	\draw[blue,dashed] (324:1cm) .. controls (0,0) .. (144:1cm);

	\draw[red] (-54:1cm) -- (234:.2cm) -- (162:1cm);
	\draw[red] (234:1cm) -- (234:.2cm);
	\draw[red] (18:1cm) .. controls (54:.4cm).. (90:1cm);
\end{tikzpicture}
\quad
\begin{tikzpicture}[rotate=288]
	\draw[blue,dashed] (36:1cm) .. controls (54:.8cm) .. (72:1cm);
	\draw[blue,dashed] (108:1cm) .. controls (54:.2cm) .. (0:1cm);
	\draw[blue,dashed] (180:1cm) .. controls (198:.8cm) .. (216:1cm);
	\draw[blue,dashed] (252:1cm) .. controls (270:.8cm) .. (288:1cm);
	\draw[blue,dashed] (324:1cm) .. controls (0,0) .. (144:1cm);

	\draw[red] (-54:1cm) -- (234:.2cm) -- (162:1cm);
	\draw[red] (234:1cm) -- (234:.2cm);
	\draw[red] (18:1cm) .. controls (54:.4cm).. (90:1cm);
\end{tikzpicture}
\end{figure}

By computing inner products we see that these spanning sets actually form bases
unless $P_{SO(3)} = 0$ (in which case the nondegenerate quotient is just a
golden category).  Hence the ABA categories are cubic categories with $\dim
\cC_5 \leq 11$.

Finally, it is straighforward to  check that the loop value is $\delta^2 \tau = d$.  We see that  the triangle value is $\bar{\tau}$:
\begin{align*}
\ngon[90]{3}
& = 
\left(\frac{1}{\sqrt{\delta}}\right)^3
\begin{tikzpicture}[baseline=0]
	\draw[red] (90:1cm) -- (90:.7cm) -- (-30:.7cm) -- (-30:1cm);
	\draw[red] (-150:1cm) -- (-150:.7cm) -- (-30:.7cm);
	\draw[red] (90:.7cm)--(-150:.7cm);
	\draw[blue,dashed] (-20:1cm) .. controls (30:.5cm) .. (80:1cm);
	\draw[blue,dashed,rotate=120] (-20:1cm) .. controls (30:.5cm) .. (80:1cm);
	\draw[blue,dashed,rotate=-120] (-20:1cm) .. controls (30:.5cm) .. (80:1cm);
	\draw[blue,dashed] (0,0) circle (.2cm);
\end{tikzpicture}
=
\left(\frac{1}{\sqrt{\delta}}\right)^3 \delta
 \begin{tikzpicture}[baseline=0]
	\draw[red] (90:1cm) -- (90:.7cm) -- (-30:.7cm) -- (-30:1cm);
	\draw[red] (-150:1cm) -- (-150:.7cm) -- (-30:.7cm);
	\draw[red] (90:.7cm)--(-150:.7cm);
	\draw[blue,dashed] (-20:1cm) .. controls (30:.5cm) .. (80:1cm);
	\draw[blue,dashed,rotate=120] (-20:1cm) .. controls (30:.5cm) .. (80:1cm);
	\draw[blue,dashed,rotate=-120] (-20:1cm) .. controls (30:.5cm) .. (80:1cm);
\end{tikzpicture} \displaybreak[1]  \\
& =
\bar{\tau} \cdot
\frac{1}{\sqrt{\delta}}
\begin{tikzpicture}[baseline]
\draw[red] (0,0) -- (90:1);
\draw[red] (0,0) -- (210:1);
\draw[red] (0,0) -- (330:1);
\draw[blue,dashed] (100:1) to[out=-90,in=30] (200:1);
\draw[blue,dashed] (220:1) to[out=30,in=150] (320:1);
\draw[blue,dashed] (340:1) to[out=150,in=-90] (80:1);
\end{tikzpicture}
= \bar{\tau} \cdot
\begin{tikzpicture}[baseline=.1cm,scale=0.75]
\draw (0,0) -- (0,1);
\draw (0,0) -- (0.7,-0.5);
\draw (0,0) -- (-0.7,-0.5);
\end{tikzpicture}.
\end{align*}
This completes the proof of Proposition \ref{prop:5:realization:ABA}. \qed

\begin{remark}
By Remark \ref{rem:generic}, when $d$ is generic there are no further relations,
and so the dimension of the $n$-boundary point space is given by $\sum_\pi
\prod_{p \in \pi} F_{| p |-1}$.  This sequence begins $1,0,1,1,4,8,25,64,\ldots$
and its ordinary generating function satisfies the relation $G(x) =
\frac{1-xG(x)}{1-xG(x)-x^2G(x)^2}$.  Therefore it is given by OEIS A046736
\cite{EIS}, which counts the number of ways to place non-intersecting diagonals
on a $n+2$-gon so as to create no triangles.  To find this generating function
identity, we use a general recipe due to Speicher \cite{MR1268597} for any such
weighted counting of non-crossing partitions.  (This approach thus applies to
the pivotal tensor category generated by $ABA$ in the free product of Temperley-
Lieb with an arbitrary trivalent category.)  Let $a_i$ be an arbitrary sequence
with $a_0 = 1$, let $b_n = \sum_\pi \prod_{p \in \pi} a_{| p |}$, and let $A(x)$
and $B(x)$ be their ordinary generating functions.  Then rewriting
\cite[Exercise 5.35]{MR1676282} yields $B(x) = A(xB(x))$.  In our particular
example, we use that the generating function for the shifted Fibonacci sequence
is $\frac{1-x}{1-x-x^2}$.
\end{remark}

\begin{remark}
It is not difficult to work out the simple objects in the ABA categories.  They
are of the form $A^{(n_1)}BA^{(n_2)}BA^{(n_3)}B\ldots B A^{(n_k)}$ where
$A^{(n)}$ denotes the $n$th Jones--Wenzl made of blue strands (so for generic
$\delta$ you allow the $n_i$ to be any positive number, while for $\delta =
\zeta+\zeta^{-1}$ there is a corresponding bound on $n_i$).  The fusion rules
are given by concatenation and applying the usual $SU(2)$ fusion rules for blue
Jones-Wenzl's and $B^2 = B+1$ for red strands.   So, for example, $(ABA) (ABA) =
ABA^{(2)}BA + ABA+A^{(2)}+1$.
\end{remark}

\subsection*{Proof of Proposition \ref{prop:5:realization:G2} (Realization)}

We recall Kuperberg's skein theoretic description of the quantum $G_2$ spider
categories \cite{MR1403861, MR1265145} (warning, there is a sign error in
\cite{MR1403861}).  We change conventions in two ways:  Kuperberg's $q$ is our
$q^2$ (which agrees with the usual quantum group conventions) and we normalize
the trivalent vertex so that the bigon equals the strand (which is possible so
long as $q$ is not a primitive $3$rd or $6$th or $16$th root of unity).

\begin{defn} 
\label{def:G2}
If $q$ is not a primitive $3$rd, $6$th, or $16$th root of unity, let $(G_2)_q'$
be the pivotal category generated by a trivalent vertex, modulo the following
skein relations, and let $(G_2)_q$ be the nondegenerate quotient of $(G_2)_q'$
by its negligible ideal.
\begin{align*}
\bigcirc &=  \Phi_7 \Phi_{14} \Phi_{24}  = q^{10} + q^8 + q^2 + 1 + q^{-2} + q^{-8} + q^{-10} \displaybreak[1] \\
\begin{tikzpicture}[scale=.5, baseline=0]
	\draw (0,.5) circle (.5cm);
	\draw (0,0)--(0,-1);
\end{tikzpicture}
 &= 0 \displaybreak[1] \\
\ngon[90]{2} &= \begin{tikzpicture}[scale=.5, baseline=0]
	\draw (0,1)--(0,-1);
\end{tikzpicture}
\displaybreak[1] \\
\ngon[90]{3} &= - \frac{\Phi_{12}}{\Phi_{16}} \begin{tikzpicture}[scale=.5,baseline=0]
	\draw (90:1cm)--(0,0)--(-30:1cm);
	\draw (0,0) -- (-150:1cm);
\end{tikzpicture}
 = - \frac{q^2-1+q^{-2}}{q^4+q^{-4}} \begin{tikzpicture}[scale=.5,baseline=0]
	\draw (90:1cm)--(0,0)--(-30:1cm);
	\draw (0,0) -- (-150:1cm);
\end{tikzpicture}
 \displaybreak[1] \\
\ngon[45]{4} &=  \frac{\Phi_8}{\Phi_3 \Phi_6 \Phi_{16}} \left(\;\drawI \; +\; \drawH \; \right) + \frac{1}{\Phi_3 \Phi_6 \Phi_{16}^2} \left(\; \cupcap \; + \; \identity \; \right) \displaybreak[1] \\
\ngon[90]{5} &= - \frac{1}{\Phi_3 \Phi_6 \Phi_{16}} \left(\mathfig{0.1}{tree1} + \mathfig{0.1}{tree2} + \mathfig{0.1}{tree3} + \mathfig{0.1}{tree4} + \mathfig{0.1}{tree5} \right) \\
& \qquad -\frac{1}{\Phi_3^2 \Phi_6^2 \Phi_{16}^2}  \left(\mathfig{0.1}{forest1} + \mathfig{0.1}{forest2} + \mathfig{0.1}{forest3} + \mathfig{0.1}{forest4} + \mathfig{0.1}{forest5}\right),
\end{align*}
where $\Phi_k$ is the $k$th symmetrized cyclotomic polynomial.  That is,  $\Phi_k = \prod_\zeta \left(q^{\frac{1}{2}}-\zeta q^{-\frac{1}{2}}\right)$ where the product is taken over all primitive $k$th roots of unity.  Explicitly, 
\begin{align*}
\Phi_3 & = q+1+q^{-1} \displaybreak[1] \\
\Phi_6 & = q-1+q^{-1} \displaybreak[1]  \\
\Phi_7 & = q^3 + q^2 + q + 1 + q^{-1} + q^{-2} + q^{-3} \displaybreak[1]  \\
\Phi_8 & = q^2+q^{-2} \displaybreak[1]  \\
\Phi_{12} & = q^2-1+q^{-2} \displaybreak[1]  \\
\Phi_{14} & = q^3 - q^2 + q - 1 + q^{-1} - q^{-2} + q^{-3} \displaybreak[1]  \\
\Phi_{16} & = q^4+q^{-4} \displaybreak[1]  \\
\Phi_{24} & = q^4-1+q^{-4}.
\end{align*}
\end{defn}

Now, suppose that in addition, $q$ is not a primitive $7$th, $14$th, or $24$th
root of unity.  We want to show that $(G_2)_q$ is a trivalent category.  We've
already seen in Corollary \ref{cor:nopents} that for $n \leq 5$ the $n$-boundary
point space of $(G_2)_q$ is spanned by diagrams in $D(n,0)$ and hence has
dimensions bounded by $1,0,1,1,4,10$.  However, a priori these relations might
collapse everything.

In \cite{MR2308953}, Sikora and Westbury introduce the notion of confluence, and
claim that the above relations are confluent (we have verified this
calculation). By definition, confluence means that if we start with a graph and
then use one of the above relations to simplify one face or another face, then
we can apply more simplifications on faces until both expressions become equal.
By the Diamond Lemma \cite{MR0007372} this shows that any two reductions give
the same answer.  This means that the inner product of diagrams is well-defined,
and then taking inner products lets us prove that the obvious spanning set is a
basis for $n \leq 4$, except when we are also on the $SO(3)$ curve (see the
following remark).  In particular, $(G_2)_q$ is a cubic category. Finally we
observe that the formulas for $d$ and $t$ agree with the ones in the
parameterization of $P_{G_2} = 0$. \qed

\begin{lem}
If $q$ is not a $3$rd, $6$th, $16$th, $7$th, $14$th, or $24$th root of unity, then the dimension of the $n$-boundary point spaces of $(G_2)_q$ are $1,0,1,1,4,10$ unless $q$ is a primitive $20$th root of unity in which case they $1,0,1,1,4,9$.
\end{lem}
\begin{proof}
Unless $(d,t)$ lies on the $Q_{1,2}$ curve, the $10$ diagrams in $D(5,0)$ are linearly independent.  The only points on the intersection of the $Q_{1,2}$ and $G_2$ curves are $(-2,-2)$, $(-1,-1)$, $(2,0)$, $(\tau, \bar{\tau})$, and $(\bar{\tau},\tau)$.  These correspond, respectively, to a primitive $3$rd or $6$th root of unity, a primitive $20$th root of unity, a primitive $12$th root of unity, and a primitive $30$th root of unity.  The only case not excluded by our assumptions is $q$ a primitive $20$th root of unity.  Calculating the determinant of a $9$-by-$9$ minor of $M(5,0)$ shows that the $4$-boundary point space is $9$ dimensional at this value.
\end{proof}

\begin{remark}
\label{rem:G2-i}
When $q = \pm i$ is a primitive $4$th root of unity $(G_2)_q$ can still be defined by the above relations and confluence argument.  Since we have an explicit spanning
set for the $4$-boundary point space, a direct calculation shows that the
relation  $$  \drawH \; - \; \drawI \; - \frac{1}{2}\; \twostrandid \; +
\frac{1}{2}\; \cupcap$$ does lie in the radical of the inner product, so in the
nondegenerate quotient $(G_2)_{\pm i}$ the dimension of the $4$-box space drops
from $4$ down to $3$.  This gives an alternate proof that the point $(d,t) = (-1,3/2)$ on $P_{SO(3)}$ can be realized.
\end{remark}

\begin{remark}
When $q$ is one of the bad values (a primitive $3$rd, $6$th, or $16$th root of unity) the above definition can be modified by normalizing the trivalent vertex differently to give a well-defined category.  However, in this category the value of the bigon will be $0$ and so the trivalent vertex will be zero.
\end{remark}

\begin{remark}
The category $(G_2)_q$ gets its name from its relationship to the quantum group $U_q(\mathfrak{g}_2)$.  If $q$ is generic, then the category of maps between tensor powers of the standard $7$-dimensional representation of $U_q(\mathfrak{g}_2)$ are given by the above diagrams.  In fact, by the results of this section it is clear that the subcategory of $\mathrm{Rep}(U_q(\mathfrak{g}_2))$ generated by the trivalent vertex must be $(G_2)_q$ and then Kuperberg showed that the dimensions match up so the subcategory is the whole category. 

When $q$ is a root of unity, the correct algebraic category is the category of tilting modules. There is a map $(G_2)_q' \to \mathrm{Rep}^{\mathrm{tilting}}(U_q(\mathfrak{g}_2))$, but it is not clear whether this is surjective, or if it descends to a map from the nondegenerate quotient $(G_2)_q$ to the non-degenerate quotient of the category of tilting modules.
\end{remark}

\section{Diagrams with six boundary points} \label{sec:six}

We now move on to the diagrams with six boundary points. We have
\begin{align*}
D(6,0) & = \left\{ 
\mathfig{0.1}{urn_sha1_1efb25f6833fbe664615559abab36b703e35cb24} + \text{1 rotation},
\mathfig{0.1}{urn_sha1_448b57b212c25633e803562e5ab6433c11d56dd8} + \text{2 rotations}, 
\mathfig{0.1}{urn_sha1_df91ef8b1687dae1b2d9fe38719e6812d7201d49} + \text{5 rotations},
\right. \displaybreak[1] \\
& \qquad
\mathfig{0.1}{urn_sha1_a0b22b5762fa24ead060f3de9448c681b424e133} + \text{5 rotations},
\mathfig{0.1}{urn_sha1_9f7ce5c387d4ae2c3b7e9317f261376c005a92c9} + \text{5 rotations},
\mathfig{0.1}{urn_sha1_1baad19a7b800a28edd9aa033244492c73330c5b} + \text{2 rotations}, \displaybreak[1] \\
& \qquad 
\left.
\mathfig{0.1}{urn_sha1_20aca0123680da1b425e3ab29d0fb3d3ceada0da} + \text{2 rotations},
\mathfig{0.1}{urn_sha1_b91a68952d267c8269410143311be8b4ce67a732} + \text{2 rotations},
\mathfig{0.1}{urn_sha1_f357e0b86229ce63736f2ab08989fc2a2f1db2f6} + \text{1 rotation}
\right\}
\displaybreak[1] \\
D^\square(6,1) & \setminus D(6,0) = 
\left\{
\mathfig{0.1}{urn_sha1_36f2a66789ebdceb39fd251941990993ef82b625} + \text{5 rotations},
\mathfig{0.1}{urn_sha1_7de6a0854fa495f044e12f0708481932e9cb121e}
\right\}
\displaybreak[1] \\
D^\square(6,2) & \setminus D^\square(6,1) = 
\left\{
\mathfig{0.1}{urn_sha1_55f53ebaabfb09b3d4727e4b45c3637b0ed8f84c.pdf} + \text{2 rotations}
\right\}
\end{align*}
so $\#D(6,0) = 34, \#D^\square(6,1) = 41$, and $\#D^\square(6,2) = 44$.

In this section we prove
\begin{thm}
\label{thm:6}
If $\cC$ is a trivalent category with $\dim \cC_4 = 4$, $\dim \cC_5 = 11$, and 
$\dim \cC_6 \leq 40$ then $d^2-3d-1=0$, $t=-\frac{2}{3}d+\frac{5}{3}$, and 
$\cC$ is the $H3$ fusion category constructed by Grossman and Snyder \cite{MR2909758} 
(which is Morita equivalent to the even parts of the Haagerup subfactor \cite{MR1686551}) 
or its Galois conjugate.
\end{thm}

This theorem follows from four propositions.

\begin{prop}[Non-existence] \label{prop:6:preliminary-nonexistence}
There are no trivalent categories with $\dim \cC_4
= 4$, $\dim \cC_5 = 11$, and $D(6,0)$ linearly dependent.
\end{prop}

\begin{prop}[Non-existence] \label{prop:6:nonexistence}
For any $(d,t)$ not satisfying $d^2-3d-1=0$ and $t=-\frac{2}{3}d+\frac{5}{3}$
 there are no trivalent categories with $\dim \cC_4 = 4$,  $\dim \cC_5 = 11$, and  $\dim \Span D^\square(6,2) \leq 40$.
\end{prop}

\begin{prop}[Uniqueness] \label{prop:6:uniqueness}
For each pair $(d,t)$ satisfying $d^2-3d-1=0$ and $t=-\frac{2}{3}d+\frac{5}{3}$, 
there is at most one trivalent category with $\dim \cC_4 = 4$, with 
$\dim \cC_5 = 11$, and with $\dim \cC_6 \leq 40$.
\end{prop}

\begin{prop}[Realization]
\label{prop:6:realization}
The $H3$ fusion category and its Galois conjugate are trivalent and have 
$\dim \cC_4 = 4$, $\dim \cC_5 = 11$, and $\dim \cC_6 = 37 \leq 40$. 
\end{prop}
These two categories exhaust the possibilities allowed by the first three propositions.

\subsection*{Proof of Proposition \ref{prop:6:preliminary-nonexistence} (Non-existence)}

We have the following values of determinants.

\begin{fact}
\label{fact:Delta6_0}
In any cubic category,
\begin{align*}
\Delta(6,0) = -d^{34} Q_{1,1}^{-8} Q_{0,1}^2 Q_{0,2}^9 Q_{2,4,a} Q_{3,5}^2 Q_{6,9} P_{SO(3)}^{19}.
\end{align*}
\end{fact}

\begin{fact}
\label{fact:Delta6_1}
In any cubic category, 
\begin{align*} 
\Delta^\square(6,1) = d^{41} Q_{1,1}^{-29} Q_{0,2}^{16} Q_{2,3} Q_{3,4}^2 Q_{7,11} P_{SO(3)}^{27} P_{G_2}^6.
\end{align*}
\end{fact}

\begin{fact}
\label{fact:Delta6_2}
In any cubic category,
\begin{align*}
\Delta^\square(6,2) = d^{44} Q_{1,1}^{-44} Q_{0,2}^{19} Q_{2,3} Q_{4,5}^2 Q_{8,12} P_{SO(3)}^{33}  P_{G_2}^9
\end{align*}
\end{fact}

\begin{fact}
\label{fact:Delta7_0}
In any cubic category,
\begin{align*}
\Delta(7,0) & = -d^{112} Q_{1,1}^{-70} Q_{0,2}^{48} Q_{11,19} Q_{36,60}^2 P_{SO(3)}^{76}
 Q_{2,4,b}
\end{align*}
\end{fact}

If $D(6,0)$ is dependent, then $\Delta(6,0)$ must vanish, and indeed
$\Delta^\square(6,1), \Delta^\square(6,2)$ and $\Delta(7,0)$ must vanish also.
This only happens at finitely many points, all of which are on the $G_2$ or
$SO(3)$ curves. (This calculation is curious; finding intersections of
$\Delta(6,0)$ with the other varieties appears to be rather hard. However the
Gr\"obner basis calculation showing $\Delta^\square(6,1)$ and
$\Delta^\square(6,2)$ intersect at finitely many points besides $Q_{0,2} Q_{2,3}
P_{SO(3)} P_{G_2} = 0$ is quite manageable, and after that we can easily find
the complete intersection.) This calculation can be found in the file {\tt
code/GroebnerBasisCalculations.nb} available with the {\tt arXiv} source of this
article. Now, using the full strength of Proposition \ref{prop:5:uniqueness}, we
see that any trivalent category with $\cC_5 \leq 11$ at one of these points must
actually be an $ABA$ or $(G_2)_q$ category, contradicting our assumption that
$\dim \cC_5 = 11$.
\qed

\subsection*{Proof of Proposition \ref{prop:6:nonexistence} (Non-existence)}

Beyond the determinant calculations in the previous section, with high probability we have the following two determinants.

\begin{conj}
\label{conj:Delta7_1}
\label{conj:Delta7_2}
In any cubic category,
\begin{align*}
\Delta^\square(7,1) & = - d^{155} Q_{1,1}^{-242} Q_{0,2}^{91} Q_{2,3}^7
 Q_{21,33} Q_{51,69}^2  P_{SO(3)}^{133} P_{G_2}^{35} \\
 \intertext{and}
\Delta^\square(7,2) & = - d^{183} Q_{1,1}^{-403} Q_{0,2}^{119}
Q_{4,5}^{14} Q_{22,36} Q_{54,78}^2  P_{SO(3)}^{189} P_{G_2}^{63}.
\end{align*}
\end{conj}

\begin{lem}
With a probability of $1- 10^{-500}$, each of these conjectures is correct.
\end{lem}
\begin{proof}
We begin by explaining what we mean. The Schwartz-Zippel lemma \cite{MR594695,MR575692,Demillo-Lipton} (pointed out to us by Dylan Thurston) gives
a method of probabilistically checking polynomial identities (we first clear denominators if necessary): if $P \in
k[x_1, \ldots, x_n]$ has total degree bounded by $D$, and the $x_i$ are drawn uniformly and independently from
a finite subset $S
\subset k$
of size $N$, then either $P$ is identically zero or $P(x_1, \ldots, x_n) \neq 0$ with probability  at least $1- \frac
{D}{N}$.

We can easily
bound the
total degree of $\Delta^\square (7,1)$ at 1611, and of $\Delta^\square(7,2)$ at 2664. Evaluating the determinant of
$M^\square(7,1)$ at a pair of values $(d,t)$ drawn uniformly from positive integers at most $10^{40}$ takes on the order
of 3 minutes  (on a 12-core Xeon E5), while the determinant of $M^\square(7,2)$ takes 6 minutes. This gives us a
probability of error of at most
one part in $10^{12}$, for $\Delta^\square(7,1)$, or $10^{6}$, for $\Delta^\square(7,2)$, per minute of running time; we
stopped after reaching $10^{500}$. These checks are implemented in {\tt code/SchwartzZippel.nb}.
\end{proof}

\begin{remark}
To guess these polynomials in the first place, we adopted the following ad-hoc
strategy. Suppose we have some large matrix $M(d,t)$ with entries in $\bbQ(d,
t)$, and want to evaluate the determinant. Arithmetic in $\bbQ(d, t)$ is
difficult, so we avoid this by first specializing one variable, $d$, to various
different primes. We work with a set of primes $\cP$, which is chosen to be big
enough that results we obtain below are successfully verified by the
Schwartz-Zippel lemma argument given above!

At each prime $p \in \cP$ we can compute $\det M(p,t)$ as a rational function in
$\bbQ(t)$ relatively quickly. We now want to recover $M(d,t)$. In our examples,
these are not irreducible, and it turns out to be most efficient
to first factorize each $\det M(p,t)$ into products of powers (possibly
negative) of polynomials in $\bbZ[t]$. For large enough primes $p$, the
factorization is uniform, in the sense that degrees and multiplicities of the
irreducible factors of the different $\det M(p,t)$ are in bijection, and so we
obtain $\det M(p, t) = \prod_{i \in \cI} \cK_{p,i}(t)^{n_i}$, for some fixed index
set $\cI$ and exponents $n_i$ for each $i \in \cI$. Write $\cK_{p,i}(t) = \sum_r
\cL_{p,i,r} t^r$ for some integers $\cL_{p,i,r}$.

We now want to recover an irreducible
polynomial $K_i(d,t) = \sum_r L_{i,r}(d) t^r = \sum_{r,s} L_{i,r,s} d^s t^r$ so $\cK_{p,i}(t) = K_i(p, t)$. This
requires that $L_{i,r}(p) = \cL_{p,i,r}$ for each $p$. In particular, this says
that the constant term $L_{i,r,0}$ of $L_{i,r}$ satisfies $$L_{i,r,0} \equiv L_{i,r}(0) \equiv \cL_{p,i,r}
\pmod{p}.$$ By the Chinese remainder theorem, we then know $L_{i,r,0} \pmod
{\prod_{p \in \cP} p}$, and we guess that it is actually equal to this residue.
We then continue making guesses recursively, using the identities
$$L_{i,r,s} \equiv \left(\cL_{p, i,r} -  \sum_{t=0}^{s-1} L_{i,r,t} p^t \right) p^{-s} \pmod{p}$$
and the Chinese remainder theorem. This method is implemented in the notebook {\tt code/GuessDeterminants.nb}.
\qed
\end{remark}

We will give two separate proofs of Proposition \ref{prop:6:nonexistence}. The first is easy to follow but depends on Conjecture \ref{conj:Delta7_1}, while the second is more difficult but unconditional.

\begin{lem} \label{lem:conditional-proof}
Fix some $(d,t)$ not satisfying $d^2-3d-1=0$ and $t=-\frac{2}{3}d+\frac{5}{3}$.
If Conjecture \ref{conj:Delta7_1} holds at this $(d,t)$ then there are no trivalent categories with $\dim \cC_4 = 4$ and $\dim \cC_5 = 11$, and $\dim \Span D^\square(6,2) \leq 40$. 
\end{lem}
\begin{proof}
Since $\dim \Span D^\square(6,2) \leq 40$ and there are $41$ diagrams in $D^\square(6,1)$, we must have a relation amongst $D^\square(6,1)$.  Thus $(d,t)$ must give a solution to $\Delta^\square(6,1) = \Delta^\square(6,2) = \Delta^\square(7,1) = \Delta^\square(7,2) = 0$,
and the possibilities are (see {\tt code/GroebnerBasisCalculations.nb}):
\begin{enumerate}[(a)]
\item $(d,t)$ is on the $P_{ABA}$ or $P_{G_2}$ curve,
\item $d^2-3d-1=0$ and $t=-\frac{2}{3}d+\frac{5}{3}$,
\item \label{item:smallbadpoint}
   $d$ is a root of the degree 33 polynomial $S_a(d)$, and $t = T_a (d)$, or
\item $d$ is a root of the degree 63 polynomial $S_b(d)$, and $t = T_b(d)$.
\end{enumerate}
The polynomials $S_a, T_a, S_b,$ and $T_b$ (the last of which is stupendously large) are available in
the file {\tt code/BadPoints.nb}.

In the first case, the full strength of Proposition \ref{prop:5:uniqueness}
shows that $\dim \cC_5 < 11$. It remains to eliminate the last two cases. For
both those values of $(d,t)$ the rank of $M^\square(6,2)$ is 43 which is
incompatible with $\dim \Span D^\square(6,2) \leq 40$.  (This calculation is
done twice in the last section of {\tt code/BadPoints.nb}.  We check slowly and
directly that the rank is exactly $43$ by doing arithmetic in the number field,
but we also quickly see that the rank is at least $43$ by calculating the rank
of the matrix modulo a prime in the number field, thereby reducing the question
to calculating the rank over $\mathbb{Z}/11\mathbb{Z}$ and
$\mathbb{Z}/41\mathbb{Z}$ respectively.  The latter approach was suggested to us
by David Roe.)
\end{proof}

Now we turn to the unconditional proof of Proposition \ref{prop:6:nonexistence}, which follows immediately from the following lemma.

\begin{lem}
If $\cC$ is a trivalent category with $\dim \cC_4 = 4$ and $\dim \cC_5 = 11$, and $\dim \Span D^\square(6,2) \leq 40$,
then Conjecture \ref{conj:Delta7_1} holds.
\end{lem}
\begin{proof}
In order for there to be such a trivalent category, we must have that $(d,t)$ is a solution to 
$\Delta^\square(6,1) = \Delta^\square(6,2) = 0$.  We consider  each factor of $\Delta^\square(6,1)$ separately.  
The $P_{G_2}$ and $P_{ABA}$ factors contradict $\dim \cC_5 = 11$ by Proposition \ref{prop:5:uniqueness}.  

On the elliptic curve  $Q_{2,3}$ any rational function can be written in the form $\alpha(t) d + \beta(t)$ for some
rational
functions $\alpha$ and $\beta$, and algebra can be done efficiently on functions of this form just
as it is done with ordinary rational functions.  We can then compute the value of the determinants
$\Delta^\square(7,1)$  and $\Delta^\square(7,2)$ exactly and verify Conjecture \ref{conj:Delta7_1} for these points.
(This calculation is performed in {\tt code/DeterminantsOnEllipticCurve.nb}.)

The other two factors $Q_{3,4}$ and $Q_{7,11}$ of $\Delta^\square(6,1)$ 
intersect $\Delta^\square(6,2) = 0$ in finitely many points (see {\tt code/GroebnerBasisCalculations.nb} for this calculation).
For each of these finitely many points we can compute the
determinants $\Delta^\square(7,1)$  and $\Delta^\square(7,2)$ exactly
(cf. {\tt code/DeterminantsOnExtraPoints.nb}). Thus
Lemma \ref{lem:conditional-proof} applies, and we note that the points described
in item \eqref{item:smallbadpoint} of Lemma \ref{lem:conditional-proof} all lie
on $Q_{2,3}$.
\end{proof}

\subsection*{Proof of 
Proposition \ref{prop:6:uniqueness} (Uniqueness)}

In the proof of Proposition \ref{prop:5:uniqueness} (Uniqueness) we used an
analogue for open graphs of the well-known theorem that any closed planar
trivalent graph has pentagonal or smaller face.  That theorem plays a key role
in the proof of the (easy) $5$-color theorem, and the study and eventual proof of the
$4$-color theorem lead to an enormous number of similar results proved using
the \emph{discharging method}. 
The typical  illustration of the discharging method is the following
well-known lemma.

\begin{defn}\mbox{}
 A \emph{very small} face is a square, triangle, or bigon.
 A \emph{pentapent} is a pair of adjacent pentagons.
 A \emph{hexapent} is a pentagon and an adjacent hexagon.
\end{defn}

\begin{lem} \label{lem:discharging-closed-pentapent}
Any closed trivalent graph contains either a very small face, a pentapent, or 
a hexapent.
\end{lem}
\begin{proof}
Suppose that the graph has no very small faces.  We assign a charge of $6-n$ 
to every $n$-gon face.  By measuring Euler characteristic, the total charge is $12$, 
which is certainly positive.  We now ``discharge" the pentagons, distributing 
their charge equally among the neighboring faces.  Since this does 
not change the total charge of the graph there must be a face with 
positive charge.  Such a face must either be a pentagon next to a pentagon, a 
hexagon next to a pentagon, or a $7$-gon next to at least $6$ pentagons.  In 
the final case, at least two of those six pentagons are adjacent so there's a 
pentapent.
\end{proof}

Just as before, in order to apply this technique to our setting, we need an 
analogue of this lemma for open graphs.

\begin{defn}
A \emph{lonely pentagon} of an open planar trivalent graph is a pentagon 
which touches at most two internal faces.  A lonely pentagon is either a 
corner pentagon or a bridge pentagon. 
A \emph{corner pentagon} is a pentagon which touches at most two 
internal faces which are adjacent to each other, and a \emph{bridge pentagon} is a pentagon which touches exactly two 
internal faces which are not adjacent to each other.
\end{defn}

\begin{lem}[Planar subgraphs]\label{lem:planarsubgraphs}
Every connected open planar trivalent graph has either a very small face, a pentapent, 
a hexapent, a growth region, or a corner pentagon.
\end{lem}
The proof of this Lemma below will use the discharging method following the same outline as in the 
closed case.  We assign the usual charges of $6-n$ to each interior $n$-gon 
face and $4-n$ to each boundary face touching $n$ edges.  By measuring Euler 
characteristic, the total charge is $6$. We then `discharge' the internal 
pentagons, distributing their charge equally among their neighboring internal 
faces.  Since this does not change the total charge of the graph, we go 
looking for faces with positive charge and find that positive charge indicates 
that we have one of the features listed in Lemma \ref{lem:planarsubgraphs}.  
This last step turns out to be somewhat delicate, so we first prove a slightly 
weaker lemma:

\begin{lem}
Every connected open planar trivalent graph has either a very small face, a pentapent, a hexapent, a growth region, or a
lonely pentagon.
\label{lem:weakerplanarsubgraphs}
\end{lem}

\begin{proof}
Suppose a connected open planar trivalent graph $T$ has no growth regions.  
By discharging we will show that $T$ has either a very small face, a pentapent, 
a hexapent, or a lonely pentagon.

Assign a charge as described above.  Since $T$ is connected, a quick argument shows that its Euler characteristic is $6$.
If the boundary has charge 6 or more, it has at least two boundary faces because 
a single boundary face has charge $4-n$ (where $n$ is the number of edges it touches).
Thus, we can divide the boundary into two proper
sub-regions.  One of these will have charge 3 or more, and hence be a growth region by Lemma \ref{lem:boundarycharge}. 

If the boundary has charge less than 6, then
there must be a positive charge among the internal faces. 
We now
have the internal pentagons discharge according to the following rule:  each 
distributes its charge of 1 evenly among the adjacent internal faces.
Suppose there are no lonely pentagons.  Then each pentagon distributes its charge 
among at least 3 faces, and so faces receive at most $\frac{1}{3}$ charge from 
each neighboring pentagon.  The total charge has not changed, so there must be 
some face with positive charge.  

We now consider the ways in which an $n$-gon may end up with positive charge.
If $n \geq 7$, either it neighbors a pair of adjacent pentagons and we are done,
or it neighbors at most $\lfloor \frac{n}{2} \rfloor $ adjacent pentagons.   In
the latter case, the total charge is at most $6-n+\frac{1}{3} \lfloor
\frac{n}{2} \rfloor \leq 0$.  If $n=5$ or $6$, the positively-charged face must
have received some charge from an adjacent pentagon, showing the existence of a
pentapent or hexapent.  Finally, we may have had a very small face all along.
\end{proof}

\begin{proof}[Proof of Lemma \ref{lem:planarsubgraphs}]
Lemma \ref{lem:weakerplanarsubgraphs} is almost what we need, except it proves 
we must have a lonely pentagon, not necessarily a corner pentagon.  

Suppose that a connected open planar trivalent graph $T$ has bridge pentagons.  Each bridge pentagon, when removed, 
disconnects the graph into two parts which each have fewer bridge pentagons.  
By descent, there must be a subgraph which is bridge-pentagon-free and 
connected to the rest of the graph by a single bridge pentagon.  Call this 
subgraph $T'$ and the bridge pentagon connecting it to the rest of the graph 
$B$. 
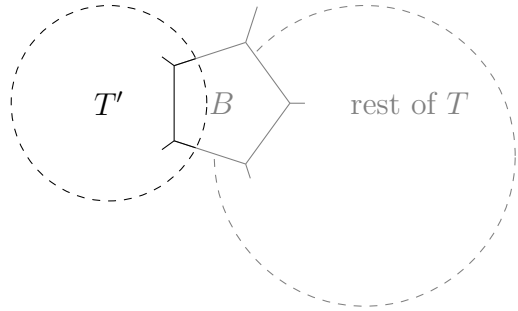
$$
\begin{tikzpicture}[baseline=0]
	\draw[dashed, gray] (3.4,-.7) circle (2cm);
	
	\fill[white] (30:1cm) -- ++(18:1cm) coordinate (X) -- ++(-54:1cm) coordinate (Z) -- ++(-126:1cm) coordinate (Y) -- (-30:1cm)--(30:1cm);

	\draw[gray] (30:1cm) --  (X);
	\draw[gray] (-30:1cm) -- (Y);
	\draw[gray] (X) --  (Z);
	\draw[gray] (Y)--(Z);
	\draw[gray] (X)-- ++(72:.5cm);
	\draw[gray] (Y)-- ++(-72:.2cm);
	\draw[gray] (Z)-- ++(0:.2cm);
	
	\draw (30:1cm)--(-30:1cm);
	\draw (30:1cm) -- ++(18:.3cm);
	\draw (30:1cm) -- ++(144:.2cm);
	\draw(-30:1cm) -- ++(-18:.3cm);
	\draw (-30:1cm) -- ++(-144:.2cm);
	\node (T') at (0,0) {$T'$};
	\node[gray] at (1.5,0) {$B$};	
	\node[gray] (T) at (4,0) {rest of $T$};
	\draw[dashed] (0,0) circle (1.3cm);
\end{tikzpicture}
$$

Observe $T'$ is still connected, by virtue of containing an edge of $B$.
Now consider the boundary of $T'$.  If it has total charge of $6$ or greater, then after $B$ has been attached it still
has charge of $2$ or greater.  To see this, consider the boundary face of $T'$ which $B$ attaches to, and its two
boundary neighbors.  A small neighborhood of these three faces has four outgoing edges, and at least one incoming edge
(otherwise $B$ would be a corner pentagon), hence its charge is at most $4$.  Thus, the boundary complement of this
region (the region enclosed in blue lines below) has charge at least $2$ and so is a growth region by Lemma \ref{lem:boundarycharge}. 

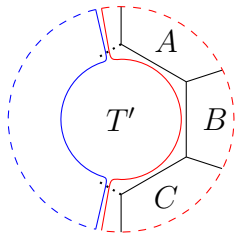
$$
\begin {tikzpicture}[baseline=0] \draw (90:1cm)--(30:1cm)--(-30:1cm) --(-90:1cm);
	\draw[dotted, thick] (-90:1cm) -- ++(150:.35cm);
	\draw[dotted, thick] (90:1cm) -- ++(-150:.35cm);
	\draw (90:1cm)--(90:1.5cm);
	\draw (-90:1cm)--(-90:1.5cm);	
	\draw (30:1cm) -- ++(18:.5cm);
	\draw(-30:1cm) -- ++(-18:.5cm);

	\node (T') at (0,0) {$T'$};
	\node at (1.25,0) {$B$};	
	\node at (60:1.2cm) {$A$};
	\node at (-60:1.2cm) {$C$};
		
	\draw[red] (100:1.5cm) -- (100:1cm) .. controls (100:.8cm) .. (90:.8cm) arc (90:-90:.8cm) .. controls (-100:.8cm) .. (-100:1cm)--(-100:1.5cm);
	\draw[red, dashed] (100:1.5cm) arc (100:-100:1.5cm);
	\draw[blue] (103:1.5cm) -- (103:1cm) .. controls (103:.8cm) .. (110:.8cm) arc (110: 250:.8cm) .. controls (-103:.8cm) .. (-103:1cm) -- (-103:1.5cm);
	\draw[blue, dashed] (103:1.5cm) arc (103:257:1.5cm);
\end{tikzpicture}
$$

Alternately, if the boundary of $T'$ had total charge of $5$ or less, then 
there is a net positive charge and at least one small face in the interior.  
As above, we discharge any pentagons and conclude we must have  a  very small 
face, pentapent, hexapent, or corner pentagon in $T'$ (recall above we ensured that $T'$ had no bridge pentagons).  The very small face, 
pentapent, or hexapent also appears in $T$.  A corner pentagon of $T'$ is 
either also a corner pentagon of $T$, or is adjacent to $B$, hence part of a 
pentapent.    
\end{proof}

\begin{lem}\label{lem:removemaxgrowth}
Let $T$ be a connected open planar trivalent graph with no very small faces, pentapents, hexapents or corner pentagons.  If
we remove a maximal growth region from $T$, the remaining graph (call it $T'$) also has no very small faces, pentapents, hexapents or corner pentagons.
\end{lem}

\begin{proof}
Since any face of $T'$ is a face of $T$, $T'$ has no very small faces, pentapents or hexapents.  If we created a corner pentagon by removing a growth region from $T$, the growth region that we removed was not maximal.  To see why, consider a corner pentagon in $T$ which was not in $T'$.  It must look like the following diagram:
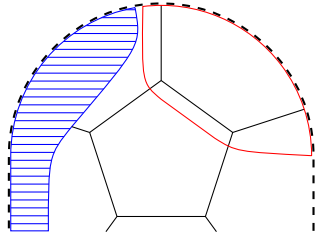
\begin{center}
\begin{tikzpicture}[scale=.5]
	\filldraw[draw=blue, pattern color=blue, pattern =  horizontal lines] (100:4cm) .. controls (100:3cm) .. (162:2.5cm) .. controls (180:3cm) .. (-3,-2) -- (-4,-2) -- (180:4cm) arc (180:100:4cm);
	\draw (18:2cm) -- (90:2cm)--(162:2cm)--(234:2cm)--(306:2cm)--(18:2cm);
	\draw (18:2cm) -- (18:4cm);
	\draw (90:2cm) -- (90:4cm);
	\draw (162:2cm) -- (162:2.5cm);
	\draw (234:2cm) -- (234:2.5cm);
	\draw (306:2cm) -- (306:2.5cm);
	\draw[red] (0:4cm) .. controls (3:2cm) .. (18:1.5cm) -- (90:1.5cm) .. controls (105:2cm) .. (97:4cm) arc (97:0:4cm);
	\draw[thick, dashed] (4.05,-2) -- (4.05,0) arc (0:180:4.05cm) -- (-4.05,-2);
\end{tikzpicture}
\end{center}
Here the dashed curve is the boundary of $T$, the shaded blue region  is the 
growth region, and the graph (perhaps also the growth region) continues below 
the bottom of the picture.   Then the union of the blue region with the region 
enclosed by red is a larger growth region.
\end{proof}

\begin{lem}\label{lem:enumerategraphs}
We can enumerate connected planar trivalent graph with no very small faces, pentapents, or hexapents by starting with the
empty diagram, and
\begin{itemize}
\item sequentially adding growth regions;
\item in the final step, simultaneously adding some ``H"s to create corner pentagons.
\end{itemize}
We can enumerate all planar trivalent graphs (not necessarily connected) with no very small faces, 
pentapents, or hexapents by taking planar disjoint
unions of such connected planar trivalent graphs.
\end{lem}

\begin{proof}
Consider a graph $T$ which has no very small faces, pentapents or hexapents.  
From each corner pentagon, remove an ``H" neighborhood of one of its sides which 
touches an external face.  Since $T$ has no pentapents,  the regions we remove 
will not overlap.

Let $T'$ be the graph with an ``H" removed from each corner pentagon.  $T'$ has 
no corner pentagons, so by repeatedly applying Lemma \ref{lem:removemaxgrowth}, 
we get a sequence of growth regions building $T'$ up from the empty diagram.
\end{proof}

This lemma shows in particular that there are a finite number of such planar trivalent graphs with any given number of
trivalent vertices, or with any given number of boundary points and internal faces.

\begin{cor}
Every planar trivalent graph with no very small faces and at most 6 boundary 
points is in $D^\square(n,1)$ or has an internal pentapent or an internal hexapent.
\end{cor}
This corollary is proved by having a computer write down all the graphs in $D^\square(n \leq 6,1)$ with no very small
faces, pentapents, or hexapents, according to the algorithm of Lemma \ref{lem:enumerategraphs}.

\begin{cor}
\label{cor:6:reductions=>span}
In a cubic category with reduction relations for pentapents and hexapents, $D^\square(n,1)$ spans $\cC_n$ for $n \leq 6$.
\end{cor}

Inside $D^\square(6,1)$ there are 6 `pentafork' diagrams. We next
analyze relations amongst these diagrams, up to lower order terms (i.e. terms
with strictly fewer vertices, which in this case is exactly the diagrams 
in $D(6,0)$).

\begin{lem}
\label{lem:pentaforks=>reductions}
We let $\rho$ denote the operator which rotates an open graph by one click counterclockwise.  A trivalent
category with relations reducng $n$-gons for $n \leq 4$ and  
a relation amongst the 6 pentaforks modulo $\Span(D(6,0))$ must also have a relation reducing a pentapent to something
in $\Span(D^\square(6,1))$ and a relation reducing a hexapent into $\Span(D^\square(7,1))$
\end{lem}
\begin{proof}
If there is a relation amongst the 6 pentaforks, then there is one of the form
$$\sum_{i=0}^6 \zeta^i \rho^i \left(\pentafork{60}\right) = 0 \mod{\Span(D(6,0))}$$
for some (not necessarily primitive) 6-th root of unity $\zeta$. Gluing an `H' 
diagram onto the upper right boundary face of each diagram, we obtain
\newcommand{\bigpentapent}{\begin{tikzpicture}[baseline=0]
	\draw (-30:.3cm)--(150:.3cm)--(110:.6cm)--(60:.7cm)--(10:.6cm)--(-30:.3cm)--(-70:.6cm)--(-120:.7cm)--(-170:.6cm)--(150:.3cm);
	\draw (10:.6cm)--(0:1cm);
	\draw (60:.7cm)--(60:1cm);
	\draw (110:.6cm)--(120:1cm);
	\draw (-70:.6cm)--(-60:1cm);
	\draw (-120:.7cm)--(-120:1cm);
	\draw (-170:.6cm)--(-180:1cm);
\end{tikzpicture}}
\newcommand{\bighexapent}{\begin{tikzpicture}[baseline=0,rotate=180]
	\draw (-20:.4cm)--(150:.3cm)--(110:.6cm)--(60:.7cm)--(10:.6cm)--(-20:.4cm)--(-45:.7cm)--(-90:.7cm)--(-135:.7cm)--(-170:.6cm)--(150:.3cm);
	\draw (10:.6cm)--(0:1cm);
	\draw (60:.7cm)--(60:1cm);
	\draw (110:.6cm)--(120:1cm);
	\draw (-170:.6cm)--(180:1cm);
	\draw (-90:.7cm)--(-90:1cm);
	\draw (-45:.7cm)--(-45:1cm);
	\draw (-135:.7cm)--(-135:1cm);
\end{tikzpicture}}  
$$
\rho^2\left(\bigpentapent\right) = - \zeta^{-2} \left(\bigpentapent\right) \mod{\Span(D^\square(6,1))}.$$
Applying this relation three times we see that
$$\bigpentapent = \rho^6 \left(\bigpentapent\right) = - \bigpentapent  \mod{\Span(D^\square(6,1))}.$$  This is only possible if the pentapent is zero modulo $\Span(D^\square(6,1))$.

Now we turn our attention to the hexapent.  Gluing lower two points of the tree
$\tikz[scale=0.4]{\draw (0:1) arc (0:180:1); \draw (60:1) -- (60:1.5); \draw
(90:1) -- (90:1.5); \draw (120:1) -- (120:1.5); }$ to the upper right boundary
face of the pentafork relation gives two hexapents, three pentapents, and a
pentagon connected via an edge to a square.  Applying the square reduction
relation and the new pentapent relation, we get the following relation among the
hexapents modulo lower order terms: $$\rho^2\left(\bighexapent\right) = -
\zeta^{-2} \bighexapent \mod{\Span(D^\square(7,1))}.$$  Applying this relation
seven times, we see that $$\bighexapent = \rho^{14}\left(\bighexapent\right) = -
\zeta^{-2} \bighexapent \mod{\Span(D^\square(7,1))}.$$  But $\zeta^{-2} \neq -1$
because $\zeta$ is a sixth root of unity and so cannot be a primitive fourth
root of unity.  Hence, we see that the hexapent is zero modulo
$\Span(D^\square(7,1))$.
\end{proof}

\begin{lem}
If $\dim \cC_4 = 4, \dim \cC_5 = 11$, and $\dim \Span D^\square(6,1) \leq 39$, 
then there is a relation amongst the 6 pentaforks modulo $\Span(D(6,0))$.
\end{lem}
\begin{proof}
By Proposition \ref{prop:6:preliminary-nonexistence}, the 34 diagrams in $D(6,0)$ must be linearly independent. 
Hence, modulo $\Span(D(6,0))$, there must either be two relations amongst the 
pentaforks, or a relation writing the hexagon as a linear combination of 
pentaforks along with a relation amongst the pentaforks.
\end{proof}

\begin{lem}
\label{lem:40=>spans}
If $\dim \cC_4 = 4, \dim \cC_5 = 11$, and $\dim \cC_6 \leq 40$, then $D^\square(6,1)$ spans.
\end{lem}
\begin{proof}
Suppose $\dim \Span D^\square(6,1) \leq 39$. By the previous relation, there's a 
relation among the pentaforks modulo lower terms.  Thus, Lemma \ref{lem:pentaforks=>reductions} 
shows there are reduction relations for pentapents and hexapents, and finally 
Corollary \ref{cor:6:reductions=>span} shows that $D^\square(6,1)$ spans.

Alternatively, if $\dim \Span D^\square(6,1) = 40$, then clearly $D^\square(6,1)$ 
spans because $\dim \cC_6 \leq 40$.
\end{proof}

We're now ready to give the proof of Proposition \ref{prop:6:uniqueness}. Its
statement assumes the hypotheses of Lemma \ref{lem:40=>spans}, and so we can now
assume that $D^\square(6,1)$ spans. In particular, any element of the kernel of
$M^\square(6,1)$, evaluated at the given values of $d$ and $t$, must be a
relation in the category. By explicit calculation, we see the kernel of
$M^\square(6,1)$ is four dimensional.  By Proposition \ref{prop:6:preliminary-nonexistence} 
we know that $D(6,0)$ is linearly independent, so there must be
$4$ relations among the pentaforks and the hexagon modulo $D(6,0)$.  In
particular, there must be at least $3$ relations among the pentaforks modulo
$D(6,0)$.  Then Lemma \ref{lem:pentaforks=>reductions} implies that there are
relations reducing pentapents and hexapents, Lemma \ref{lem:discharging-closed-pentapent} 
implies that we have enough relations to evaluate all closed
diagrams, and Corollary \ref{cor:reductions=>uniqueness} says that there is a
unique cubic category at the given values of $d$ and $t$.
\qed

\begin{remark}
Although we don't need to know the form of the pentafork, pentapent, and hexapent relations to prove uniqueness, 
we can obtain them explicitly as follows.  By looking at the radical of the inner product on $D^\square(6,1)$, we obtain three
relations amongst the pentaforks (modulo lower order terms), and one relation giving a hexagon in terms of the
pentaforks and lower order terms. The pentafork relations have rotational eigenvalues $-1, \omega, \omega^2$, for
$\omega$ a primitive cube root of unity. One can
then use the argument from Lemma \ref{lem:pentaforks=>reductions} to obtain explicit pentapent and hexapent reductions.
See {\tt code/H3_relations.nb}.
\end{remark}


\begin{remark}
The pentafork relations guarantee pentapent and hexapent relations.  We do not know whether one can use non-degeneracy to go the other way and recover the pentafork relations from the pentapent and hexapent relations.
\end{remark}

\subsection*{Proof of Proposition \ref{prop:6:realization} (Realization)}
First we recall the definition of the $H3$ category.  Then we must show it is
trivalent, has $d = \frac{3+\sqrt{13}}{2}$ and $t=-\frac{2}{3}d+\frac{5}{3}$,
and has dimension sequence $\dim \cC_4 = 4$, $\dim \cC_5 = 11$, and $\dim \cC_6
= 37 \leq 40$.  The construction of the Haagerup fusion category and hence $H3$
involves a choice of square root of $13$.  Thus taking Galois conjugation
$\sqrt{13} \mapsto -\sqrt{13}$ gives another category which realizes the other
value of $(d,t)$.

Recall that $H2$ is the Haagerup category with three invertible objects ($1$,
$h$, and $h^2$) and three non-invertible objects ($Y$, $hY$ and $h^2Y$) with
fusion rules $hY = Yh^2$ and $Y^2 = 1+Y+hY+h^2Y$.   Since the associator on the
$1$, $h$, $h^2$ subcategory is trivial, there's a canonical algebra
$\mathbb{C}[\mathbb{Z}/\mathbb Z3]$ in $H2$.  As defined in \cite{MR2909758},
$H3$ is the tensor category of bimodule objects in $H2$ over this algebra.
Recall that $H3$ has the same Grothendieck ring as category $H2$, with $g^3 =
1$, $gX=Xg^2$ and $X^2 = 1+ X + gX + g^2X$.  We will need two additional facts
about $H3$.  First, as shown in \cite{MR2909758}, there's no algebra structure on $1+X$ in $H3$.  (In fact, this is the property that was used to show $H3$ is not the same as $H2$, because the Haagerup subfactor gives an algebra structure on $1+Y$ in $H2$.  The second fact is that $H3$ is isomorphic to its complex conjugate.  This follows from $H2$ being isomorphic to its complex conjugate and the structure constants for the algebra $\mathbb{C}[\mathbb{Z}/\mathbb Z3]$ being real.

From the fusion rules for $H3$ we see that $\dim \Inv(X^{\otimes 3}) = 1$, and
thus that there is a map $f: X \tensor X \to X$. We need to see that this is a
trivalent vertex, and that it generates the category.  We know that $f$ must be
a rotational eigenvector, but we do not yet know the eigenvalue.  Note that
since conjugation by $g$ permutes $X$, $gX$, and $g^2 X$ we must have the same
rotational eigenvalue for each of the three maps $X \otimes X \rightarrow X$,
$gX \otimes gX \rightarrow gX$, and $g^2 X \otimes g^2 X \rightarrow g^2 X$.
Since $H3$ is isomorphic to its complex conjugate, we see that these eigenvalues
must all be $1$.

Next, we consider the pivotal subcategory $H3' \subset H3$ generated by the
trivalent vertex $f$.  Since $1+X$ does not have the structure of an algebra object in $H3$, we see that 
$H3$ cannot take a functor from the $SO(3)$ category at this value of $(d,t)$.  Hence, $\dim H3'_4$ cannot be smaller than $4$.

If $\dim \Inv_{H3'}(X^{\otimes 5}) < 11$, by Theorem \ref{thm:5}, $H3'$ is
either a $(G_2)_q$ category or  an ABA category.  In both cases we compute
$\tr{y_+} = \dim g X = \frac{3+\sqrt{13}}{2} = \dim g^2 X = \tr{y_-}$ using the
formulas from Theorem \ref{thm:idempotents-and-traces} and get a contradiction.
If $H3'$ were an $ABA$ category then $d=\frac{3+\sqrt{13}}{2}$ and $t^2-t-1=0$.
This splits into two cases based on the choice of $t$: either $\tr{y_+} \approx
-6.344$ while $\tr{y_-} \approx 12.9496$ if $t=\frac{1+\sqrt{5}}{2}$ (which is a
contradiction) or $\tr{y_+} \approx 1.04123$ while $\tr{y_-} \approx 5.56432$ if
$t=\frac{1-\sqrt{5}}{2}$ which is again a contradiction.

Similarly, we can rule out $H3'$ being a $(G_2)_q$ category, although the
calculations are messier. We first find the possible values of $q$ so  $$
\frac{3+\sqrt{13}}{2} = q^{10} + q^{8} + q^{2} + 1 + q^{-2} + q^{-8} + q^{-10}
$$ and for each, verify $\xi$ (with $t = - \frac{q^2-1+q^{-2}}{q^4+q^{-4}}$) is
invertible and $\tr{y_+} \neq \tr{y_-}$. (This calculation is contained in {\tt
code/idempotents.nb}.)

It is clear from the fusion rules that $d = \frac{3+\sqrt{13}}{2}$.  Since $\dim
\Inv_{H3'}(X^{\otimes 6}) \leq \dim \Inv_{H3}(X^{\otimes 6}) = 37$,  then by
Proposition \ref{prop:6:nonexistence} we must have
$t=-\frac{2}{3}d+\frac{5}{3}$.

Our final task is to show that $H3' = H3$.  Since the principal graph for $H3$
is depth 3, $H3$ is generated by its morphisms in $\Inv_{H3}(X^{\otimes k})$ for
$k \leq 6$.  Hence it is enough to show that  $\dim \Inv_{H3'}(X^{\otimes 6}) =
37$. Consider the diagrams in  $D^\square(6,1)$, leaving out (any) 4 of the
pentaforks, and compute the  determinant of the corresponding 37-by-37 matrix.
This is the nonzero number
\begin{multline*}
-\frac{12874212105079943047176987387755967947399861}{
   278128389443693511257285776231761} \\ -
   \frac{3570663990466532246521487414951846015270252}{
   278128389443693511257285776231761} \sqrt{13}
\end{multline*}
when $d= \frac{3+\sqrt{13}}{2}$ and $t=-\frac{2}{3}d+\frac{5}{3}$. (See {\tt code/H3_relations.nb} for this calculation.) Hence these 37
diagrams must in fact be linearly
independent in $H3'_6$. Thus $\dim \Inv_{H3'}(X^{\otimes 6}) \geq 37$, and in 
fact $H3'_6 = H3_6$, with $D^\square(6,1)$ spanning.  Thus $H3' = H3$, so $H3$ is a trivalent category satisfying the conditions of the theorem.
\qed

\section{Non-trivial rotational eigenvalues} \label{sec:rotationalev}

Suppose that $\cC$ is a pivotal category with the sequence $\dim \cC_n$
beginning $1,0,1,1$ which is generated by the map $1 \to X \tensor X \tensor X$.
This map must be an eigenvector for rotation, whose eigenvalue must be a cube
root of unity.  Thus far we have considered the case where the rotational
eigenvalue is $1$, and in this section we consider the case where the rotational
eigenvalue for counterclockwise rotation is a primitive cube root of unity
$\omega$.  We call such a category a twisted trivalent category.  Similarly to
before, a twisted trivalent category gives an invariant of planar trivalent
graphs which are decorated with a choice of direction at each vertex (which we
denote by placing a dot in one of the three regions adjacent to the vertex)
subject to the following skein relation $$\trivalentotherdotb =
\rho\left(\trivalentb\right) =  \omega \trivalentb.$$ As before we normalize the
bigon (with dots inward) to be $1$.

In a twisted trivalent category, the triangle with all dots pointing inward is some multiple of the trivalent vertex,
but it's manifestly rotationally invariant and hence must be zero.  This significantly simplifies the analysis, because
all the determinants considered above become polynomials in just one variable, the loop value $d$. It is then easy to
detect intersections between the corresponding varieties, by factoring into irreducible polynomials.

We now quickly run through the analogues of all the above results in the twisted
case.  Let $D_\omega^\square(n,k)$ be defined as before as diagrams with $n$
boundary points and no more than $k$ faces, none of which are squares or
smaller.  Note that there is an ambiguity here; for each diagram you must fix
the location of the dots.  We let $M_\omega^\square(n,k)$ be the matrix of inner
products and $\Delta_\omega^\square(n,k)$ be the determinant of this matrix.
Note that $M_\omega^\square(n,k)$ is well-defined only up to rescaling rows by a
power of $\omega$, and thus $\Delta_\omega^\square(n,k)$ is only well-defined up
to an overall rescaling by a power of $\omega$.  Below we always fix this
normalization in a way that makes the determinant a polynomial with real
coefficients (though it is not clear that this is possible in general).

Many of the small determinants used below can be calculated by hand. The larger
ones, however, rely on a newer software implementation of our methods, which
is available at \url{https://github.com/kim-em/toolkit/tree/9575deaefda74ad156fda01aa26e0a63d7d0918d/algebra-spiders}. (The older implementation, used and
described above, was not designed to keep track of rotations of vertices.) Our
plans for further investigations of small skein theories will make use of the
newer implementation, so we defer a description until a later paper.

\begin{prop}
For any $d \neq 2$ there are no twisted trivalent categories with $\dim \Span D_\omega(4,0) \leq 3$.
\end{prop}
\begin{proof}
If $D_\omega(4,0) \leq 3$ then the determinant $\Delta_\omega(4,0) = d^5 (d - 2)$ vanishes.  Since $d \neq 0$, we see that this forces $d=2$.
\end{proof}

\begin{prop}
When $d=2$ there is at most one trivalent category with $\dim \cC_4 \leq 4$.
\end{prop}
\begin{proof}
If $D_\omega(4,0)$ is linearly dependent then we get a relation of the form
\begin{equation*}
 \; 
\drawIb
\; = \alpha \;
 \drawHb
\;
+ \beta
\; 
\twostrandid
\; + \gamma \;
 \cupcap
\;.
\end{equation*} 
As before, this relation shows that $D_\omega(n,0)$ spans $\cC_n$.  On the other 
hand if $D_\omega(4,0)$ is linearly independent then it is a basis of $\cC_4$.  
Thus computing the kernel of the inner product we see that
\begin{equation*}
 \left( \; 
\drawIb
\; - \;
 \drawHb
\; \right)
 = 
\left( \; 
\twostrandid
\; - \;
 \cupcap
\; \right).
\end{equation*}
In either case, the relation shows that $D_\omega^\square(n,0)$ spans $\cC_n$ 
and hence the relation is enough to evaluate all closed diagrams.  Thus there is 
at most one such category.
\end{proof}

\begin{remark}
Note that this uniqueness statement applies for a particular choice of primitive 
cube root of unity for the rotational
eigenvalue; typically, as below, examples will appear in pairs corresponding to 
both choices.
\end{remark}

For realization we use \cite{MR2098028} which gives Chmutova's classification of
fusion categories of global dimension $6$.  We recall the following notation
from \cite[\S 3]{MR1976233}.  Suppose that $H$ is a subgroup of $G$, that $\xi
\in Z^3(G, \mathbb{C}^\times)$ is a $3$-cocycle and that $\psi \in C^2(H,
\mathbb{C}^\times)$ is a $2$-cochain whose coboundary is the restriction of
$\xi$.  Then we have a fusion category $\mathrm{Vec}(G,\xi)$ of twisted
$G$-graded vector spaces, and the twisted group ring $\mathbb{C}_\psi[H]$ is an
algebra object.  Let $\cC(G,H,\xi,\psi)$ denote the category of bimodules over
the twisted group ring.  Recall that $H^3(S_3, \mathbb C^\times) =
\mathbb{Z}/6\mathbb{Z}$, $H^3(S_2, \mathbb C^\times) = \mathbb{Z}/2 \mathbb{Z}$,
and $H^2(S_2, \mathbb C^\times) = 0$.  If $\xi$ is an element of order $3$ in
$H^3(S_3, \mathbb C^\times)$, then its restriction to $H^3(S_2, \mathbb
C^\times)$ is trivial.  So there's a 2-cochain $\psi$ on $S_2$ (which is unique
up to homology) such that $d \psi = \xi$ on $S_2$.

\begin{prop}
If $\xi$ is a $3$-cocycle of order $3$ in $H^3(S_3, \mathbb C^\times)$ and $\psi$ a $2$-cochain on $S_2$ such that $d
\psi = \xi$ on $S_2$, then $\cC(S_3, S_2, \xi, \psi)$ gives a trivalent category with $d=2$ and $\dim \cC_4 = 3$.
\end{prop}
\begin{proof}
These two categories (one for each choice of $\xi$) each have three objects 
$1$, $X$, $g$ with $g^2 = 1$, $g X = X g$, and $X^2 = 1+2X+g$.  (In other words, 
they are near group categories \cite{MR1997336,MR3167494,1512.04288} for the group 
$\mathbb{Z}/2\mathbb{Z}$.)  So a direct calculation shows that the dimensions of 
the Hom spaces are given by $1,0,1,1,3, \ldots$.  Since these are distinct from 
$\mathrm{Rep}(S_3) = SO(3)_{\zeta_{12}}$, our previous classification shows that 
the rotational eigenvalue cannot be $1$ so they must both be twisted trivalent 
categories.
\end{proof}

Combining these results we get the following classification.

\begin{thm}
A twisted trivalent category $\cC$ with $\dim \cC_4 \leq 3$ must be one of the 
two $\cC(S_3, S_2, \xi, \psi)$ categories.
\end{thm}

(Here we haven't specified which values of $\xi$ correspond to which rotational 
eigenvalues, although it must be a bijective correspondence; it would be 
interesting to work this out.)

Now, a twisted cubic category $\cC$ (that is, one with $\dim \cC_4 = 4$) must 
(by the analogue of Theorem \ref{prop:cubic}) have $d \neq 2$, $D_\omega(4,0)$ a 
basis of $\cC_4$ and satisfy the relation
\begin{equation*}
  \ngonb[45]{4} =  - \frac{1}{d}
 \left( \; 
\drawIb
\; + \;
 \drawHb
\; \right)
 + \frac{1}{d}
\left( \; 
\twostrandid
\; + \;
 \cupcap
\; \right)
\end{equation*}

Using only this relation (along with the known values $d$ for loops, $1$ for bigons, and $0$ for triangles), we can readily compute
all the following determinants.
\begin{align*}
\Delta_\omega(5,0) & = d^{10} (d - 2)^5 \\
\Delta_\omega^\square(5,1) & = d^9 (d - 2)^6\\
\Delta_\omega(6,0) & = - d^{26} (d-2)^{23} (d-1)  \left(d^2-d-1\right) \left(d^3-2 d^2-3 d+1\right) \left(d^4-4 d^3+3 d^2-d-1\right)\\
\Delta_\omega^\square(6,1) & = - d^{12} (d-2)^{31} (d-1)^2 (d+1)^2 \left(d^2-3 d-1\right)^4 \left(d^4-2 d^3-3 d^2-d+2\right)\\
\Delta_\omega^\square(7,1) & = - d^{-86} (d-2)^{141} (d+1)^{16} \left(d^2-3 d-1\right)^{35}  Q_{\omega,9} Q_{\omega,60}
\end{align*}
(The polynomials $Q_{\omega,i}$ appear in the appendix.)

\begin{prop}
In any twisted cubic category $D_\omega^\square(5,1)$ is linearly independent.
\end{prop}
\begin{proof}
Since $d$ is not $0$ or $2$, we have $\Delta_\omega^\square(5,1) \neq 0$.
\end{proof}

It follows that there are no twisted cubic categories with $\dim \cC_5 \leq 10$.  It also follows that if $\dim \cC_5 = 11$, then $D_\omega^\square(5,1)$ is a basis for $\cC_5$.

\begin{prop}
If $\cC$ is a twisted cubic category with $\dim \cC_5 = 11$ and $\dim \cC_6 \leq 40$, then $d = -1$ or $d = \frac{3\pm \sqrt{13}}{2}$.
\end{prop}
\begin{proof}
First, $\dim \cC_6 \leq 40$ implies that $D_\omega^\square(6,1)$ is linearly 
dependent and hence $\Delta_\omega^\square(6,1)$ and $\Delta_\omega^\square(7,1)$ 
both vanish.  But their only shared factors are $d+1$ and $d^2 -3d- 1$.  
\end{proof}

\begin{prop}
There is at most one twisted cubic category with $\dim \cC_5 = 11$ and 
$\dim \cC_6 \leq 40$ for each of the three points $d = -1$ or 
$d = \frac{3 \pm \sqrt{13}}{2}$.
\end{prop}
\begin{proof}
Again we must have a pentafork relation.  The method of Lemma
\ref{lem:pentaforks=>reductions} can still be used to guarantee pentapent and
hexapent reductions and thus uniqueness, as follows.  The first argument there
shows that we have a relation for reducing the pentapent unless $1 =
(-\zeta^{-2} \omega)^3 = -1$.  The second argument there shows that we have a
relation for reducing the hexapent unless $1 = (-\zeta^{-2} \omega)^7 =
-\zeta^{-2} \omega$ which can only occur when $\zeta^2 = -\omega$.  Here
$\zeta^2$ is a third root of unity, while $-\omega$ is a primitive sixth root of
unity, so this can not happen.
\end{proof}

The two cases with $d = \frac{3 + \sqrt{13}}{2}$ are realized by the twisted
Haagerup fusion categories conjectured in \cite{MR2837122}.  Ostrik observed
that these can be constructed as follows.  Start with the 
construction of two $\mathbb{Z}/9\mathbb{Z}$ Izumi near group categories in 
\cite{MR3167494}.
Following Izumi and Evans-Gannon, the center of one of these categories contains a copy
of $\mathrm{Rep}(\mathbb{Z}/3\mathbb{Z})$ as a symmetric tensor subcategory. 
One can then de-equivariantize \cite{MR2609644} by this to get a new category which has objects
$1, g, g^2, X, gX, g^2 X$ but where the three invertible elements have
nontrivial associator.  We will denote the categories obtained this way
$H_\omega$
(one for each primitive cube root of unity).  The two with $d = \frac{3 -
\sqrt{13}}{2}$ are realized by the Galois conjugates of the $H_\omega$.

\begin{remark}
Note that the untwisted $H2$ (one of the even parts of the Haagerup subfactor)
and $H3$ can be constructed in a similar way.  Start with the unique $\mathbb
{Z}/3\mathbb{Z} \times \mathbb{Z}/3\mathbb{Z}$ Izumi near-group category.
The center of this category has two different copies of 
$\mathrm{Rep}(\mathbb {Z}/3\mathbb{Z})$ (one is the diagonal and the other the
anti-diagonal). 
De-equivariantizing by each of these gives $H2$ and $H3$ (although it's not
clear which is which).
\end{remark}

\begin{prop}
The two categories $H_\omega$ are twisted cubic categories with $\dim \cC_5 = 11$, $\dim \cC_6 = 37$, and $d = \frac{3 + \sqrt{13}}{2}$.
\end{prop}
\begin{proof}
This argument closely follows the argument from \ref{prop:6:realization}.  Again
let $H'$ be the subcategory generated by the trivalent vertex.  The same
argument as before shows that $H'$ must be a trivalent or twisted trivalent
category.    By the fusion rules we know that $d = \frac{3+\sqrt{13}}{2}$.  We
need to see that it is twisted and show that $H' = H_\omega$ which would show
that $H_\omega$ is trivalent and, from the fusion rules, that the dimensions of
the invariant spaces begin $1,0,1,1,4,11,37$.  If $H'$ were untwisted, then it
would have to be on our list of untwisted trivalent categories.  First, $H'$
cannot be $SO(3)_q$ because $1+X$ is not an algebra (if it were then $g$ would
be in the normalizer of that algebra contradicting the nontrivial associator).
Second, $H'$ cannot be an $ABA$ or $(G_2)_q$ category for the same dimensional
considerations that showed $H3'$ couldn't lie in those familes.  Third, $H'$
can't be $H3$ because of the nontriviality of the associator.   Hence it is a
twisted trivalent category.  Finally, in order to show that $H_\omega = H'$ it
is enough to show that $\dim \Inv_{H'}(X^{\otimes 6}) = 37$ which follows
calculating that the $41$-by-$41$ matrix $M_\omega^\square(6,1)$ has rank at
least $37$ at this value of $d$.  This rank calculation can be easily done
modulo a prime sitting above $3$ in $\mathbb{Z}[d]$.
\end{proof}
\begin{question}
Does there exist a twisted trivalent category $\cQ_\omega$ with $d=-1$, $\dim \cC_5 = 11$ and $\dim \cC_6 \leq 40$?  
\end{question}

Such a category is unique if it exists.  It would have $\dim \cC_6 = 39$  and would satisfy
the following two relations\footnote{Although the computer alerted us to the 
existence of these relations, we actually computed them by hand, since it is 
difficult to read off from our computer program where the dots belong.  This 
by-hand calculation following \cite{MR1265145} took two people-days.} (where 
$\zeta$ is the primitive sixth root of unity which is a square root of $\omega$):
\begin{multline}\nonumber 
\ngonb[0]{6}  \; + \; \sum_{i=0}^5 \rho^i\left( \pentaforkb  \right) \; + \; \sum_{i=0}^5 \rho^i\left( \Icupb \right) \; + \; \sum_{i=0}^1 \rho^i\left( \threecupsb \right) \\ =  \sum_{i=0}^2 \rho^i\left(  \doubletrib  \right) \; + \; \sum_{i=0}^2 \rho^i\left(\firstpinwheelb) \right) \; + \; \sum_{i=0}^2 \rho^i\left( \secondpinwheelb \right) 
\end{multline}
and
\begin{multline}\nonumber 
0 = \sum_{i=0}^5 \zeta^i \rho^i\left( \pentaforkb \right) \; + \; \sum_{i=0}^5 \zeta^i \rho^i\left( \Hcupb \right) \\ + \; \sum_{i=0}^5 \zeta^i \rho^i\left(\Icupb \right) \; + \; \sum_{i=0}^5 \zeta^i \rho^i\left( \sixforestb \right).
\end{multline}

 Such a $\cQ_\omega$ cannot come from any operator algebraic construction since
$d=-1 < 0$.  As we will see in the next section, $\cQ_\omega$ is not braided and
so $\cQ_\omega$ is not a Drinfel'd-Jimbo quantum group.  Furthermore, the
dimensions of the objects at depth $2$ are not real (they are conjugate
primitive sixth roots of unity), so $\cQ_\omega$ must have infinitely many
simple objects.  We do not think it comes from any well understood construction,
but it has also passed every test we have attempted to use to rule it out.  We
also note that the above relations are particularly nice as they involve only 24
terms each.

In conclusion, we have the following classification of twisted trivalent categories.

\begin{center}
\begin{tabular}{l|l}
dimension bounds         & new examples 			\\ \hline
1,0,1,1,3,\ldots           &$\cC(S_3, S_2, \xi, \psi)$   	 \\
1,0,1,1,4,10,\ldots		   & nothing \\
1,0,1,1,4,11,37,\ldots     & $H_\omega$ 			\\
1,0,1,1,4,11,39,\ldots     &  $\cQ_\omega$, if it exists 	\\
1,0,1,1,4,11,40,\ldots     & nothing more		\\
\end{tabular}
\end{center}

\section{Braided Trivalent Categories}
\label{sec:braided}

\begin{defn}
We call a trivalent or twisted trivalent category braided if there is an
element in the $4$-boundary point space, which we write using a crossing,
which satisfies the following relations (in the untwisted case ignore the dots): 
$$ \Rtwo = \twostrandid,$$
\begin{equation} \label{eq:pullthrough1} \trivalentpullthroughone = \trivalentpullthroughtwo,\end{equation}
\begin{equation} \label{eq:pullthrough2} \trivalentpullthroughthree = \trivalentpullthroughfour. \end{equation}
\end{defn}

Note that the latter two relations above imply the Reidemeister $3$ relations
also hold via the Kauffman trick (since the category is generated by the
trivalent vertex and so the crossing can be written in terms of trivalent
vertices).  As usual in quantum topology, the Reidemeister $1$ relation need not
hold.

\begin{lem}
There are no braided twisted trivalent categories.
\end{lem}
\begin{proof}
By dimensional considerations we have the following relations
$$\strandtwist = \alpha \singlecap$$
$$\trivalenttwistb = \beta \trivalentb$$ 
for some numbers $\alpha$ and $\beta$.  If we use the other crossing in these 
pictures we get the same relations with $\alpha^{-1}$ and $\beta^{-1}$.  We can 
now compute the action of rotation on the twisted trivalent vertex in two different 
ways:  $$\overviolinb = \rotatedtrivalentb = \underviolinb.$$  Thus, 
$\alpha^{-1}\beta^2 = \omega = \alpha \beta^{-2}$, and so $\omega = \omega^{-1}$ 
which is a contradiction.
\end{proof}

\begin{remark}
The same argument shows that if $X$ is a simple object in a braided tensor 
category and $f: 1 \rightarrow X^{\otimes n}$ is an eigenvector for both rotation 
and braiding, then the rotational eigenvalue is $\pm 1$.
\end{remark}

\begin{lem}
\label{lem:braiding-C5-bound}
If $\cC$ is a braided trivalent category with $\dim \cC_4 \leq 4$, then $\dim \cC_5 \leq 10$.
\end{lem}
\begin{proof}
Since $\dim \cC_4 \leq 4$ we get that $D(4,0)$ forms a spanning set for $\cC_4$.
Look at Equations \eqref{eq:pullthrough1} and \eqref{eq:pullthrough2} and expand
every crossing as a sum of diagrams in $D(4,0)$.  These each give a relation
between diagrams in $D^\square(5,1)$.  We claim that at least one of these
relations is nontrivial.  Since the crossings are rotations of each other, at
least one of the crossings has a nontrivial coefficient of either
$\smalltwostrandid$ or $\smalldrawI$.  In the former case the expanded relation
has a nontrivial coefficient of $\trivalentcap$,  and in the latter case the
expanded relation has a nontrivial coefficient of the pentagon.  Thus we have a
nontrivial relation among $D^\square(5,1)$, so by Lemma
\ref{lem:5:dependent=>spans}, we get that $D(5,0)$ spans $\cC_5$.  In
particular, $\dim \cC_5 \leq 10$.
\end{proof}

This is enough to give a complete classification of braided trivalent categories, but before stating this classification we list the examples that occur.

\begin{ex}
The standard braiding on $SO(3)_q$ is

\begin{align*}
\overcrossing & = (q^2 - 1) \twostrandid + q^{-2} \cupcap - (q^2+q^{-2}) \drawH.
\end{align*}

Note that although $SO(3)_{\pm q^{\pm 1}}$ are the same fusion category, there are two distinct braided tensor categories corresponding to $\pm q$ and $\pm q^{-1}$.  Finally, note that when $q= \pm i$ this formula gives the standard symmetric braiding on $OSp(1|2)$.
\end{ex}

\begin{ex}
Let $X$ be the standard $2$-dimensional representation of $S_3$ and $X \otimes X \rightarrow X$ be a non-zero map (which is unique up to scalar), then this generates a trivalent category.  The standard symmetric braiding is  

\begin{align*}
\symmetriccrossing & = \twostrandid - \cupcap + 2 \drawH \\
	& = \drawH + \drawI.
\end{align*}

It is easy to see from our classification that this category agrees with $SO(3)_q$ for $q$ a primitive $12$th root of unity as a trivalent category.  However, the standard symmetric braiding on representations of $S_3$ does not agree with the standard braiding for $SO(3)_q$.  We will denote this braided trivalent category by $S_3$, to distinguish it from $SO(3)_{\zeta_{12}}$.
\end{ex}

\begin{ex}
The standard braiding for $(G_2)_q$ for $q \neq \pm i$ is

\begin{align*}
\overcrossing & = \frac{1}{q+q^{-1}} \left(q^3 \twostrandid + q^{-3} \cupcap\right) \\
& \qquad
	- \frac{q^6+q^4+q^2+q^{-2}+q^{-4}+q^{-6}}{q+q^{-1}} 
	\left( q \drawH + q^{-1} \drawI \right).
\end{align*}

Note that although $(G_2)_{\pm q^{\pm 1}}$ give the same fusion category, there are two distinct braided tensor categories corresponding to $\pm q$ and $\pm q^{-1}$.
\end{ex}

\begin{cor}
The only braided trivalent categories with $\dim \cC_4 \leq 4$ are $SO(3)_q$ (for any $q$ including $SO(3)_{\pm i} = OSp(1|2)$), $S_3$, and $(G_2)_q$ for $q \neq \pm i$.
\end{cor}
\begin{proof}
By our classification, we need only classify all braidings for the $SO(3)_q$, ABA, and $(G_2)_q$ categories.  In the $G_2$ case, when $q$ is not a fourth or twentieth root of unity, this was done by Kuperberg in \cite{MR1265145}.  Following Kuperberg, you write down a general element of the $4$-boundary point space and check whether it satisfies the braiding relations.  Since this is somewhat tedious and since the hardest case was already done by Kuperberg, we will skip much of the details here.

For the $SO(3)_q$ categories (which agree for $\pm q^{\pm 1}$), generically there are exactly two braidings on corresponding to one corresponding to the standard braiding of $SO(3)_{\pm q}$ and the other to $SO(3)_{\pm q^{-1}}$.  When $q = \pm 1$ or $q= \pm i$, these two braidings agree and yield the symmetric braidings on $SO(3)$ and $OSp(1|2)$.  When $q$ is a primitive $12$th root of unity, there are three braidings, two corresponding to the $SO(3)_q$ braidings and one corresponding to the symmetric braiding which corresponds to the standard braiding on $S_3$.

A direct calculation shows that the ABA categories do not have a braiding.  When $\delta$, the loop value for $A$, is not the golden ratio or its conjugate, there is a simple more conceptual approach.  Namely, the fusion rules are noncommutative since $A^{(2)}(ABA) = A^{(3)} B A + ABA$ while $(ABA) A^{(2)} = A B A^{(3)} + ABA$.  (When $\delta$ is the golden ratio, the fusion rules are commutative so one must use the brutal approach.)

Finally, for $(G_2)_q$, Kuperberg proved that when the $5$-boundary point space is $10$-dimensional, there are exactly two braidings for such a category, one corresponding to the standard braiding of $(G_2)_{\pm q}$ and the other to $(G_2)_{\pm q^{-1}}$.  The only remaining cases are $q$ is a primitive $4$th or $20$th root of unity.  The former case was already dealt with above, and it turns out in the latter case there are still only the two standard braidings.
\end{proof}

We expect this theorem to give Wenzl-style recognition results \cite{MR1237835, MR2132671} for $(G_2)_q$ and Deligne's $S_t$, showing that they are the only braided tensor categories with their Grothendieck rings.

\begin{cor}
The only symmetric trivalent categories with $\dim \cC_4 \leq 4$ are $SO(3)$, $S_3$, $(G_2)$, and $OSp(1|2)$.
\end{cor}

Note that by Deligne's theorem \cite{MR1944506}, any symmetric abelian category of exponential growth must be the category of representations of a supergroup, so this corollary is not surprising.  The exponential growth condition is satisfied in our setting, but this corollary does not follow directly from Deligne's theorem because {\it a priori} there could be symmetric trivalent categories which do not come from abelian categories.

\section{Prospects}

There are several obstacles to pushing the above techniques further in the study
of trivalent categories.  First, even though in principal just knowing $d$ and
$t$ should be enough to calculate the determinants $D^\square(n,k)$ for $n+2k <
12$, as we saw in Conjecture \ref{conj:Delta7_1}, in practice arithmetic of two
variable rational functions is sufficiently difficult that we cannot compute all
of these determinants exactly.  If one is willing to accept probabilistic proofs
then we can compute more of these determinants.  Second, if we want to go beyond
$n+2k = 12$ then we need to introduce the dodecahedron as a third variable,
which will make things quite a bit more complicated.  Finally, we are already
pushing up against the limits of practical Gr\"obner basis calculations: for
example we already cannot directly intersect $\Delta(7,0)$ and
$\Delta^\square(7,1)$.  For all these reasons, it is unlikely that we will be
able to push the classification of trivalent categories much further without new
ideas.

Instead we plan to continue investigating these ideas in other settings than the trivalent setting.  There are numerous good candidates for investigation, including the following.

\begin{itemize}
\item Braided trivalent categories with dimensions bounded by $1,0,1,1,5$.  This includes the conjectured exceptional series of Deligne and Vogel \cite{MR1378507, MR2769234}.
\item Skein theoretic invariants of planar graphs together with a $2$-coloring of the vertices and a $2$-coloring of the faces.  These correspond to quadrilaterals of subfactors and we hope to strengthen the classification results of Grossman, Izumi, and Jones \cite{MR2257402, MR2418197}.
\item Categories generated by a 4-valent vertex with a checkerboard shading.  These were studied by Bisch, Jones, and Liu in \cite{MR1733737, MR1972635, 1410.2876}, and we hope to push their techniques beyond what can be done by hand.
\item Categories generated by a $2n$-valent vertex with a checkerboard shading.  For $n=2$ this is the previous example, and for $n=3$ they were studied by Dylan Thurston \cite{0405482}.
\item Skein theoretic invariants of virtual knots.  This includes the representation theory of the Higman--Sims sporadic finite simple group \cite{MR1188082, MR1469634}.
\end{itemize}

\appendix
\section{Skein theoretic invariants and pivotal categories}
\label{sec:local}
The goal of this section is to provide background so that this paper is accessible to knot theorists, graph theorists, and other readers unfamiliar with tensor categories or planar algebras.

Suppose we want to study certain numerical invariants $f$ of planar trivalent graphs. Assume that $f$ of the empty
diagram is $1$, that $f \left( \tikz[baseline=-1mm] \draw (0,0) circle (3mm); \right)$ and $f \left( \begin{tikzpicture}[baseline=-1mm]
	\draw (0,0) circle (3mm); 
	\draw (.3,0)--(-.3,0);
\end{tikzpicture}
\right)$ are nonzero, and that $f$ satisfies the following multiplicative conditions:
\begin{enumerate}
\setcounter{enumi}{-1}
\item 
$f \left( \begin{tikzpicture}[baseline=-1mm]
	\node[circle,draw,dotted] (X) at (0,0) {X};
	\node[circle,draw,dotted] (Y) at (1,0) {Y};	
\end{tikzpicture} \right) 
=  f(X) \cdot f(Y)$

\item 
$f \left( \begin{tikzpicture}[baseline=-1mm]
	\node[circle,draw,dotted] (X) at (0,0) {X};
	\node[circle,draw,dotted] (Y) at (1,0) {Y};
	\draw (X.east)--(Y.west);	
\end{tikzpicture} \right)
= 0$

\item 
$f \left( \begin{tikzpicture}[baseline=-1mm]
	\node[circle,draw,dotted] (X) at (0,0) {X};
	\node[circle,draw,dotted] (Y) at (1,0) {Y};
	\draw (X.north east) .. controls (.5,.4) .. (Y.north west);	
	\draw (X.south east) .. controls (.5,-.4) .. (Y.south west);	
\end{tikzpicture} \right) 
=  f \left( \begin{tikzpicture}[baseline=-1mm]
	\node[circle,draw,dotted] (X) at (0,0) {X};
	\draw (X.north east) .. controls (.5,.4) and (.7,.2).. (.7,0) .. controls (.7,-.2) and (.5,-.4) .. (X.south east);	
\end{tikzpicture} \right) 
\cdot
f \left( \begin{tikzpicture}[baseline=-1mm]
	\node[circle,draw,dotted] (X) at (0,0) {Y};
	\draw (X.north west) .. controls (-.5,.4) and (-.7,.2).. (-.7,0) .. controls (-.7,-.2) and (-.5,-.4) .. (X.south west);	
\end{tikzpicture} \right) /
f \left( \tikz[baseline=-1mm] \draw (0,0) circle (3mm); \right)
$

\item 
$f \left( \begin{tikzpicture}[baseline=-1mm]
	\node[circle,draw,dotted] (X) at (0,0) {X};
	\node[circle,draw,dotted] (Y) at (1,0) {Y};
	\draw (X.east) -- (Y.west);	
	\draw (X.north east) .. controls (.5,.4) .. (Y.north west);	
	\draw (X.south east) .. controls (.5,-.4) .. (Y.south west);	
\end{tikzpicture} \right) 
=  f \left( \begin{tikzpicture}[baseline=-1mm]
	\node[circle,draw,dotted] (X) at (0,0) {X};
	\draw (X.north east) .. controls (.5,.4) and (.7,.2).. (.7,0) .. controls (.7,-.2) and (.5,-.4) .. (X.south east);
	\draw (X.east)--(.7,0);		
\end{tikzpicture} \right) 
\cdot
f \left( \begin{tikzpicture}[baseline=-1mm]
	\node[circle,draw,dotted] (X) at (0,0) {Y};
	\draw (X.north west) .. controls (-.5,.4) and (-.7,.2).. (-.7,0) .. controls (-.7,-.2) and (-.5,-.4) .. (X.south west);
	\draw (X.west)--(-.7,0);	
\end{tikzpicture} \right) /
f \left( \begin{tikzpicture}[baseline=-1mm] 
	\draw (0,0) circle (3mm); 
	\draw (.3,0)--(-.3,0);
\end{tikzpicture}
\right)
$
\end{enumerate}
Thus the invariant of any $k$-disconnected graph for $k \leq 3$ is determined by the invariants of the pieces.

\begin{ex}
\label{ex:vertex-counting-invariant}
An almost trivial example of a multiplicative invariant of graphs is $a^{\# V}$, for some number $a$, where $\# V$
denotes the number of trivalent vertices in the graph.
\end{ex}
\begin{ex}
\label{ex:chromatic-invariant}
An important example of a multiplicative invariant of graphs is the number of $n$-colorings of the faces of the graph,
divided by $n$.  (The division by $n$ is a normalization factor ensuring that the empty graph is assigned $1$ instead
of $n$.) This example can be generalized by considering non-integer specializations of the chromatic polynomial.
\end{ex}

\begin{question*}
What examples are there of such multiplicative invariants of trivalent planar graphs?
\end{question*}

While the question appears to be an elementary question about 
planar trivalent graphs, we discover that the examples are actually related to quite distant subjects in mathematics.
In particular, we are able to identify each of the small examples we encounter with some suprising
or exotic object coming from representation theory or the theory of subfactors!

In order to understand the main results of the paper in the language of graph
invariants, we first want to extend this invariant of closed trivalent graphs to
an invariant of planar graphs with boundary.  That is, we extract a sequence of
vector spaces, the `open graphs, modulo negligibles'.  We now describe how these
vector spaces have the structure of a pivotal tensor category (or planar
algebra).

Let $\hat{\cC}_n$ denote the (infinite dimensional) vector space with basis the
planar trivalent graphs drawn in the disc, with $n$ fixed boundary points, up to
isotopy rel boundary. This vector space has a natural bilinear pairing, given by
gluing two open graphs together (starting at a preferred boundary point), to
obtain a closed planar graph, which we then evaluate to a number using our
multiplicative invariant $f$. The kernel of this bilinear pairing is called `the
negligible elements'. Let $\cC^f_n$ denote the quotient vector space of
$\hat{\cC}_n$ by negligible elements.

One may assemble these vector spaces into a single algebraic structure,
variously axiomatized as an (unshaded) planar algebra \cite{math.QA/9909027}, a
spider \cite{MR1403861} or a pivotal tensor category \cite{MR1686423}.  We'll
only describe the last in any detail. The category, which we'll  call $\cC^f$,
has as objects the natural numbers. We'll first describe a bigger category of
trivalent graphs, which we call $\hat{\cC}$ and which does not depend at all on
our multiplicative invariant. In $\hat{\cC}$, the morphisms from $n$ to $m$ are
simply the formal linear combinations of planar graphs drawn in a rectangle with
$n$ points along the bottom edge and $m$ points along the top edge, i.e.~the
vector space $\hat{\cC}_{n+m}$. We can compose morphisms in the obvious way, by
stacking rectangles. This category is a tensor category, with the tensor product
given by drawing diagrams side by side. Finally it is a pivotal category, with
the evaluation and coevaluation maps given by caps and cups.

Inside $\hat{\cC}$, the negligible (with respect to $f$) elements form a planar ideal --- if some 
(linear combination of) graphs pair with arbitrary other graphs to give zero, 
then glueing more graph to the boundary preserves this property. We thus define the 
category $\cC^f$ to be the quotient of $\hat{\cC}$ by the negligible ideal.  
This ``ideal" property says that we can treat the negligible elements as skein 
relations: they can be applied locally in any part of a graph.  
Furthermore, typically this ideal is finitely generated by a few particular 
skein relations.

Thus, in $\cC^f$, the objects are still the natural numbers and the morphisms
from $n$ to $m$ are just $\cC^f_{n+m}$. The category $\cC^f$ is still a pivotal
tensor category, and now it is \emph{evaluable} (i.e. $\dim \cC^f_0 = 1$, and in
fact $\cC^f_0$ may be identified with the ground field by sending the empty
diagram to $1$) and \emph{non-degenerate} (i.e. for every morphism $x: a \to b$,
there is another morphism $x': b \to a$ so $\langle x, x' \rangle \neq 0 \in
\cC^f_0$).  Writing $X$ for the generating object in $\cC^f$ (i.e. 1 in the
natural numbers!), we see that $X$ is a symmetrically self-dual object, with
duality pairings and copairings given by the cap and cup diagrams. Moreover, the
trivalent vertex is a rotationally symmetric map $1 \to X \tensor X \tensor X$.

\begin{ex}
If the invariant is the normalized number of $n$-colorings described in Example \ref{ex:chromatic-invariant}, then a linear combinations of
graphs is negligible if and only if for any coloring of the boundary faces the given linear combination of the numbers of ways of extending that coloring to the interior is zero.  For example, the following element of $\cC_3$ is negligible:
$$
\ngon[90]{3}
- (n-3) \cdot
\begin{tikzpicture}[baseline=.1cm,scale=0.75]
\draw (0,0) -- (0,1);
\draw (0,0) -- (0.7,-0.5);
\draw (0,0) -- (-0.7,-0.5);
\end{tikzpicture}
$$
In particular, this gives a skein relation in $\cC^f$ which says that you can remove a triangle and multiply by $(n-3)$.
There are also other negligible elements; in fact after renormalizing the trivalent vertex, $\cC^f$ becomes equivalent
to the pivotal category $SO(3)_q$ coming from quantum groups where $q$ is a number satisfying $ (q+q^{-1})^2=n$
(see Section \ref{sec:four} for a description of $SO(3)_q$).
\end{ex}

\begin{prop} The construction of $\cC^f$ from $f$ gives a bijective correspondence between trivalent categories and
multiplicative  invariants of planar graphs. \end{prop} 

\begin{proof}
First we prove that the category $\cC^f$ constructed from a multiplicative
invariant $f$ is trivalent.  Consider $\cC^f_0$.  The empty diagram is not
negligible, so we need only show that any closed diagram is a multiple of the
empty diagram.  If $\alpha$ is a closed diagram and $\beta$ is the empty
diagram, then $\alpha - f(\alpha) \beta$ is negligible, so in $\cC^f_0$ we have
that $\alpha = f(\alpha) \beta$.  Now we look at $\cC^f_1$. By multiplicativity
we have that any diagram with one boundary point is negligible, so $\dim\cC^f_1
= 0$.  The remaining cases are similar.

Given a trivalent category $\cC$, we need to construct a multiplicative invariant of planar graphs.  
The usual diagrammatic calculus for pivotal categories shows that any trivalent category gives an invariant of closed graphs just by interpreting the graphs as elements of $\cC_0$ and sending the empty diagram to $1$. 

We want to check that this invariant is multiplicative, in which case it is
clear that it provides an inverse to $f \mapsto \cC^f$.  We first check that the
loop and the theta are nonzero.  The single strand in $\cC_2$ must be nonzero,
because if it were zero then all nonempty diagrams would be zero.  Since $\dim
\cC_2 = 1$, we see that any diagram in $\cC_2$ is a multiple of the single
strand, hence nondegeneracy says that the inner product of the strand with
itself is nonzero, hence the loop value is nonzero.  Similarly, by considering
$\cC_3$ we see that the theta value is nonzero.  Next we want to prove the
multiplicative properties.  Each of these are similar, so we only prove (2).  We
have that
\begin{tikzpicture}[baseline=-1mm]
	\draw (0,-.5)--(0,.5);
	\node[circle,draw,dotted, fill=white] (X) at (0,0) {X};
\end{tikzpicture}
 is some multiple of the single strand, so  we see that 
 $\begin{tikzpicture}[baseline=-1mm]
	\draw (0,-.5)--(0,.5);
	\node[circle,draw,dotted, fill=white] (X) at (0,0) {X};
\end{tikzpicture}
=
\left( \begin{tikzpicture}[baseline=-1mm]
	\node[circle,draw,dotted] (X) at (0,0) {X};
	\draw (X.north east) .. controls (.5,.4) and (.7,.2).. (.7,0) .. controls (.7,-.2) and (.5,-.4) .. (X.south east);	
\end{tikzpicture} 
/ 
\begin{tikzpicture}[baseline=-1mm]
	\draw (0,0) circle (.3cm); 	
 \end{tikzpicture} 
\right)
\cdot
\begin{tikzpicture}[baseline=-1mm]
 \draw (0,-.5)--(0,.5);
 \end{tikzpicture} 
$ (by pairing with the strand).  Substituting this into the LHS of (2) gives the RHS. 
\end{proof}


\section{Polynomials appearing in determinants}
This appendix contains some of the irreducible factors of determinants appearing
in this paper. The other irreducible factors, which are very large, are
contained in text files packaged with the {\tt arXiv} source of this paper, and
described here.  Each polynomial is named as $Q_{i,j}$, where $i$ is the largest
exponent of $d$ and $j$ is the largest exponent of $t$. Where two polynomials
have the same pair of largest exponents, we name them with an additional
character in the subscript, as in $Q_{2,4,a}$ and $Q_{2,4,b}$.

\begin{align*}
P_{SO(3)} & = d (t-1)-t+2 \\\displaybreak[1]
P_{ABA} & = t^2-t-1 \\\displaybreak[1]
P_{G_2} & = d^2 t^5+d \left(2 t^5-4 t^4-t^3+6 t^2+4 t+1\right)+t^5-4 t^4+t^3+7 t^2-2 \\\displaybreak[1]
Q_{0,1} & = t+1 \\\displaybreak[1]
Q_{1,1} & = d (t+1)+t \\\displaybreak[1]
Q_{1,2} & = d \left(2 t^2+2 t+1\right)+3 t^2-2 \\\displaybreak[1]
Q_{2,3} & = d^2 \left(t^3+t^2-2 t-1\right)+d \left(2 t^3-2 t^2+t\right)+t^3-3 t^2+t+4 \\\displaybreak[1]
Q_{3,4} & = d^3 \left(t^4+3 t^3-t^2-3 t-1\right)+d^2 \left(2 t^4+t^2+2 t+1\right)+ \\
& \qquad + d \left(t^4-3 t^3+3 t^2+6 t+1\right)-t^2+2 t+2 \displaybreak[1]\\
Q_{3,5} & = d^3 \left(3 t^5+4 t^4-2 t^3-6 t^2-4 t-1\right)+d^2 \left(8 t^5+2 t^4-11 t^3-5 t^2+5 t+3\right)+\\
& \qquad + d \left(7 t^5-6 t^4-6 t^3+7 t^2+3 t-1\right)+2 t^5-4 t^4+t^3+5 t^2-2 t-2 \displaybreak[1]\\
Q_{2,4,a} & = d^2 \left(t^4-t^3-4 t^2-3 t-1\right)+d \left(2 t^4-6 t^3-7 t^2+t+3\right)+t^4-5 t^3+t^2+2 t-2 \\\displaybreak[1]
Q_{2,4,b} & = d^2 \left(t^4+2 t^3-t^2-2 t-1\right)+d \left(2 t^4-2 t^3-2 t^2+3 t+4\right)+t^4-4 t^3+5 t^2+2 t-4 \\\displaybreak[1]
Q_{4,5} & = d^4 t^5+d^3 \left(3 t^5-3 t^4-3 t^3+7 t^2+5 t+1\right)+d^2 \left(3 t^5-5 t^4-5 t^3+10 t^2+12 t+2\right)+\\
& \qquad + d \left(t^5-t^4-5 t^3+3 t^2+9 t+5\right)+t^4-3 t^3+4 t+1 \displaybreak[1]\\
Q_{6,9} & = d^6 \left(4 t^8+t^7-15 t^6-20 t^5-6 t^4+8 t^3+10 t^2+5 t+1\right)+\\
& \qquad + d^5 \left(2 t^9+12 t^8-19 t^7-54 t^6-17 t^5+21 t^4-11 t^3-43 t^2-30 t-7\right)+\\
& \qquad + d^4 \left(6 t^9-6 t^8-31 t^7+11 t^6-119 t^4-130 t^3-21 t^2+35 t+14\right)+\\
& \qquad + d^3 \left(2 t^9-32 t^8+72 t^7+59 t^6-227 t^5-258 t^4+59 t^3+164 t^2+43 t-3\right)+\\
& \qquad + d^2 \left(-10 t^9+10 t^8+123 t^7-136 t^6-305 t^5+103 t^4+225 t^3+23 t^2-38 t-13\right)+\\
& \qquad + d \left(-12 t^9+56 t^8-9 t^7-149 t^6-16 t^5+175 t^4+46 t^3-89 t^2-17 t+16\right) + \\
& \qquad  -4 t^9+28 t^8-49 t^7-4 t^6+69 t^5-54 t^4-9 t^3+54 t^2-14 t-20 \displaybreak[1]\\
Q_{\omega,9} & = d^9-7 d^8+15 d^7-2 d^6-14 d^5-16 d^4+41 d^3-23 d^2+d+5 \displaybreak[1]\\
Q_{\omega,60} & = d^{60}-42 d^{59}+825 d^{58}-10050 d^{57}+84827 d^{56}-524435 d^{55}+2444075 d^{54}\displaybreak[1]\\ 
&\qquad -8680920 d^{53}+23364055 d^{52}-46267136 d^{51}+62172868 d^{50}\displaybreak[1]\\ 
&\qquad -43026307 d^{49}-10724689 d^{48}+19327948 d^{47}+113757871 d^{46}\displaybreak[1]\\ 
&\qquad -289556454 d^{45}+161677043 d^{44}+403173198 d^{43}-822414523 d^{42}\displaybreak[1]\\ 
&\qquad +340360209 d^{41}+658154819 d^{40}-734499791 d^{39}-499750302 d^{38}\displaybreak[1]\\ 
&\qquad +1417408819 d^{37}-680996389 d^{36}-701113119 d^{35}+1161482902 d^{34}\displaybreak[1]\\ 
&\qquad -934417344 d^{33}+751648667 d^{32}-23523738 d^{31}-1359642298 d^{30}\displaybreak[1]\\ 
&\qquad +1528218917 d^{29}+342409869 d^{28}-1836361788 d^{27}+946900947 d^{26}\displaybreak[1]\\ 
&\qquad +763927401 d^{25}-1172104767 d^{24}+652553812 d^{23}-193252562 d^{22}\displaybreak[1]\\ 
&\qquad -352541742 d^{21}+857069723 d^{20}-561108191 d^{19}-289399926 d^{18}\displaybreak[1]\\ 
&\qquad +602082003 d^{17}-186224613 d^{16}-206339296 d^{15}+185432097 d^{14}\displaybreak[1]\\ 
&\qquad -10906225 d^{13}-54265030 d^{12}+26840191 d^{11}+547786 d^{10}\displaybreak[1]\\ 
&\qquad -5118901 d^9+1967134 d^8-218389 d^7-37050 d^6+47054 d^5\displaybreak[1]\\ 
&\qquad -35063 d^4+10325 d^3-903 d^2-49 d+7 \displaybreak[1]\\
\end{align*}

The other factors, $Q_{7,11}, Q_{8,12}, Q_{11,19}, Q_{21,33}, Q_{22,36}, Q_{51,69}, Q_{54,78}$, and $Q_{36,60}$ are available in \LaTeX{} and {\tt Mathematica} formats 
in the {\tt polynomials/}  subdirectory of the {\tt arXiv} source as files {\tt Q_i,j.tex} and {\tt Q_i,j.m}, and also in the {\tt Mathematica} notebook
{\tt{code/GroebnerBasisCalculations.nb}}

\bibliographystyle{alpha}
\bibliography{../../bibliography/bibliography}

\end{document}